\documentclass[12pt]{article}
\usepackage{amsmath,amsfonts,amssymb,amsthm,mathrsfs}
\usepackage[a4paper,vmargin={3.5cm,3.5cm},hmargin={2.5cm,2.5cm}]{geometry}
\usepackage[margin=1cm]{caption}
\usepackage{graphicx,graphics}
\usepackage{epsfig}
\usepackage{latexsym}
\usepackage[utf8]{inputenc}
\usepackage{ae,aecompl}
\usepackage[english]{babel}
\usepackage[colorlinks=true]{hyperref}
\usepackage{enumerate}
\usepackage{bbm}
\usepackage{pxfonts}
\usepackage{dsfont}
\usepackage{tikz}
\usepackage{tabularx}
\usepackage{multirow}
\usepackage{subcaption}
\usepackage{stackrel}

\usepackage[auto]{contour}
\usetikzlibrary{shapes, angles, decorations.text, patterns, fadings,decorations.pathreplacing,fit,calc}
\pgfdeclarelayer{edgelayer}
\pgfdeclarelayer{nodelayer}
\pgfsetlayers{edgelayer,nodelayer,main}

\date{}

\newtheorem{theorem}{Theorem}[]
\newtheorem*{theorem*}{Theorem}
\newtheorem{proposition}{Proposition}[section]
\newtheorem{lemma}[proposition]{Lemma}
\newtheorem{cor}[proposition]{Corollary}
\theoremstyle{definition}
\newtheorem{df}{Definition}[section]

\newcommand{\R}{\mathcal{R}}
\newcommand{\F}{\mathcal{F}} 
\newcommand{\FF}{\tilde{\mathcal{F}}} 

\newcommand{\dgr}{d_{\mathrm{gr}}}

\renewcommand{\P}{\mathbb{P}} 
\newcommand{\sQ}{\mathsf{Q}} 
\newcommand{\sM}{\mathsf{M}} 
\newcommand{\LT}{\mathsf{LT}} 
\renewcommand{\epsilon}{\varepsilon}

\newcommand{\st}{\mid}

\newcommand{\coll}[1]{\operatorname{coll}(#1)}

\DeclareSymbolFont{extraup}{U}{zavm}{m}{n}
\DeclareMathSymbol{\varheart}{\mathalpha}{extraup}{86}
\DeclareMathSymbol{\vardiamond}{\mathalpha}{extraup}{87}

\makeatletter
\renewcommand*{\@fnsymbol}[1]{\ensuremath{\ifcase#1\or  \vardiamond \or \clubsuit\or \spadesuit\or
   \mathsection\or \mathparagraph\or \|\or **\or \dagger\dagger
   \or \ddagger\ddagger \else\@ctrerr\fi}}
\makeatother

\title{\bf \textsc{A polynomial upper bound for the mixing time of edge rotations on planar maps}}
\author{Alessandra Caraceni\thanks{Department of Statistics, University of Oxford, UK.\hfill  \texttt{Alessandra.Caraceni@stats.ox.ac.uk}.\newline The author would like to acknowledge that part of this work was done while they were supported by the EPSRC grant EP/N004566/1 ``Mathematical Analysis of Strongly Correlated Processes on Discrete Dynamic Structures''.}}

\begin{document}
\maketitle
\tikzstyle{vertex}=[circle, fill=black, inner sep=3pt, draw=white, ultra thick]
\tikzstyle{tree}=[ultra thick]
\tikzstyle{P}=[ultra thick, blue!60!white]
\tikzstyle{P'}=[ultra thick, green!70!black]
\tikzstyle{map}=[thick, black]
\tikzstyle{quad}=[thick, red]
\tikzstyle{b}=[line width=4pt, white]

\abstract{We consider a natural local dynamic on the set of all rooted planar maps with $n$ edges that is in some sense analogous to ``edge flip'' Markov chains, which have been considered before on a variety of combinatorial structures (triangulations of the $n$-gon and quadrangulations of the sphere, among others). We provide the first polynomial upper bound for the mixing time of this ``edge rotation'' chain on planar maps: we show that the spectral gap of the edge rotation chain is bounded below by an appropriate constant times $n^{-11/2}$. In doing so, we provide a partially new proof of the fact that the same bound applies to the spectral gap of edge flips on quadrangulations as defined in \cite{CS2019}, which makes it possible to generalise the result of \cite{CS2019} to a variant of the edge flip chain related to edge rotations via Tutte's bijection.}

\section{Introduction}

This work is concerned with estimating the mixing time of a particular Markov chain on the set of all possible (rooted) planar maps with $n$ edges.

Many different Markov chains with a geometric flavour have been considered on a variety of interesting state spaces given by the sets of all possible planar combinatorial structures of a certain type and size -- e.g.~plane trees, binary trees, triangulations of the $n$-gon, lattice triangulations, quadrangulations of the sphere, etc.

A natural family of Markov chains which have sparked a lot of interest, both because of their deceptive simplicity and their potential applications (e.g.~to systematic biology \cite{Ald00}), is that of ``edge flip'' chains. The archetypal example of an edge flip chain is Aldous' so-called \emph{triangulation walk} \cite{Ald94b}, defined on the state space of all possible triangulations of the $n$-gon (i.e.~of maximal configurations of non-crossing diagonals). Its transitions are \emph{edge flips} in the following sense: given a triangulation of the $n$-gon, a single step of the chain consists in choosing a diagonal uniformly at random and, with probability $1/2$, replacing it with the \emph{other} diagonal of the unique quadrilateral formed by the two triangles adjacent to it (see Figure~\ref{fig:triangulation walk}).

Giving a sharp estimate for the mixing time of the triangulation walk as a function of $n$ is a notoriously difficult open problem. The lower bound of $\Omega(n^{3/2})$ shown by Molloy, Reed and Steiger~\cite{MRS99}, which is in fact Aldous' original conjecture for the actual growth rate of the relaxation time, is still quite distant from the best upper bound to date, which is the $O(n^5\log n)$ obtained by McShine and Tetali~\cite{Catalan97}.

But triangulations of the $n$-gon are not the only structures that are well-suited to supporting an edge flip chain, though they provide perhaps the simplest possible example; edge flip dynamics have been considered for example on lattice triangulations~\cite{CMSS15,CMSS16,S17} and rectangular dissections~\cite{CMRdyadic, CLSdyadic}. Recently, Alexandre Stauffer and the author proved a polynomial upper bound for the mixing time of edge flips on quadrangulations of the sphere~\cite{CS2019}.

Some very natural classes of combinatorial objects able to support edge flip chains are specific sets of so-called \emph{planar maps}, where by planar map we mean a connected, locally finite planar (multi)graph endowed with a cellular embedding in the two-dimensional sphere, considered up to orientation-preserving homeomorphisms of the sphere itself. For example, one might consider triangulations of the sphere with $n$ edges -- that is, planar maps whose faces have degree 3 -- rather than triangulations of the $n$-gon. An edge flip would then consist in choosing an edge uniformly at random and, with probability 1/2, replacing it with the other diagonal of the quadrilateral formed by the two faces adjacent to it -- or, if the edge is adjacent to only one face, leaving it unchanged (see Figure~\ref{fig:flips on triangulations of the sphere}). This chain has been considered by Budzinski in \cite{B17}, where he shows a lower bound of $\Omega(n^{5/4})$ for the mixing time.

Analogous chains can be defined on the set of $p$-angulations of the sphere with $n$ edges for any $p>3$: one chooses an edge uniformly at random and, if it is adjacent to two distinct faces, erases it to obtain a ($2p-2$)-angular face $f$, and then draws an edge joining the $i$-th corner of $f$, where $i$ is chosen uniformly at random in $\{0,1,\ldots,2p-3\}$ (and corners are labelled, say, clockwise), to corner $i+p-1\pmod{2p-2}$, so as to recreate two $p$-angular faces within $f$. Some care must be taken (and some non-canonical choices made) in dealing with edges that are adjacent to a single face on both sides.

An especially attractive case is $p=4$, namely, that of quadrangulations (see Figure~\ref{fig:flips on quadrangulations}). In this case, an edge separating two faces, if flipped, will be replaced by one of three edges cutting the hexagon created in its absence ``in half'', chosen uniformly at random. In particular, it remains unchanged with probability $1/3$. It is therefore natural to define a flip for a quadrangulation edge adjacent to the same face on both sides as leaving the edge unchanged with probability $1/3$ and, with probability $2/3$, replacing it with an edge joining its degree $1$ endpoint to the unique vertex of the face which was not an endpoint of the original edge (see Figure~\ref{fig:flips}, and more generally Section~\ref{sec:from rotations to flips} for a detailed description of the dynamics). 

The case of quadrangulations of the sphere is interesting for multiple reasons. One is that it is still very simple and preserves a strong similarity to the case of edge flips on triangulations of the sphere and of the $n$-gon. Another is the fact that quadrangulations in particular come with a very handy toolset, including Schaeffer-type bijections with labelled plane trees \cite{Sch98}: they fall within the scope of so-called Catalan structures, that is, combinatorial structures whose enumeration is closely related to Catalan numbers (e.g.~plane trees, triangulations of the $n$-gon, binary trees etc.); as a consequence, opportunities arise for a number of possible Markov chain comparisons.

One such comparison, made with a ``leaf translation'' Markov chain on labelled plane trees, is what made it possible to show the main result of \cite{CS2019}, namely an upper bound of order $n^{11/2}$ for the relaxation time of the edge flip Markov chain on quadrangulations of the sphere.

It should now be mentioned that, in order to have the Schaeffer bijection with labelled plane trees and to have Catalan numbers emerge when enumerating quadrangulations, one considers \emph{pointed}, \emph{rooted} quadrangulations of the sphere -- that is, quadrangulations endowed with a distinguished vertex and a distinguished oriented edge. Redefining the dynamics to take the pointing and rooting into account poses no difficulties; the choice made in \cite{CS2019} is that of performing edge flips exactly as described, preserving the pointing and the orientation of the root edge when flipped (Figure~\ref{fig:flips}). It seems quite reasonable that pointing and rooting should not be truly relevant, and indeed the pointing can be quickly dealt with and does not appear in the results of \cite{CS2019}. As for the rooting, however, it is worth noting that its role is more central. While for example it is natural to conjecture that the upper bound of $O(n^{11/2})$ for the relaxation time proved in \cite{CS2019} should also hold for the mixing time of -- say -- a Markov chain that censors flips of the root edge, or that excludes the root edge from the set of ``flippable'' edges, this fact is not easy to show; moreover, the proof in the aforementioned paper relies heavily on some ad hoc geometric constructions that build upon the Schaeffer bijection, and root edge flips feature prominently in its canonical paths, so that adapting the proof is utterly non-trivial.

On the other hand, the argument in \cite{CS2019} does have the potential for generalisation, and one may very well wish to apply variants of it to other edge flip Markov chains and to other classes of planar maps.

We have mentioned how one could consider edge flips on $p$-angulations for $p\neq 4$; one other avenue for generalisation would be to consider, rather than edge flips on $p$-angulations, dynamics on the set of all planar maps with -- say -- a fixed number of edges, with no restrictions on face degrees. This is exactly what we propose to do in this paper. We shall  consider a natural dynamic on planar maps that, in analogy to edge flips, involves the local manipulation of a single random edge at each step. What we will introduce is a Markov chain which we will call the \emph{edge rotation} chain on (rooted) planar maps with $n$ edges. A single step consists essentially in choosing an oriented edge uniformly at random and sliding its ``tip'' one step to the left, or one step to the right, or leaving everything unchanged (each with probability 1/3), see Figure~\ref{fig:edge rotation chain, informal}. This description, though it should give the right general idea, needs to be formalised and amended to take into account some rather degenerate cases (e.g.~loops and degree 0 vertices); carefully reading Section~\ref{sec:edge rotations} should make it apparent how the more complex presentation given there is truly the one sensible formalisation of the edge rotation chain.

\begin{figure}
\centering
\begin{subfigure}[b]{0.3\textwidth}\centering
\begin{tikzpicture}
\node[draw=none,minimum size=4cm,regular polygon,regular polygon sides=8] (a) {};
\foreach \x in {1,2,...,8}
\node[vertex] (\x) at (a.corner \x) {};
\draw[very thick] (1)--(2)--(3)--(4)--(5) (6)--(7)--(8)--(1);
\draw[very thick, ->] (5)--(6);
\draw (1)--(3)--(5) (1)--(6)--(8);
\draw[red, very thick] (5)--(1);
\draw[blue, very thick, dashed] (3)--(6);
\end{tikzpicture}
\caption{\label{fig:triangulation walk}}
\end{subfigure}
\begin{subfigure}[b]{0.3\textwidth}\centering
\begin{tikzpicture}[scale=.8]
\draw[use as bounding box, draw=none] (-3.5,-1.3) rectangle (3,5.2);
	\begin{pgfonlayer}{nodelayer}
		\node [style=vertex] (0) at (-1, 1) {};
		\node [style=vertex] (1) at (-0.75, 3.5) {};
		\node [style=vertex] (2) at (-0.25, 1) {};
		\node [style=vertex] (3) at (0.5, 2.5) {};
		\node [style=vertex] (4) at (0.75, -0.75) {};
		\node [style=vertex] (5) at (2.5, 2.25) {};
		\node [style=vertex] (6) at (-1.5, -0.75) {};
		\node [style=vertex] (7) at (-2.75, 3.25) {};
		\node [style=vertex] (8) at (0.5, 4) {};
		\node [style=vertex] (9) at (0.25, 0.25) {};
	\end{pgfonlayer}
	\begin{pgfonlayer}{edgelayer}
		\draw (1) to (2);
		\draw (2) to (3);
		\draw (1) to (3);
		\draw [very thick, red, in=150, out=120, looseness=2.75] (2) to (6);
		\draw [in=-135, out=135, looseness=20] (2) to (2);
		\draw (2) to (0);
		\draw (6) to (2);
		\draw (1) to (8);
		\draw (8) to (3);
		\draw (5) to (2);
		\draw (6) to (4);
		\draw [bend right=45, looseness=1.25] (2) to (4);
		\draw [bend left, looseness=1.00] (7) to (2);
		\draw (8) to (5);
		\draw (7) to (1);
		\draw (5) to (4);
		\draw [bend left=45, looseness=1.25] (2) to (4);
		\draw (2) to (9);
		\draw (9) to (4);
		\draw [bend left=45, looseness=1.50] (8) to (2);
		\draw [bend right=45, looseness=1.00] (7) to (6);
		\draw [bend left, looseness=1.00] (7) to (8);
		\draw [very thick, ->, bend right=60, looseness=1.50] (6) to (5);
		\draw [bend left=90, looseness=1.50] (7) to (5);
		\draw [very thick, dashed, blue, in=-135, out=-90, looseness=1.75] (7) to (2);
	\end{pgfonlayer}
\end{tikzpicture}
\caption{\label{fig:flips on triangulations of the sphere}}
\end{subfigure}
\begin{subfigure}[b]{0.3\textwidth}\centering
\begin{tikzpicture}
	\begin{pgfonlayer}{nodelayer}
		\node [style=vertex] (0) at (-2.25, -1) {};
		\node [style=vertex] (1) at (2, -1) {};
		\node [style=vertex] (2) at (2, 3) {};
		\node [style=vertex] (3) at (-2.25, 3) {};
		\node [style=vertex] (4) at (-1.25, -0) {};
		\node [style=vertex] (5) at (-1.25, 2) {};
		\node [style=vertex] (6) at (-0.75, 1) {};
		\node [style=vertex] (7) at (0.5, 0.75) {};
		\node [style=vertex] (8) at (1.25, 2) {};
		\node [style=vertex] (9) at (1, 1.25) {};
		\node [style=vertex] (10) at (1, -0.25) {};
	\end{pgfonlayer}
	\begin{pgfonlayer}{edgelayer}
		\draw[very thick, ->] (0) to (1);
		\draw (1) to (2);
		\draw (2) to (3);
		\draw (3) to (0);
		\draw (0) to (4);
		\draw (3) to (5);
		\draw [very thick, red, bend left=75, looseness=1.75] (5) to (4);
		\draw (4) to (6);
		\draw (5) to (4);
		\draw [bend right=15, looseness=1.00] (0) to (7);
		\draw [bend left, looseness=1.00] (5) to (7);
		\draw [bend left, looseness=0.75] (7) to (8);
		\draw (7) to (9);
		\draw [bend right=60, looseness=1.25] (7) to (8);
		\draw [bend left=45, looseness=1.25] (7) to (2);
		\draw [in=-105, out=-15, looseness=1.25] (7) to (2);
		\draw (7) to (10);
		\draw (10) to (1);
		\draw[very thick, dashed, blue] (6) to (7);
		\draw [very thick, dashed, orange, in=82, out=23, looseness=5.8] (0) to (4);
	\end{pgfonlayer}
\end{tikzpicture}
\caption{\label{fig:flips on quadrangulations}}
\end{subfigure}
\caption{\label{fig:different flips} An edge flip performed on a triangulation of the octagon (a); an edge flip on a rooted triangulation of the sphere -- drawn so that the infinite face lies to the right of the root edge (b); an edge flip on a quadrangulation of the sphere, where both possible new alternative edges are drawn, dashed, in different colours (c).}
\end{figure}
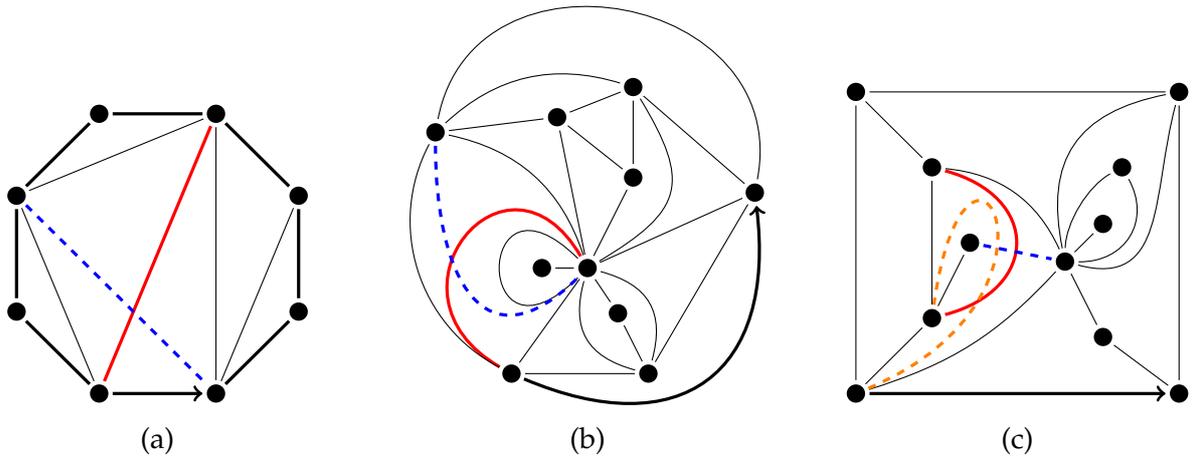

Note that, by considering the edge rotation chain on \emph{rooted} planar maps, we can take advantage of how general rooted planar maps with a fixed number of edges can themselves be thought of as Catalan structures. Indeed, thanks to Tutte's bijection \cite{Tut63} we shall directly relate the edge rotation chain to a version of the edge flip Markov chain on rooted quadrangulations where the quadrangulation root edge is not included in the set of ``flippable'' edges. 

We will then proceed to give an upper bound that will apply to both the mixing time of the edge rotation chain and that of the variant edge flip chain on rooted quadrangulations. Our main result is the following:
\begin{theorem}\label{main theorem}
Let $\nu_n$ and $\mu_n$ be the spectral gaps of the (non-root-flipping) edge flip Markov chain $\FF^n$ on the set $\sQ_n$ of quadrangulations with $n$ faces and of the edge rotation Markov chain $\R_n$ on the set $\sM_n$ of rooted planar maps with $n$ edges, respectively. We have $\nu_n=\mu_n$, and there are positive constants $C_1,C_2$ (independent of $n$) such that
$$C_1n^{-5/4}\geq\nu_n\geq C_2n^{-11/2}$$
for all $n$. Consequently, the mixing time of both chains is $O(n^{13/2})$.
\end{theorem}

The proof will combine part of the approach of \cite{CS2019} with some new ideas, which render it almost completely independent of Schaeffer's bijection: we shall construct probabilistic canonical paths on the set of rooted quadrangulations rather than the set of plane trees, thus making the approach more readily generalisable.

Section~\ref{sec:trivial bijection} introduces relevant objects -- maps, quadrangulations -- and contains a brief description of Tutte's bijection, which will be used in Section~\ref{sec:from rotations to flips} to relate the edge rotation chain presented in Section~\ref{sec:edge rotations} to an edge flip chain.

The rest of the paper will develop the necessary tools to prove Theorem~\ref{main theorem}. The argument is based on an algorithm to grow quadrangulations uniformly at random by ``adding faces'' one at a time (Section~\ref{sec:growth algorithm}) and a construction of probabilistic canonical paths (Section~\ref{sec:canonical paths}) which is truly the core of this paper. Section~\ref{sec:final bound} concludes the proof.

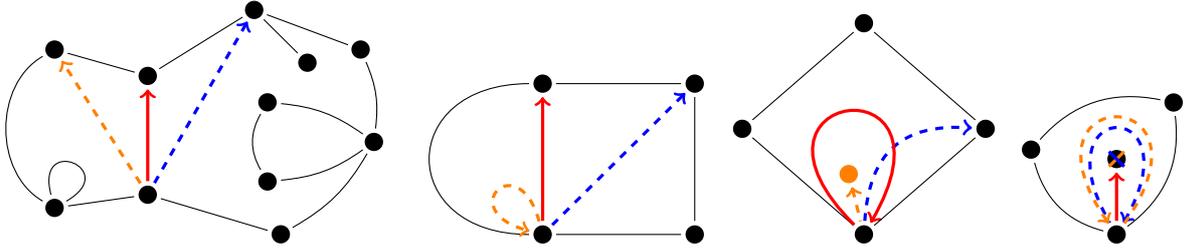
\begin{figure}
\centering
\begin{tikzpicture}[scale=.7]
	\begin{pgfonlayer}{nodelayer}
		\node [style=vertex] (0) at (0, -0.5) {};
		\node [style=vertex] (1) at (0, 1.75) {};
		\node [style=vertex] (2) at (2, 3) {};
		\node [style=vertex] (3) at (4, 2.25) {};
		\node [style=vertex] (4) at (4.25, 0.5) {};
		\node [style=vertex] (5) at (-1.75, -0.75) {};
		\node [style=vertex] (6) at (-1.75, 2.25) {};
		\node [style=vertex] (7) at (2.25, 1.25) {};
		\node [style=vertex] (8) at (2.25, -0.25) {};
		\node [style=vertex] (9) at (2.5, -1.25) {};
		\node [style=vertex] (10) at (3, 2) {};
	\end{pgfonlayer}
	\begin{pgfonlayer}{edgelayer}
		\draw (1) to (2);
		\draw (2) to (3);
		\draw [bend left=15, looseness=1.00] (3) to (4);
		\draw [bend right=15, looseness=1.00] (4) to (7);
		\draw [bend right, looseness=1.00] (7) to (8);
		\draw [bend right=15, looseness=1.00] (8) to (4);
		\draw (9) to (0);
		\draw [bend right=15, looseness=1.00] (9) to (4);
		\draw (0) to (5);
		\draw [bend left=60, looseness=1.00] (5) to (6);
		\draw (6) to (1);
		\draw [in=105, out=30, looseness=10] (5) to (5);
		\draw (2) to (10);
		\draw[very thick, red, ->] (0) to (1);
		\draw[very thick, blue, dashed, ->] (0) to (2);
		\draw[very thick, orange, dashed, ->] (0) to (6);
	\end{pgfonlayer}
\end{tikzpicture}
\begin{tikzpicture}
	\begin{pgfonlayer}{nodelayer}
		\node [style=vertex] (0) at (0, -1) {};
		\node [style=vertex] (1) at (0, 1) {};
		\node [style=vertex] (2) at (2, 1) {};
		\node [style=vertex] (3) at (2, -1) {};
	\end{pgfonlayer}
	\begin{pgfonlayer}{edgelayer}
		\draw [bend left=90, looseness=2.25] (0) to (1);
		\draw (0) to (3);
		\draw (3) to (2);
		\draw (1) to (2);
		\draw [very thick, red, ->] (0) to (1);
		\draw [very thick, dashed, orange, in=165, out=105, looseness=15, ->] (0) to (0);
		\draw [very thick, dashed, blue, ->] (0) to (2);
	\end{pgfonlayer}
\end{tikzpicture}
\begin{tikzpicture}[scale=.8]
	\begin{pgfonlayer}{nodelayer}
		\node [style=vertex] (0) at (0, -2) {};
		\node [style=vertex] (1) at (-2, -0.25) {};
		\node [style=vertex] (2) at (0, 1.5) {};
		\node [style=vertex] (3) at (2, -0.25) {};
		\node [style=vertex, fill=orange] (4) at (-0.25, -1) {};
	\end{pgfonlayer}
	\begin{pgfonlayer}{edgelayer}
		\draw (0) to (1);
		\draw [very thick, dashed, blue, bend left=45, looseness=1.25, ->] (0) to (3);
		\draw (1) to (2);
		\draw (2) to (3);
		\draw [very thick, red, in=135, out=60, looseness=30, <-] (0) to (0);
		\draw [very thick, orange, dashed, <-](4) to (0);
		\draw (0) to (3);
	\end{pgfonlayer}
\end{tikzpicture}
\begin{tikzpicture}[scale=.5]
	\begin{pgfonlayer}{nodelayer}
		\node [style=vertex] (0) at (0, -2) {};
		\node [style=vertex] (1) at (-2.25, 0.25) {};
		\node [style=vertex] (2) at (1.5, 1.5) {};
		\node [style=vertex] (3) at (0, 0) {};
			\draw [very thick, orange] (-0.2,-0.2)--(0.2,0.2);
			\draw [very thick, blue] (0.2,-0.2)--(-0.2,0.2);
	\end{pgfonlayer}
	\begin{pgfonlayer}{edgelayer}
		\draw [bend right, looseness=1.00] (1) to (0);
		\draw [bend right=30] (2) to (1);
		\draw [bend right, looseness=1.00] (0) to (2);
		\draw [very thick, red, ->] (0) to (3);
		\draw [very thick, dashed, blue, out=60, in=120, looseness=28, <-] (0) to (0);
		\draw [very thick, dashed, orange, out=55, in=125, looseness=29, ->] (0) to (0);
	\end{pgfonlayer}
\end{tikzpicture}
\caption{\label{fig:edge rotation chain, informal}An informal look at the edge rotation chain on planar maps with $n$ edges; in each case, an oriented edge is shown, together with the two (or one, in the last picture) faces adjacent to it. One can rotate it clockwise or counterclocwise, which in some cases may create a loop (second picture). If the edge is itself a loop enclosing a face of degree 1 (third picture) one of the two edge rotations causes it to ``detach itself'' from the boundary of its external face and create a new degree 1 vertex. If the edge is oriented towards an endpoint of degree 1 (fourth picture) then rotating it in either direction eliminates that endpoint in favour of a loop. A complete presentation (not in terms of the rotated oriented edge but of the corner that the tip ``rotates through'') is given in Section~\ref{sec:edge rotations}.}
\end{figure}

\section{Quadrangulations, general planar maps and edge flips}\label{sec:trivial bijection}
\begin{df}
A \emph{planar map} is a connected, locally finite planar (multi)graph endowed with a cellular embedding in the two-dimensional sphere, considered up to orientation-preserving homeomorphisms of the sphere itself.	
\end{df}

Of course, planar maps inherit terminology and features from graphs -- we shall speak of their \emph{vertices} and \emph{edges} -- but with their built-in planar embedding comes the added perk of having well defined \emph{faces} (i.e.~the connected components of the complement of the image of vertices and edges via the cellular embedding, see Figure~\ref{planar map}). It will often prove useful to endow an edge with an \emph{orientation} (each edge has two possible orientations). Given an oriented edge $\vec{e}$ in a map $m$ whose endpoints are a vertex $e^-$ (the tail) and a vertex $e^+$, we shall informally say that the \emph{corner} corresponding to $\vec{e}$ is a suitably small neighbourhood of the vertex $e^-$ intersected with the face lying directly to the right of $\vec{e}$.

We will speak of corners as ``belonging to'' faces (the corner corresponding to $\vec{e}\,$ belongs to the face lying directly to the right of $\vec{e}$) and also to vertices (the corner corresponding to $\vec{e}$ is a corner of vertex $e^{-}$). Corners of a single vertex and corners of a single face have two natural cyclic orderings: \emph{clockwise} and \emph{counterclockwise}. Given a face $f$ of a map $m$, we shall call the cyclic sequence $(c_i)_{i=1}^{\deg f}$ of all corners of $f$ in clockwise (resp.~counterclockwise) order, where the index is considered modulo $\deg f$, a \emph{clockwise} (resp.~\emph{counterclockwise}) \emph{contour} of $f$; the number $\deg f$ of corners of $f$ is the \emph{degree} of the face $f$. When mentioning a contour of the face $f$ without specifying its direction, we shall be referring to its \emph{clockwise} contour.

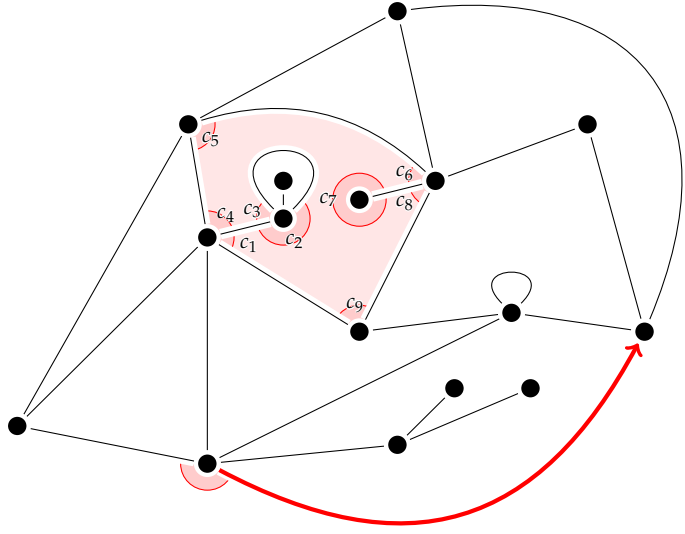
\begin{figure}
\centering
\begin{tikzpicture}
	\begin{pgfonlayer}{nodelayer}
		\node [style=vertex] (0) at (-2.75, 2.5) {};
		\node [style=vertex] (1) at (-0.75, 1.25) {};
		\node [style=vertex] (2) at (0.25, 3.25) {};
		\node [style=vertex] (3) at (-0.25, 5.5) {};
		\node [style=vertex] (4) at (-3, 4) {};
		\node [style=vertex] (5) at (-5.25, -0) {};
		\node [style=vertex] (6) at (-2.75, -0.5) {};
		\node [style=vertex] (7) at (-0.25, -0.25) {};
		\node [style=vertex] (8) at (1.25, 1.5) {};
		\node [style=vertex] (9) at (3, 1.25) {};
		\node [style=vertex] (10) at (2.25, 4) {};
		\node [style=vertex] (11) at (-1.75, 2.75) {};
		\node [style=vertex] (12) at (0.5, 0.5) {};
		\node [style=vertex] (13) at (1.5, 0.5) {};
		\node [style=vertex] (14) at (-1.75, 3.25) {};
		\node [style=vertex] (15) at (-0.75, 3) {};
	\end{pgfonlayer}
	\begin{pgfonlayer}{edgelayer}
	\fill[red!10] (4.center) to (0.center) to (1.center) to (2.center) [bend right, looseness=1.1] to (4.center);
	\begin{scope}
		\clip (4.center) to (0.center) to (1.center) to (2.center) [bend right, looseness=1.1] to (4.center);
		\draw [draw=red, fill=red!20] (0) circle (10pt);
		\draw [draw=red, fill=red!20] (1) circle (10pt);
		\draw [draw=red, fill=red!20] (2) circle (10pt);
		\draw [draw=red, fill=red!20] (4) circle (10pt);
		\draw [draw=red, fill=red!20] (11) circle (10pt);
		\draw [draw=red, fill=red!20] (15) circle (10pt);
		\node () at (-2.2, 2.4) {\scriptsize\contour{white}{$c_1$}};
		\node () at (-1.6, 2.45) {\scriptsize\contour{white}{$c_2$}};
		\node () at (-2.15, 2.85) {\scriptsize\contour{white}{$c_3$}};
		\node () at (-2.5, 2.8) {\scriptsize\contour{white}{$c_4$}};
		\node () at (-2.7, 3.8) {\scriptsize\contour{white}{$c_5$}};
		\node () at (-0.15, 3.35) {\scriptsize\contour{white}{$c_6$}};
		\node () at (-0.15, 2.95) {\scriptsize\contour{white}{$c_8$}};
		\node () at (-1.15, 3) {\scriptsize\contour{white}{$c_7$}};
		\node () at (-0.8, 1.6) {\scriptsize\contour{white}{$c_9$}};
	\end{scope}
	\begin{scope}
	\clip [bend right=45, looseness=1.25] (6) to (9) [bend left=40] to (5.center)--cycle;
\draw [draw=red, fill=red!20] (6) circle (10pt);
\end{scope}
	
		\draw [b] (11) to (0);
		\draw [b] (0) to (1);
		\draw [b] (1) to (8);
		\draw [b, in=135, out=45, loop] (8) to ();
		\draw [b] (8) to (9);
		\draw [b] (1) to (2);
		\draw [b] (2) to (10);
		\draw [b] (10) to (9);
		\draw [b] (2) to (3);
		\draw [b] (3) to (4);
		\draw [b] (4) to (0);
		\draw [b] (0) to (5);
		\draw [b] (5) to (6);
		\draw [b] (6) to (0);
		\draw [b] (6) to (8);
		\draw [b] (6) to (7);
		\draw [b, bend right=45, looseness=1.25] (6) to (9);
		\draw [b] (7) to (12);
		\draw [b] (7) to (13);
		\draw [b, in=135, out=45, looseness=15] (11) to (11);
		\draw [b] (5) to (4);
		\draw [b] (14) to (11);
		\draw [b] (15) to (2);
		\draw [b, bend left, looseness=1.00] (4) to (2);
		\draw [b,bend left=60, looseness=1.50] (3) to (9);

		\draw (11) to (0);
		\draw (0) to (1);
		\draw (1) to (8);
		\draw [in=135, out=45, loop] (8) to ();
		\draw (8) to (9);
		\draw (1) to (2);
		\draw (2) to (10);
		\draw (10) to (9);
		\draw (2) to (3);
		\draw (3) to (4);
		\draw (4) to (0);
		\draw (0) to (5);
		\draw (5) to (6);
		\draw (6) to (0);
		\draw (6) to (8);
		\draw (6) to (7);
		\draw [red, ->, ultra thick, bend right=45, looseness=1.25] (6) to (9);
		\draw (7) to (12);
		\draw (7) to (13);
		\fill [white, in=135, out=45, looseness=15] (11) to (11);
		\draw [in=135, out=45, looseness=15] (11) to (11);
		\draw (5) to (4);
		\draw (14) to (11);
		\draw (15) to (2);
		\draw [bend left, looseness=1.00] (4) to (2);
		\draw [bend left=60, looseness=1.50] (3) to (9);
	\end{pgfonlayer}
\end{tikzpicture}
\caption{A rooted planar map drawn in the plane in such a way that the infinite face contains the root corner. A face of degree 9 is shaded and the labels $c_1,\ldots,c_9$ placed along its clockwise contour.\label{planar map}}
\end{figure}

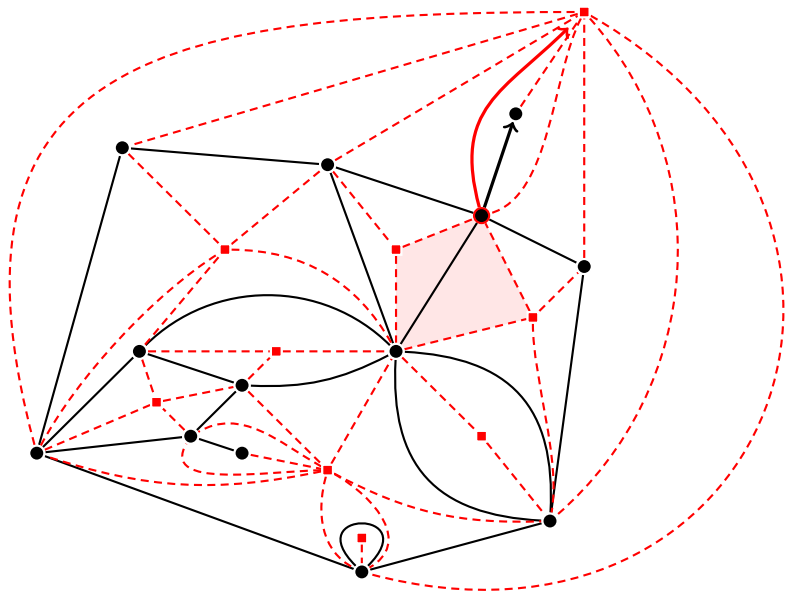
\begin{figure}\center
\begin{tikzpicture}[scale=.9, real/.style={circle, fill=black, draw=white, thick, inner sep=2pt}, new/.style={fill=red, draw=white, thick, inner sep=2pt}, bla/.style={red, thick, densely dashed}, every path/.style={thick}]
	\begin{pgfonlayer}{nodelayer}
		\node [style=real] (0) at (-4.5, 2.5) {};
		\node [style=real] (1) at (-5.75, -2) {};
		\node [style=real] (2) at (-1, -3.75) {};
		\node [style=real] (3) at (2.25, 0.75) {};
		\node [style=real] (4) at (1.25, 3) {};
		\node [style=real] (5) at (-4.25, -0.5) {};
		\node [style=real] (6) at (-3.5, -1.75) {};
		\node [style=real] (7) at (-2.75, -1) {};
		\node [style=real] (8) at (-1.5, 2.25) {};
		\node [style=real] (9) at (-0.5, -0.5) {};
		\node [style=real, thick, draw=red] (10) at (0.75, 1.5) {};
		\node [style=real] (11) at (1.75, -3) {};
		\node [style=real] (12) at (-2.75, -2) {};
		\node [style=new] (13) at (-3, 1) {};
		\node [style=new] (14) at (-0.5, 1) {};
		\node [style=new] (15) at (1.5, 0) {};
		\node [style=new] (16) at (0.75, -1.75) {};
		\node [style=new] (17) at (-2.25, -0.5) {};
		\node [style=new] (18) at (-4, -1.25) {};
		\node [style=new] (19) at (-1.5, -2.25) {};
		\node [style=new] (20) at (-1, -3.25) {};
		\node [style=new] (21) at (2.25, 4.5) {};
	\end{pgfonlayer}
	\begin{pgfonlayer}{edgelayer}
	\fill[red!10] (9.center)--(15.center)--(10.center)--(14.center)--(9.center);
		\draw (1) to (0);
		\draw (1) to (6);
		\draw (6) to (7);
		\draw (7) to (5);
		\draw (1) to (5);
		\draw (1) to (2);
		\draw (2) to (11);
		\draw (11) to (3);
		\draw (3) to (10);
		\draw (9) to (10);
		\draw [bend right=45, looseness=1.25] (9) to (11);
		\draw [bend left=15] (9) to (7);
		\draw (8) to (9);
		\draw (10) to (8);
		\draw[<-, very thick] (4) to (10);
		\draw (0) to (8);
		\draw [in=135, out=45, looseness=17] (2) to (2);
		\draw [bend left=45] (5) to (9);
		\draw [bend left=45, looseness=1.25] (9) to (11);
		\draw (12) to (6);
		\draw [style=bla] (20) to (2);
		\draw [style=bla] (5) to (17);
		\draw [style=bla] (7) to (17);
		\draw [style=bla] (5) to (13);
		\draw [style=bla] (17) to (9);
		\draw [style=bla] (5) to (18);
		\draw [style=bla] (18) to (7);
		\draw [style=bla] (6) to (18);
		\draw [style=bla] (1) to (18);
		\draw [style=bla] (14) to (9);
		\draw [style=bla] (8) to (14);
		\draw [style=bla] (14) to (10);
		\draw [style=bla] (15) to (10);
		\draw [style=bla] (9) to (15);
		\draw [style=bla, in=80, out=-90, looseness=0.75] (15) to (11);
		\draw [style=bla] (15) to (3);
		\draw [style=bla] (13) to (8);
		\draw [style=bla, bend left=15, looseness=0.75] (1) to (13);
		\draw [style=bla] (0) to (13);
		\draw [style=bla, bend left] (13) to (9);
		\draw [style=bla] (7) to (19);
		\draw [style=bla] (19) to (9);
		\draw [style=bla, bend right=15] (19) to (11);
		\draw [style=bla, in=30, out=-30, looseness=1.25] (19) to (2);
		\draw [style=bla, in=153, out=-105] (19) to (2);
		\draw [style=bla] (12) to (19);
		\draw [style=bla, bend left=15] (19) to (1);
		\draw [style=bla, in=-120, out=-179] (19) to (6);
		\draw [style=bla, in=30, out=150] (19) to (6);
		\draw [style=bla] (16) to (11);
		\draw [style=bla] (9) to (16);
		\draw [style=bla] (4) to (21);
		\draw [style=bla, in=15, out=-117, looseness=0.75] (21) to (10);
		\draw [style=bla] (3) to (21);
		\draw [style=bla] (8) to (21);
		\draw [style=bla] (21) to (0);
		\draw [style=bla, bend left=45] (21) to (11);
		\draw [style=bla, in=180, out=105, looseness=1.50] (1) to (21);
		\draw [style=bla, in=-15, out=-30, looseness=1.75] (21) to (2);
		\draw [red, in=-135, out=105, looseness=1.25, ->, very thick, shorten >=5pt] (10) to (21);
	\end{pgfonlayer}
\end{tikzpicture}
\caption{\label{fig:trivial bijection}Given a rooted planar map (black vertices and edges in the picture), the corresponding quadrangulation (red and black vertices, red dashed edges) has one face for each edge of the original map (look for example to the shaded red face, which encloses the single black edge joining its two black vertices). The map is recovered from the quadrangulation by drawing the edge in each face, which joins the two corners adjacent to black vertices.}
\end{figure}

A \emph{rooted} planar map is a pair $(m,c)$, where $m$ is a planar map and $c$ is a corner of $m$; since -- as explained above -- there is a direct correspondence between corners and oriented edges, we may also see a rooted planar map as being endowed with a distinguished oriented edge rather than a distinguished corner: we will adopt either point of view, depending of what is most convenient.

We shall call the vertex that the root corner $c$ belongs to, i.e.~the tail of the root edge, the \emph{origin} of the rooted map $(m,c)$; we shall often denote the origin of a map by $\emptyset$. Note that all maps we will refer to in this paper will be rooted; we will therefore, for the sake of simplicity, usually denote them by a single letter and not as a pair: we will write $\sM_n$ for the set of all rooted planar maps with $n$ edges and will write $m\in\sM_n$ to indicate that $m$ is a planar map with $n$ edges and is also endowed with a root corner/edge, which will normally be denoted by $\rho$.

\begin{df}
A quadrangulation is a planar map all of whose faces have degree 4. We shall write $\sQ_n$ for the set of all rooted quadrangulations with $n$ faces.	
\end{df}

It is a classical result of Tutte \cite{Tut63} that we have $|\sQ_n|=|\sM_n|$; and in fact, Tutte himself provides a simple explicit bijection $\Phi:\sM_n\to\sQ_n$, which we shall briefly describe here before making use of it for our purposes.

Given a rooted planar map $m\in\sM_n$, build a new rooted planar map as follows:
\begin{itemize}
\item draw one vertex within each face of $m$;
\item connect each newly drawn vertex to all corners in the face it belongs to (draw new edges in such a way that they do not cross);
\item erase all original edges of $m$;
\item there is one edge drawn by this procedure that crosses the original root corner of $m$; let that edge be the new root edge, oriented away from the original root corner.
\end{itemize}
The procedure described above yields a rooted planar map $\Phi(m)$ which has $|V(m)|+|F(m)|=n+2$ vertices and $2|E(m)|=2n$ edges, hence $n$ faces, each of which can be shown to be a quadrangle; in other words, $\Phi(m)\in\sQ_n$.

An inverse procedure can be described just as easily: given a quadrangulation $q\in\sQ_n$,
\begin{itemize}
\item partition the set of vertices of $q$ into two parts: we shall call \emph{real vertices} those at even graph distance from the origin and \emph{face vertices} those at odd distance (notice that real vertices are only adjacent to face vertices and vice-versa, so each face has two corners of real vertices and two corners of face vertices);
\item within each face, draw an edge joining its two corners belonging to real vertices;
\item erase all face vertices and all original edges of $q$;
\item root the newly formed map in the one corner that the original root edge of $q$ was issued from.
\end{itemize}
The map resulting from this procedure, which clearly has $n$ edges, one for each face of $q$, is none other than $\Phi^{-1}(q)$. Indeed, we have the following:

\begin{theorem}[Tutte] The mapping $\Phi$ is a bijection between the set $\sM_n$ of rooted planar maps with $n$ edges and the set $\sQ_n$ of rooted quadrangulations with $n$ faces; it induces a correspondence between the set of edges of each map $m$ and the set of faces of $\Phi(m)$.\end{theorem}


\section{The edge rotation Markov chain on $\sM_n$}\label{sec:edge rotations}

\begin{figure}
\centering
\begin{tikzpicture}
	\begin{pgfonlayer}{nodelayer}
		\node [style=vertex] (0) at (-5, -0) {};
		\node [style=vertex] (1) at (-2.5, -0) {};
		\node [style=vertex] (2) at (-2.5, 2.5) {};
		\node [style=vertex] (3) at (-5, 2.5) {};
		\node [style=vertex] (4) at (-0.25, -0) {};
		\node [style=vertex] (5) at (-0.25, 2.5) {};
		\node [style=vertex] (6) at (2.75, -0) {};
		\node [style=vertex] (7) at (4.75, 2.5) {};
		\node [style=vertex] (8) at (6, -0) {};
		\node [style=vertex] (9) at (6, 1.25) {};
		\node [style=vertex] (10) at (7.25, 2.5) {};
		\node [style=vertex] (11) at (2.75, 1.25) {};
	\end{pgfonlayer}
	\begin{pgfonlayer}{edgelayer}
	\begin{scope}
	\clip (0.center) to (1.center) to (2.center) to (3.center) to (0.center);
	\filldraw[draw=red, fill=red!20] (0) circle (12pt);	
	\filldraw[draw=red, fill=red!20] (1) circle (12pt);
	\filldraw[draw=red, fill=red!20] (3) circle (12pt);
	\end{scope}
	
	\begin{scope}
	\clip [bend left=45, looseness=1.00](4.center) to (5.center) to (4.center);
	\filldraw[draw=red, fill=red!20] (4) circle (12pt);	
	\filldraw[draw=red, fill=red!20] (5) circle (12pt);
	\end{scope}
	
	\begin{scope}
	\clip [in=135, out=45, looseness=30] (6) to (6);
	\filldraw[draw=red, fill=red!20] (6) circle (12pt);
	\end{scope}
	
	\begin{scope}
	\clip [bend left=15, looseness=1.00] (8.center) to (7.center) [bend left=0] to (10.center) [bend left=15, looseness=1.00] to (8.center);
	\filldraw[draw=red, fill=red!20] (8) circle (12pt);
	\filldraw[draw=red, fill=red!20] (9) circle (9pt);
	\end{scope}
	
	\node (x) at (-4.7, 0.3) {\contour{white}{$c$}};
	\node (x) at (-4.7, 1.75) {\contour{white}{$c_+$}};
	\node (x) at (-3.1, 0.25) {\contour{white}{$c_-$}};
	\node (x) at (-4, 1) {\contour{white}{$e_=$}};
	\node (x) at (-4, -0.25) {\contour{white}{$e_-$}};
	\node (x) at (-5.25, 1) {\contour{white}{$e_+$}};
	
	\node (x) at (-0.25, 0.45) {\contour{white}{$c$}};
	\node (x) at (-0.25, 1.35) {\contour{white}{$e_=$}};
	\node (x) at (-1.1, 1.25) {\contour{white}{$e_+$}};
	\node (x) at (0.7, 1.25) {\contour{white}{$e_-$}};

		\draw[very thick, blue] (0) to (1);
		\draw (1) to (2);
		\draw (2) to (3);
		\draw[very thick, red] (3) to (0);
		\draw [very thick, red, bend left=45, looseness=1.00] (4) to (5);
		\draw [very thick, blue, bend right=45, looseness=1.00] (4) to (5);
		\draw [very thick, violet, in=135, out=45, looseness=30] (6) to (6);
		\draw[very thick, violet] (9) to (8);
		\draw [bend left=15, looseness=1.00] (8) to (7);
		\draw (7) to (10);
		\draw [bend left=15, looseness=1.00] (10) to (8);
		\draw [very thick, dashed] (3) to (1);
		\draw [very thick, dashed, out=-50, in=-130, looseness=15] (5) to (5);
		\draw [very thick, dashed] (11) to (6);
		\draw [very thick, dashed, out=130, in=50, looseness=30] (8) to (8);
		
		\node (x) at (2.8, 0.45) {\contour{white}{$c=c_+=c_-$}};
		\node (x) at (2.8, 1.85) {\contour{white}{$e_+=e_-$}};
		\node (x) at (3.2, 1.15) {\contour{white}{$e_=$}};
		\node (x) at (6, 2) {\contour{white}{$e_=$}};
		\node (x) at (6.3, 1.4) {\contour{white}{$c$}};
		\node (x) at (6, 0.7) {\contour{white}{$e_-=e_+$}};
		\node (x) at (6.35, 0.3) {\contour{white}{$c_-$}};
		\node (x) at (5.8, 0.3) {\contour{white}{$c_+$}};
	\end{pgfonlayer}
\end{tikzpicture}	
\caption{The corners $c,c_-,c_+$ within the face $f_c$; the edge $e_=$ drawn by the procedure is dashed. Notice that $e_=$ is a loop if $e_-$ and $e_+$ are corners of the same vertex, which happens when $\deg f_c=2$ and when the degree of the vertex of $c$ is 1 (second and fourth image).\label{draw triangle}}
\end{figure}
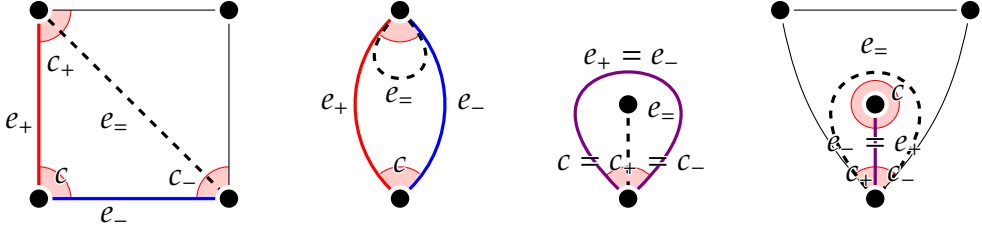

Let $m$ be a map in $\sM_n$, let $c$ be a corner of $m$ other than the root corner and let $s$ be an element of $\{=,+,-\}$. Construct a map $m^{c,s}\in\sM_n$ as follows:
\begin{itemize}
\item the corner $c$ belongs to a face $f_c$ of $m$; let $c_-$ and $c_+$ be the corners immediately before and immediately after $c$ in a clockwise contour of $f_c$;
\item let $e_-$ and $e_+$ be the edges of $f_c$ joining $c_-$ to $c$ and $c$ to $c_+$ respectively; in the case where $\deg f_c=2$, in which $c_-=c_+$, $e_-$ and $e_+$ are the two edges forming the boundary of $f_c$, and they are named in such a way that $f_c$ lies to the right of the edge $e_+$, oriented away from $c$ (Figure~\ref{draw triangle}); in the case where $\deg f_c=1$, in which $c=c_-=c_+$, we set $e_-=e_+$ to be the one loop which constitutes the boundary of $f_c$;
\item if $\deg f_c>1$, draw an edge $e_=$ joining corner $c_-$ to corner $c_+$ (in the case where $c_-=c_+$ and the case where the vertex of $c$ has degree 1, notice that $e_=$ will be a loop);
\item if $c_-=c=c_+$ (that is if $\deg f_c = 1$) draw a new vertex within the loop $e_-$ and join it to the vertex of $c$ by a new edge $e_=$;
\item notice now that, whatever the case for $\deg f_c$, one new edge $e_=$ has been drawn and a triangular face containing $c$, whose boundary edges are $\{e_=,e_-,e_+\}$ (which are not necessarily all distinct), has been created;
\item finally, erase the edge $e_s$ (and any vertices adjacent only to $e_s$); if $s=\pm$ and the root corner $\rho$ of $m$ is in $\{c_-,c_+\}$, set the new root corner to be the part of $\rho$ that did not belong to the triangular face containing $c$ created by drawing $e_=$; otherwise, set the new root corner to be the one that contains the original corner $\rho$ (which is ``larger'' than the original only if $e_s$ is an edge adjacent to $\rho$). The new rooted map obtained in this way is $m^{c,s}$; it has exactly as many edges as $m$ and therefore belongs to $\sM_n$.
\end{itemize}

We shall say that $m^{c,s}$ is obtained from $m$ via an \emph{edge rotation}; though rotating edges is not explicitly mentioned in the construction above, the reason for the name should be clear: except for some rather degenerate cases, the whole construction -- when $s\in\{-,+\}$ -- essentially consists of ``rotating'' the edge $e_s$ within the corner $c_s$ by ``detaching it'' from the vertex of $c$ and instead setting its other endpoint to be within the next corner in the clockwise/counterclockwise contour of $f_c$, thus effectively turning it into the new edge $e_=$ (see Figure~\ref{edge rotation}). Even the case where $e_-=e_+$, which turns an edge with an endpoint of degree 1 into a loop and vice-versa, can be thought of as an edge rotation of sorts.

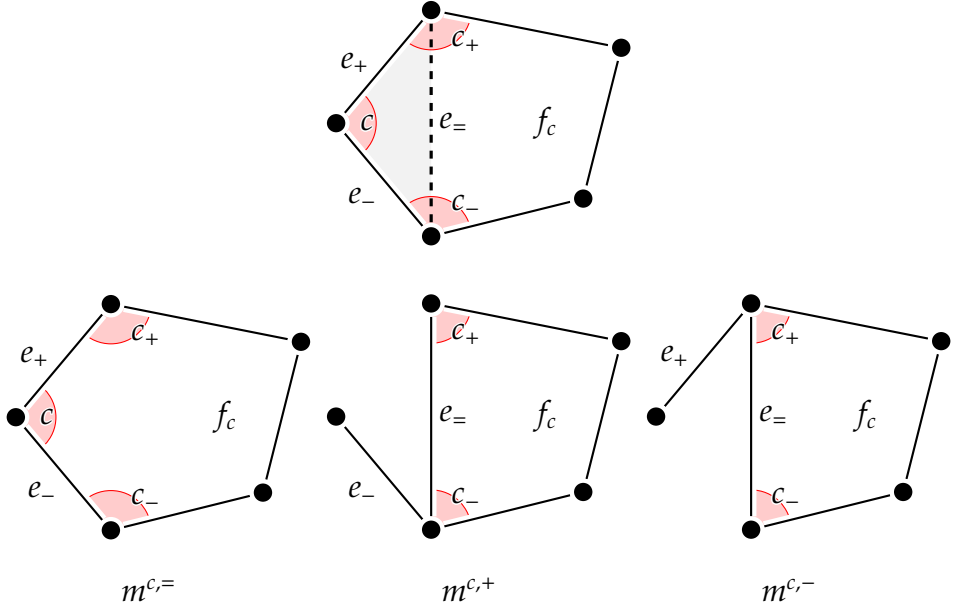
\begin{figure}\centering
\begin{tikzpicture}
	\begin{pgfonlayer}{nodelayer}
		\node [style=vertex, label=below right:\contour{white}{$c_+$}] (0) at (-2, 3) {};
		\node [style=vertex, label=right:\contour{white}{$c$}] (1) at (-3.25, 1.5) {};
		\node [style=vertex, label=above right:\contour{white}{$c_-$}] (2) at (-2, -0) {};
		\node [style=vertex] (3) at (0.5, 2.5) {};
		\node [style=vertex] (4) at (0, 0.5) {};
	\end{pgfonlayer}
	\begin{pgfonlayer}{edgelayer}
	
	\fill[black!5] (1.center) to (2.center) to (0.center) to (1.center);
	
	\begin{scope}
	\clip (1.center) to (2.center) to (4.center) to (3.center) to (0.center) to (1.center);
	\draw[draw=red, fill=red!20] (0) circle (15pt);
	\draw[draw=red, fill=red!20] (1) circle (15pt);
	\draw[draw=red, fill=red!20] (2) circle (15pt);	
	\end{scope}
	
		\draw[b] (1) to (0);
		\draw[b] (0) to (3);
		\draw[b] (1) to (2);
		\draw[b] (2) to (4);
		\draw[b] (4) to (3);
		
		\draw[map] (1) to (0);
		\draw[map] (0) to (3);
		\draw[map] (1) to (2);
		\draw[map] (2) to (4);
		\draw[map] (4) to (3);
		
		\draw[dashed, very thick] (0) to (2);
		
		\node (x) at (-0.5, 1.5) {\contour{white}{$f_c$}};
		\node (e=) at (-1.7, 1.5) {\contour{white}{$e_=$}};
		\node (x) at (-3, 2.3) {\contour{white}{$e_+$}};
		\node (x) at (-2.9, 0.5) {\contour{white}{$e_-$}};
	\end{pgfonlayer}
\end{tikzpicture}\\[.5cm]
\begin{tikzpicture}
	\begin{pgfonlayer}{nodelayer}
		\node [style=vertex, label=below right:\contour{white}{$c_+$}] (0) at (-2, 3) {};
		\node [style=vertex, label=right:\contour{white}{$c$}] (1) at (-3.25, 1.5) {};
		\node [style=vertex, label=above right:\contour{white}{$c_-$}] (2) at (-2, -0) {};
		\node [style=vertex] (3) at (0.5, 2.5) {};
		\node [style=vertex] (4) at (0, 0.5) {};
	\end{pgfonlayer}
	\begin{pgfonlayer}{edgelayer}
	
	\begin{scope}
	\clip (1.center) to (2.center) to (4.center) to (3.center) to (0.center) to (1.center);
	\draw[draw=red, fill=red!20] (0) circle (15pt);
	\draw[draw=red, fill=red!20] (1) circle (15pt);
	\draw[draw=red, fill=red!20] (2) circle (15pt);	
	\end{scope}
	
		\draw[b] (1) to (0);
		\draw[b] (0) to (3);
		\draw[b] (1) to (2);
		\draw[b] (2) to (4);
		\draw[b] (4) to (3);
		
		\draw[map] (1) to (0);
		\draw[map] (0) to (3);
		\draw[map] (1) to (2);
		\draw[map] (2) to (4);
		\draw[map] (4) to (3);
		
		\node (x) at (-0.5, 1.5) {\contour{white}{$f_c$}};
		\node (x) at (-3, 2.3) {\contour{white}{$e_+$}};
		\node (x) at (-2.9, 0.5) {\contour{white}{$e_-$}};
		\node (x) at (-1.5, -0.8) {\contour{white}{$m^{c,=}$}};
	\end{pgfonlayer}
\end{tikzpicture}
\begin{tikzpicture}
	\begin{pgfonlayer}{nodelayer}
		\node [style=vertex, label=below right:\contour{white}{$c_+$}] (0) at (-2, 3) {};
		\node [style=vertex] (1) at (-3.25, 1.5) {};
		\node [style=vertex, label=above right:\contour{white}{$c_-$}] (2) at (-2, -0) {};
		\node [style=vertex] (3) at (0.5, 2.5) {};
		\node [style=vertex] (4) at (0, 0.5) {};
	\end{pgfonlayer}
	\begin{pgfonlayer}{edgelayer}
	
	\begin{scope}
	\clip (2.center) to (4.center) to (3.center) to (0.center) to (2.center);
	\draw[draw=red, fill=red!20] (0) circle (15pt);
	\draw[draw=red, fill=red!20] (1) circle (15pt);
	\draw[draw=red, fill=red!20] (2) circle (15pt);	
	\end{scope}
	
		\draw[b] (0) to (3);
		\draw[b] (1) to (2);
		\draw[b] (2) to (4);
		\draw[b] (4) to (3);
		\draw[b] (0) to (2);
		
		\draw[map] (0) to (3);
		\draw[map] (1) to (2);
		\draw[map] (2) to (4);
		\draw[map] (4) to (3);
		
		\draw[map] (0) to (2);
		
		\node (x) at (-0.5, 1.5) {\contour{white}{$f_c$}};
		\node (e=) at (-1.7, 1.5) {\contour{white}{$e_=$}};
		\node (x) at (-2.9, 0.5) {\contour{white}{$e_-$}};
		\node (x) at (-1.5, -0.8) {\contour{white}{$m^{c,+}$}};
	\end{pgfonlayer}
\end{tikzpicture}
\begin{tikzpicture}
	\begin{pgfonlayer}{nodelayer}
		\node [style=vertex, label=below right:\contour{white}{$c_+$}] (0) at (-2, 3) {};
		\node [style=vertex] (1) at (-3.25, 1.5) {};
		\node [style=vertex, label=above right:\contour{white}{$c_-$}] (2) at (-2, -0) {};
		\node [style=vertex] (3) at (0.5, 2.5) {};
		\node [style=vertex] (4) at (0, 0.5) {};
	\end{pgfonlayer}
	\begin{pgfonlayer}{edgelayer}
	
	\begin{scope}
	\clip (2.center) to (4.center) to (3.center) to (0.center) to (2.center);
	\draw[draw=red, fill=red!20] (0) circle (15pt);
	\draw[draw=red, fill=red!20] (1) circle (15pt);
	\draw[draw=red, fill=red!20] (2) circle (15pt);	
	\end{scope}
	
		\draw[b] (1) to (0);
		\draw[b] (0) to (3);
		\draw[b] (2) to (4);
		\draw[b] (4) to (3);
		\draw[b] (0) to (2);
		
		\draw[map] (1) to (0);
		\draw[map] (0) to (3);
		\draw[map] (2) to (4);
		\draw[map] (4) to (3);
		
		\draw[map] (0) to (2);
		
		\node (x) at (-0.5, 1.5) {\contour{white}{$f_c$}};
		\node (e=) at (-1.7, 1.5) {\contour{white}{$e_=$}};
		\node (x) at (-3, 2.3) {\contour{white}{$e_+$}};
		\node (x) at (-1.5, -0.8) {\contour{white}{$m^{c,-}$}};
	\end{pgfonlayer}
\end{tikzpicture}
\caption{\label{edge rotation}The three maps of the form $m^{c,s}$ as constructed from $m\in \sM_n$.}
	
\end{figure}

We can naturally identify the edge $e_s$ in $m$ with the ``rotated edge'' $e_=$ in $m^{c,s}$ (when $s\in\{+,-\}$) and thus have a natural identification between edges of $m$ and edges of $m^{c,s}$. Faces and vertices cannot be as readily identified between $m$ and $m^{c,s}$, because the number of faces and vertices may increase or decrease; and, even though the number of corners remains unchanged after an edge rotation, defining a 1-to-1 correspondence is not entirely canonical, although there is one that is compatible with our choice of the rerooting, in the sense that it allows us to interpret the new choice of the root corner as ``leaving it unchanged''.

Corners other than those involving the vertices of $c,c_-,c_+$ are of course untouched, and will be denoted by the same symbols in $m$ and $m^{c,s}$. If $s\in\{-,+\}$, then we shall identify $c_-$ and $c_+$ with the two (not necessarily distinct) corners joined by the newly drawn edge $e_=$ in the face lying directly to the right of $e_=$, oriented from $c_-$ to $c_+$ (again, see Figure~\ref{edge rotation}).

Notice that, if $s=+$, then the number of corners around the vertex $v_+$ of $c_+$ is unchanged; if those corners were $c_+,c_1,\ldots,c_k$ in clockwise order around $v_+$, starting with $c_+$ in the map $m$, we will identify them with the $k+1$ corners around the vertex of $c_+$ in $m^{c,s}$, having already identified $c_+$, by just keeping the same order.

	

The dynamic on maps we shall be considering throughout this paper is given by a Markov chain $\R^n$ on $\sM_n$ which is such that, assuming $\R^n_k=m$ for some $m\in\sM_n$, we have $\R^n_{k+1}=m^{c,s}$, where $c$ and $s$ are independent random variables, $s$ being uniformly distributed in $\{=,+,-\}$ and $c$ being a corner of $m$ other than its root corner, chosen uniformly at random; in other words, transitions probabilities for $\R^n$ are of the form
\begin{equation}p_\R(m,m')=\frac{1}{3(2n-1)}\sum_{c\in C(m)\setminus \rho}\left(1_{m'=m^{c,+}}+1_{m'=m^{c,-}}+1_{m'=m}\right),\label{R^n transitions}\end{equation}
where $C(m)$ is the set of all corners of $m$ and $\rho$ is its root corner.
\begin{lemma}
The Markov chain $\R^n$ is reversible, aperiodic and irreducible.	
\end{lemma}
\begin{proof}
The first two properties are clear by construction (they can be inferred immediately from expression \eqref{R^n transitions} for the transition probabilities).

As for irreducibility, we shall show that every map $m\in\sM_n$ can be turned into the map $m_0$ made of $n$ nested loops, rooted in the corner within the central loop, via a sequence of edge rotations.

 Given $m\in\sM_n$, consider the face $f_\rho$ containing the root corner $\rho$ of $m$. Suppose it has clockwise contour $\rho,c_1,\ldots,c_k$, with $k\geq 1$: by considering the edge-rotated map $m^{c_1,-}$ we can reduce the degree of $f_\rho$ by 1 (by which we mean the the face containing the root corner in $m^{c_1,-}$ has degree $k$ rather than $k+1$). We can therefore reduce $m$, via a sequence of edge rotations, to a map whose root corner lies within a loop (i.e.~such that the root edge is a loop). Now, given any map $\tilde{m}$ such that the root edge is a loop, consider the first corner $c$ in counterclockwise order around the origin, starting with the root corner $\rho$, such that $c$ lies within a face $f_c$ of degree strictly more than 2, whose clockwise contour we will call $c,c_1,\ldots,c_k$ (with $k>1$). Taking the map $\tilde{m}^{c_1,-}$ decreases the degree of the face $f_c$ by 1: repeating this operation yields a map such that the root edge is a loop and all corners around the origin lie within faces of degree 1 or 2. Such a map can only be $m_0$: if one draws it on the plane in such a way that the root corner lies within the infinite face, so that the root edge is an ``external'' loop, one finds that the finite face adjacent to it is either a degree 1 face -- in which case the map has only one edge -- or has degree 2, in which case its boundary is completed by one ``internal'' loop; repeating this argument inductively identifies the map as $m_0$.
 \end{proof}

\section{From edge rotations on maps to edge flips on quadrangulations}\label{sec:from rotations to flips}

We now wish to relate our edge rotation dynamic on planar maps to the Markov chain of edge flips on quadrangulations as introduced in~\cite{CS2019} -- or rather, to a slight variant thereof.

The edge flip Markov chain $\F^n$ on the set $\sQ_n$ was introduced in \cite{CS2019} as a chain whose steps consist in, given a quadrangulation, selecting one of its edges uniformly at random and then making an independent uniform choice among the following three options: leaving it unchanged, flipping it clockwise or flipping it counterclockwise. The choice of the root edge was allowed and flipping the root edge would preserve its orientation.
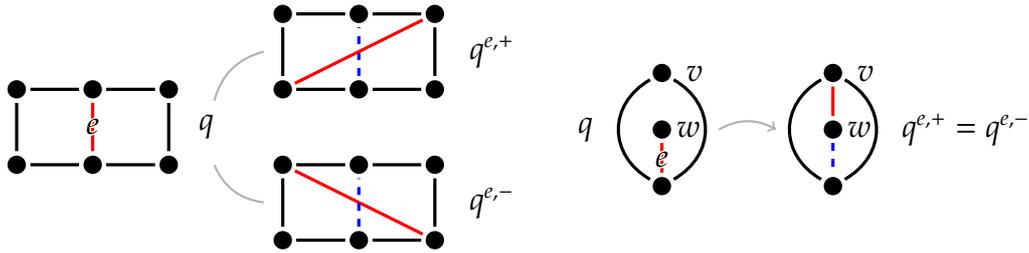
\begin{figure}[t]\centering
\begin{minipage}{7cm}\begin{tikzpicture}
	\begin{pgfonlayer}{nodelayer}
		\node [style=vertex] (0) at (-4, 3) {};
		\node [style=vertex] (1) at (-4, 4) {};
		\node [style=vertex] (2) at (-3, 4) {};
		\node (x) at (-3,3.5) {\contour{white}{$e$}};
		\node [style=vertex] (3) at (-2, 4) {};
		\node [style=vertex] (4) at (-2, 3) {};
		\node [style=vertex] (5) at (-3, 3) {};
		\node [style=vertex] (6) at (-0.5, 4) {};
		\node [style=vertex] (7) at (0.5, 5) {};
		\node [style=vertex] (8) at (1.5, 4) {};
		\node [style=vertex] (9) at (0.5, 4) {};
		\node [style=vertex] (10) at (-0.5, 5) {};
		\node [style=vertex] (11) at (1.5, 5) {};
		\node [style=vertex] (12) at (0.5, 3) {};
		\node [style=vertex] (13) at (1.5, 2) {};
		\node [style=vertex] (14) at (0.5, 2) {};
		\node [style=vertex] (15) at (-0.5, 3) {};
		\node [style=vertex] (16) at (-0.5, 2) {};
		\node [style=vertex] (17) at (1.5, 3) {};
		\node [] (30) at (-1.5, 3.5) {$q$};
		\node [] (33) at (2.25, 4.5) {$q^{e,+}$};
		\node [] (34) at (2.25, 2.5) {$q^{e,-}$};
	\end{pgfonlayer}
	\begin{pgfonlayer}{edgelayer}
	\draw [thick, bend left, black!30] (-1.4,3.85) to (-0.75,4.5);
	\draw [thick, bend right, black!30] (-1.4,3) to (-0.75,2.5);
		\draw [very thick] (0) to (1);
		\draw [very thick] (1) to (2);
		\draw [very thick] (2) to (3);
		\draw [very thick] (3) to (4);
		\draw [very thick] (4) to (5);
		\draw [very thick] (5) to (0);
		\draw [very thick, red] (5) to (2);
		\draw [very thick] (6) to (10);
		\draw [very thick] (10) to (7);
		\draw [very thick] (7) to (11);
		\draw [very thick] (11) to (8);
		\draw [very thick] (8) to (9);
		\draw [very thick] (9) to (6);
		\draw [blue, very thick, dashed] (9) to (7);
		\draw [very thick, red] (6) to (11);
		\draw [very thick] (16) to (15);
		\draw [very thick] (15) to (12);
		\draw [very thick] (12) to (17);
		\draw [very thick] (17) to (13);
		\draw [very thick] (13) to (14);
		\draw [very thick] (14) to (16);
		\draw [very thick, red] (13) to (15);
		\draw [blue, very thick, dashed] (14) to (12);
	\end{pgfonlayer}
\end{tikzpicture}\end{minipage}\quad
\begin{minipage}{6cm}
\begin{tikzpicture}
	\begin{pgfonlayer}{nodelayer}
		\node [style=vertex] (0) at (-1.5, -1) {};
		\node [style=vertex] (1) at (-1.5, -0.25) {};
		\node [style=vertex, label=right:$v$] (2) at (-1.5, 0.5) {};
		\node [style=vertex] (3) at (0.75, -1) {};
		\node [style=vertex, label=right:$v$] (4) at (0.75, 0.5) {};
		\node [style=vertex] (5) at (0.75, -0.25) {};
		\node [] (w) at (1.1, -0.25) {$w$};
		\node [] (w) at (-1.15, -0.25) {$w$};
		\node [] (6) at (-2.5, -0.25) {$q$};
		\node [] (7) at (2.5, -0.25) {$q^{e,+}=q^{e,-}$};
		\node (x) at (-1.5,-0.65) {\contour{white}{$e$}};
	\end{pgfonlayer}
	\begin{pgfonlayer}{edgelayer}
	\draw [thick, ->, bend left, black!30] (-0.75,-0.25) to (0,-0.25);
		\draw [very thick, bend left=60, looseness=1.25] (0) to (2);
		\draw [very thick, bend right=60, looseness=1.25] (0) to (2);
		\draw [very thick, red] (0) to (1);
		\draw [very thick, bend left=60, looseness=1.25] (3) to (4);
		\draw [very thick, bend right=60, looseness=1.25] (3) to (4);
		\draw [very thick, red] (5) to (4);
		\draw [very thick, blue, dashed] (5) to (3);
	\end{pgfonlayer}
\end{tikzpicture}\end{minipage}
\caption{\label{fig:flips}Clockwise and counterclockwise flips for a simple and a double edge in a quadrangulation.}\end{figure}

More formally, given a quadrangulation $q\in\sQ_n$ and an edge $e$ of $q$, we denote by $q^{e,+}$ (resp.~$q^{e,-}$), \emph{the quadrangulation obtained from $q$ by flipping edge $e$ clockwise} (resp.~\emph{counterclockwise}), by which we mean the quadrangulation given by the following procedure:
\begin{itemize}
\item if $e$ is adjacent to two distinct faces of $q$, erase $e$ from $q$ (thus obtaining a new face with exactly 6 corners) and replace it with the edge obtained by rotating $e$ clockwise (resp. counterclockwise) by one corner (see Figure~\ref{fig:flips}).
\item if $e$ is an internal edge within a degenerate face, let $v$ be the vertex of that face that is \emph{not} an endpoint of $e$ and let $w$ be the endpoint of $e$ having degree 1; erase $e$ and replace it with an edge within the same face having endpoints $v,w$. If $e$ is the root edge of $q$, let the newly drawn edge be the root of $q^{e,+}$ (resp.~$q^{e,-}$), oriented in the same way as before (with respect to $w$).\end{itemize}

The edge flip Markov chain as originally described has transition probabilities
$$p_\F(q,q')=\frac{1}{6n}\sum_{e\in E(q)}\left(1_{q'=q^{e,+}}+1_{q'=q^{e,-}}+1_{q=q'}\right).$$

In order to directly relate a dynamic on maps to edge flips on 	quadrangulations, however, one is led to consider a variant of the chain $\F^n$ that does not allow flipping the root edge. Indeed, the Tutte bijection $\Phi$ assigns very different roles to quadrangulation vertices at even and odd distance from the origin: since a root flip (at least as described) would change the parity of the distance to the origin for each vertex in the quadrangulation, the two maps corresponding to the quadrangulation before and after the flip are potentially completely different from each other.

It therefore becomes necessary to redefine or completely eliminate root edge flips from the chain $\F^n$; in particular, we shall from here on consider a new edge flip Markov chain where the choice of the edge to flip is uniform \emph{among all edges other than the root edge}. We shall still refer to this as \emph{the edge flip} Markov chain on $\sQ_n$ and we shall denote it by $\FF^n$; its transition probabilities are of the form
  $$p_{\FF}(q,q')=\frac{1}{3(2n-1)}\sum_{e\in E(q)\setminus \rho}\left(1_{q'=q^{e,+}}+1_{q'=q^{e,-}}+1_{q=q'}\right),$$
  where $\rho$ is the root edge of $q$.
  
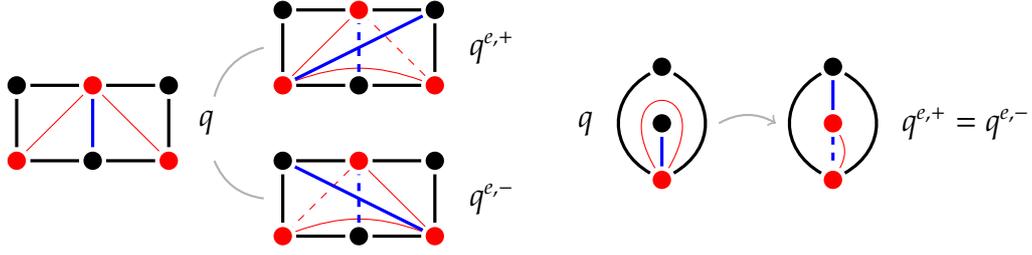
\begin{figure}[t]\centering
\begin{minipage}{7cm}\begin{tikzpicture}
	\begin{pgfonlayer}{nodelayer}
		\node [style=vertex, fill=red] (0) at (-4, 3) {};
		\node [style=vertex] (1) at (-4, 4) {};
		\node [style=vertex, fill=red] (2) at (-3, 4) {};
		\node [style=vertex] (3) at (-2, 4) {};
		\node [style=vertex, fill=red] (4) at (-2, 3) {};
		\node [style=vertex] (5) at (-3, 3) {};
		\node [style=vertex, fill=red] (6) at (-0.5, 4) {};
		\node [style=vertex, fill=red] (7) at (0.5, 5) {};
		\node [style=vertex, fill=red] (8) at (1.5, 4) {};
		\node [style=vertex] (9) at (0.5, 4) {};
		\node [style=vertex] (10) at (-0.5, 5) {};
		\node [style=vertex] (11) at (1.5, 5) {};
		\node [style=vertex, fill=red] (12) at (0.5, 3) {};
		\node [style=vertex, fill=red] (13) at (1.5, 2) {};
		\node [style=vertex] (14) at (0.5, 2) {};
		\node [style=vertex] (15) at (-0.5, 3) {};
		\node [style=vertex, fill=red] (16) at (-0.5, 2) {};
		\node [style=vertex] (17) at (1.5, 3) {};
		\node [] (30) at (-1.5, 3.5) {$q$};
		\node [] (33) at (2.25, 4.5) {$q^{e,+}$};
		\node [] (34) at (2.25, 2.5) {$q^{e,-}$};
	\end{pgfonlayer}
	\begin{pgfonlayer}{edgelayer}
	\draw [thick, bend left, black!30] (-1.4,3.85) to (-0.75,4.5);
	\draw [thick, bend right, black!30] (-1.4,3) to (-0.75,2.5);
	\draw [red] (0) to (2);
	\draw [red] (2) to (4);
	\draw [red, dashed] (8) to (7);
	\draw [red] (7) to (6);
	\draw [red, bend right=20] (8) to (6);
	\draw [red, dashed] (12) to (16);
	\draw [red] (12) to (13);
	\draw [red, bend right=20] (13) to (16);
	
		\draw [very thick] (0) to (1);
		\draw [very thick] (1) to (2);
		\draw [very thick] (2) to (3);
		\draw [very thick] (3) to (4);
		\draw [very thick] (4) to (5);
		\draw [very thick] (5) to (0);
		\draw [very thick, blue] (5) to (2);
		\draw [very thick] (6) to (10);
		\draw [very thick] (10) to (7);
		\draw [very thick] (7) to (11);
		\draw [very thick] (11) to (8);
		\draw [very thick] (8) to (9);
		\draw [very thick] (9) to (6);
		\draw [blue, very thick, dashed] (9) to (7);
		\draw [very thick, blue] (6) to (11);
		\draw [very thick] (16) to (15);
		\draw [very thick] (15) to (12);
		\draw [very thick] (12) to (17);
		\draw [very thick] (17) to (13);
		\draw [very thick] (13) to (14);
		\draw [very thick] (14) to (16);
		\draw [very thick, blue] (13) to (15);
		\draw [blue, very thick, dashed] (14) to (12);
	\end{pgfonlayer}
\end{tikzpicture}\end{minipage}\quad
\begin{minipage}{6cm}
\begin{tikzpicture}
	\begin{pgfonlayer}{nodelayer}
		\node [style=vertex, fill=red] (0) at (-1.5, -1) {};
		\node [style=vertex] (1) at (-1.5, -0.25) {};
		\node [style=vertex] (2) at (-1.5, 0.5) {};
		\node [style=vertex, fill=red] (3) at (0.75, -1) {};
		\node [style=vertex] (4) at (0.75, 0.5) {};
		\node [style=vertex, fill=red] (5) at (0.75, -0.25) {};
		\node [] (6) at (-2.5, -0.25) {$q$};
		\node [] (7) at (2.5, -0.25) {$q^{e,+}=q^{e,-}$};
	\end{pgfonlayer}
	\begin{pgfonlayer}{edgelayer}
	\draw[red, looseness=20, out=60, in=120] (0) to (0);
	\draw[red, bend right=30] (3) to (5);
	\draw [thick, ->, bend left, black!30] (-0.75,-0.25) to (0,-0.25);
		\draw [very thick, bend left=60, looseness=1.25] (0) to (2);
		\draw [very thick, bend right=60, looseness=1.25] (0) to (2);
		\draw [very thick, blue] (0) to (1);
		\draw [very thick, bend left=60, looseness=1.25] (3) to (4);
		\draw [very thick, bend right=60, looseness=1.25] (3) to (4);
		\draw [very thick, blue] (5) to (4);
		\draw [very thick, blue, dashed] (5) to (3);
	\end{pgfonlayer}
\end{tikzpicture}\end{minipage}
\caption{\label{fig:flips to edge rotations}Clockwise and counterclockwise flips for a simple and a double edge in a quadrangulation.}\end{figure}

  \begin{proposition}\label{prop:from rotations to flips}
  Given a quadrangulation $q\in\sQ_n$ with root edge $\rho$, an edge $e\in E(q)\setminus\rho$ and $s\in\{+,-\}$, we have $\Phi(q^{e,s})=\Phi(q)^{c,s}$, where $\Phi$ is the Tutte bijection from Section~\ref{sec:trivial bijection} and $c$ is the corner of $\Phi(q)$ that corresponds to the edge $e$ of $q$.
  \end{proposition}
  
  \begin{proof}
Consider the case where $e$ is not an internal edge within a degenerate face, but rather is adjacent to two distinct faces of $q$, within each of which a map edge is drawn by the construction $\Phi$. Orient the edge $e$ away from its endpoint at even distance from the origin and let $f_-$ be the face lying to its right and $f_+$ the one lying to its left. If we take $c$ to be the corner of $\Phi(q)$ that the edge $e$ is issued from, it should be clear that the \emph{map} edges $e_-$ and $e_+$ constructed as a function of $c$ in Section~\ref{sec:edge rotations} correspond to quadrangulation faces $f_-$ and $f_+$ respectively.

Now consider for example the edge-flipped quadrangulation $q^{e,+}$; it is clear that, since $e$ is not the root edge, the parity of distances from the origin is unchanged. The endpoints of $e_-$ are therefore still ``real vertices'' which need to be joined by a map edge lying within the quadrangular face next to the flipped edge $e$: we can draw $e_-$ exactly as before. This is in contrast to the edge $e_+$, which would now cross the flipped edge $e$, and therefore needs to be erased. The edge that replaces $e_+$ is an edge $e_=$ that would form a triangle containing $c$, along with $e_-$ and $e_+$, in the original map (left part of Figure~\ref{fig:flips to edge rotations}).

Furthermore, notice that the fact that the root edge is unchanged in $q^{e,s}$ implies that the root corner of $\Phi(q^{e,s})$ must still be the one that the quadrangulation root edge is issued from. In order for this to be true in the non-trivial cases where the root corner of $\Phi(q)$ is ``split'' by the addition of $e_=$, the root corner must become the part of the corner lying outside the $e_-,e_+,e_=$ triangle, exactly as described in Section~\ref{sec:edge rotations}.

This shows that $\Phi(q^{e,s})=\Phi(q)^{c,s}$ when $e$ is not an internal edge within a degenerate face and $s=+$ (including the case where $e_=$ ends up being a loop, which one can see arise in the situation depicted in Figure~\ref{fig:flips to edge rotations} when edges on the boundary of $f_+$ and $f_-$ are identified); the case of $s=-$ is identical.

Now consider the case where the endpoint $u$ of $e$ at odd distance from the origin has degree 1 in $q$. This corresponds to its degenerate face having two face vertices and one real vertex, and a map loop edge ($e_-=e_+$) being drawn within it by the construction $\Phi$. In this case, we know that $q^{e,s}$, for $s=\pm$, is the quadrangulation that replaces $e$ with an edge drawn between $u$ and the other vertex on the external boundary of the degenerate face. It is immediately apparent that $\Phi(q^{e,s})$ is, indeed, the map which replaces the loop $e_-=e_+$ with an edge $e_=$ having a brand new vertex as an endpoint, which is $\Phi(q)^{c,s}$. The rooting poses no real issues, since $e$ cannot be the root edge of $q$ and the identification of corners between $\Phi(q)$ and $\Phi(q^{e,s})$ is clear.

Finally, the case where the endpoint of $e$ at \emph{even} distance from the root has degree 1 in $q$ is precisely the inverse of the one above.
  \end{proof}

  We have thus shown that the two Markov chains ${\FF}^n$ and $\R^n$ are isomorphic; in particular, they have the same relaxation and mixing time. 
  
  The main result in \cite{CS2019} consisted in the following bounds for the spectral gap of the Markov chain $\F_n$:
  
\begin{theorem*}[C., Stauffer]
Let $\gamma_n$ be the spectral gap of the edge flip Markov chain $\F_n$ on the set $\sQ_n$ of rooted quadrangulations with $n$ faces. There are positive constants $C_1,C_2$ independent of $n$ such that 
$$C_1n^{-\frac{11}{2}}\leq \gamma_n\leq C_2n^{-\frac54}.$$
Consequently, the mixing time for $\F_n$ is $O(n^{13/2})$.
\end{theorem*}

  While the upper bound above for the spectral gap of $\F_n$ immediately yields a lower bound for the relaxation time of ${\FF}^n$ and therefore $\R^n$, an upper bound for the relaxation time of ${\FF}^n$ cannot trivially be gleaned from \cite{CS2019}. The rest of this paper will therefore be devoted to analysing the chain ${\FF}^n$ to obtain an upper bound which applies to the edge rotation Markov chain. Though the general strategy is not dissimilar to the one employed in \cite{CS2019}, some ad hoc constructions and ideas will be necessary; as a result, we will have a partially new proof of an upper bound for the relaxation time of the original chain $\F^n$ which (mostly) does not rely on the Cori-Vanquelin-Schaeffer correspondence with plane trees, and should therefore be better suited for further generalisations.


\section{Growing quadrangulations uniformly at random}\label{sec:growth algorithm}
Consider the following operation which, given a quadrangulation $q\in\sQ_n$ (with $n>1$) and a corner $c$ of $q$, yields a quadrangulation $\coll{q,c}$ in $\sQ_{n-1}$. If the face $f$ of $q$ containing $c$ has $4$ distinct vertices and $c,c_1,c_2,c_3$ is a clockwise contour of $f$, then ``collapse'' $f$ by identifying the edge joining $c$ to $c_1$ with the one joining $c$ to $c_3$ and the edge joining $c_1$ to $c_2$ with the one joining $c_3$ to $c_2$ as in Figure~\ref{fig:collapse a face}, thus identifying the vertices of corners $c_1$ and $c_3$. If $f$ has fewer than $4$ distinct vertices then it has exactly $3$, one of which is adjacent to two corners $c_1, c_3$ of $f$, where $c_1,c_2,c_3,c_4$ is a clockwise contour; in this case, whatever $c$, identify the edge joining $c_1$ to $c_2$ with the edge joining $c_4$ to $c_1$ and the edge joining $c_2$ to $c_3$ with the edge joining $c_3$ to $c_4$, thus also identifying the vertices of $c_2$ and $c_4$. Note that the latter procedure applies to the case of $f$ being a degenerate face (lower part of Figure~\ref{fig:collapse a face}). If the root corner does not belong to $f$, it is simply preserved; if it belongs to $f$, then we root the quadrangulation $\coll{q,c}$, in such a way that the root edge is the collapsed image of the original root edge, oriented as before.

We shall say that two quadrangulations $q\in\sQ_n$ and $q'\in\sQ_{n-1}$ \emph{differ by collapsing a face} if there is a corner $c$ of $q$ such that $q'=\coll{q,c}$.

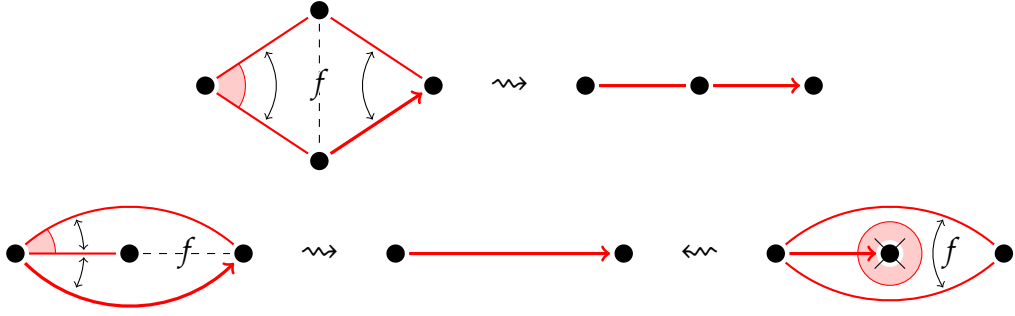
\begin{figure}\centering
\begin{tikzpicture}
	\begin{pgfonlayer}{nodelayer}
		\node [style=vertex] (0) at (-1.5, -0) {};
		\node [style=vertex] (1) at (0, 1) {};
		\node [style=vertex] (2) at (1.5, -0) {};
		\node [style=vertex] (3) at (0, -1) {};
		\node (f) at (0,0) {\contour{white}{$f$}};
		
		\node [style=vertex] (0b) at (3.5, -0) {};
		\node [style=vertex] (1b) at (5, 0) {};
		\node [style=vertex] (2b) at (6.5, -0) {};
		
		\node [] (->) at (2.5, 0) {$\rightsquigarrow$};
	\end{pgfonlayer}
	\begin{pgfonlayer}{edgelayer}
	
	\begin{scope}
	\clip (1.center) to (2.center) to (3.center) to (0.center) to (1.center);
	\draw[draw=red, fill=red!20] (0) circle (15pt);
	\end{scope}
	
	\draw[<->, bend left] (-0.7,0.45) to (-0.7,-0.45);
	\draw[<->, bend right] (0.7,0.45) to (0.7,-0.45);
	\draw[quad] (0) to (1) to (2) to (3) to (0);
	\draw[quad, very thick,->] (3) to (2);
	\draw[dashed] (1) to (3);
	
	\draw[quad, very thick] (0b) to (1b) [->] to (2b);
	
	\end{pgfonlayer}
\end{tikzpicture}\\[5pt]
\begin{tikzpicture}
	\begin{pgfonlayer}{nodelayer}
		\node [style=vertex] (0) at (-1.5, -0) {};
		\node [style=vertex] (1) at (0, 0) {};
		\node [style=vertex] (2) at (1.5, -0) {};
		\node (f) at (0.75,0) {\contour{white}{$f$}};
		
		\node [style=vertex] (0b) at (3.5, -0) {};
		\node [style=vertex] (2b) at (6.5, -0) {};
		
		\node [] (->) at (2.5, 0) {$\rightsquigarrow$};
		
		\node [style=vertex] (0c) at (8.5, -0) {};
		\node [style=vertex] (1c) at (10, 0) {};
		\node [style=vertex] (2c) at (11.5, -0) {};
		\node (fc) at (10.8,0) {\contour{white}{$f$}};
		\draw (9.8,0.2) to (10.2,-0.2);
		\draw (10.2,0.2) to (9.8,-0.2);
		
		\node [] (<-) at (7.5, 0) {$\leftsquigarrow$};
	\end{pgfonlayer}
	\begin{pgfonlayer}{edgelayer}
	
	\begin{scope}
	\clip[bend left=45] (0.center) to (2.center) [bend left=0] to (0.center);
	\draw[draw=red, fill=red!20] (0) circle (15pt);
	\end{scope}
	
	\draw[<->, bend left=10] (-0.7,0.45) to (-0.6,0.05);
	\draw[<->, bend right=10] (-0.7,-0.45) to (-0.6,-0.05);
	\draw[quad, bend left=45] (0.center) to (2.center);
	\draw[quad, bend left=45, very thick, <-] (2) to (0);
	\draw[quad] (0) to (1);
	\draw[dashed] (1) to (2);
	
	\draw[quad, very thick,->] (0b) to (2b);
	
	\draw[draw=red, fill=red!20] (1c) circle (12pt);
	\draw[<->, bend right] (10.7,0.45) to (10.7,-0.45);
	\draw[quad, bend left=45] (0c.center) to (2c.center) to (0c.center);
	\draw[quad, very thick, ->] (0c) to (1c);
	
	\end{pgfonlayer}
\end{tikzpicture}
\caption{\label{fig:collapse a face}Collapsing a face in a quadrangulation.}	
\end{figure}

What we need is a ``hierarchy'' like the one described for coloured plane trees in Section 5 of \cite{CS2019}, but for quadrangulations, where the adjacency condition of differing by erasing a leaf is replaced by the one given by collapsing faces. What we wish to produce is a collection of mappings $g_n:\sQ_n\times \sQ_{n-1}\to \mathbb{R}_{\geq 0}$ with the following properties:
\begin{itemize}
	\item[(i)] $g_n(q,q')=0$ if $q'$ cannot be obtained from $q$ by collapsing a face;
	\item[(ii)] $\displaystyle\sum_{q'\in\sQ_{n-1}}g_n(q,q')=1$ for all $q$ in $\sQ_n$;
	\item[(iii)] $\displaystyle\sum_{q\in\sQ_n}g_n(q,q')=\frac{|\sQ_n|}{|\sQ_{n-1}|}$ for all $q'$ in $\sQ_{n-1}$.
\end{itemize}

Such a collection of mappings $g_n$ can be built explicitly from the mappings $f_n$ given in Section~5 of \cite{CS2019}. In order to do this, we need to briefly recall some notation and standard results.

\begin{df}
A \emph{labelled tree} is a plane tree $t$ (i.e.~a rooted planar map with a single face) endowed with a labelling $l:V(t)\to \mathbb{Z}$ such that
\begin{itemize}
\item if $\emptyset$ is the origin of $t$, $l(\emptyset)=0$;
\item for any vertex $v\in V(t)\setminus\{\emptyset\}$, $|l(v)-l(p(v))|\in \{1,-1,0\}$, where $p(v)$ denotes the parent of $v$.	
\end{itemize}
We shall call $\LT_n$ the set of all labelled trees with $n$ edges, and conventionally set $\LT_0=\{\bullet\}$ to be the set containing the graph with a single vertex labelled 0.
\end{df}

We shall also write $\sQ_n^\bullet$ for the set of all pairs $(q,\delta)$, where $q\in\sQ_n$ and $\delta\in V(q)$, i.e.~for the set of all \emph{pointed} (rooted) quadrangulations of the sphere with $n$ faces. We shall conventionally define the set $\sQ_0=\{\rightarrow\}$ as the one containing a rooted planar map with a single edge and two distinct vertices; as a consequence, $\sQ_0^\bullet$ has two elements. Also by convention, but consistently with our previous definition, we shall set $\coll{q,c}$, where $c$ is any corner in a quadrangulation $q\in\sQ_1$, to be the ``root-only map'' $\rightarrow\in\sQ_0$.

Labelled trees (and pointed quadrangulations) will be useful thanks to the Cori-Vanquelin-Schaeffer correspondence (see~\cite{Sch98}), which is an explicit bijective construction $\phi:\LT_n\times\{-1,1\}\to\sQ_n^\bullet$ transforming tree labels into quadrangulation graph distances: given $t\in\LT_n$ and $\epsilon\in\{-1,1\}$, the mapping $\phi$ naturally induces an identification between vertices of $t$ and vertices of $\phi(t,\epsilon)$ other than the distinguished vertex such that, if $l$ is the labelling of $t$, we have $l(v)=\dgr(v,\delta)-\dgr(\delta,\emptyset)$, where $v$ is interpreted as a vertex of $t$ in the left hand side of the equation and as a vertex of $\phi(t,\epsilon)$ in the right hand side, $\emptyset$ is the origin of $\phi(t,\epsilon)$ and $\delta$ its distinguished vertex, and $\dgr$ is the graph distance on the vertex set of $\phi(t,\epsilon)$.

In Section~5 of \cite{CS2019} we provided a collection of maps $f_n:\LT_n\times\LT_{n-1}\to\mathbb{R}$ with the exact properties stated for $g_n$ above, where every instance of $\sQ_k$ is replaced by $\LT_k$ and (i) is replaced by

\begin{itemize}
\item[(i)]	$f_n(t,t')=0$ if $t'$ cannot be obtained from $t$ by erasing a leaf.
\end{itemize}

In particular, we showed that such properties hold for the collection of maps constructed recursively as follows:
\begin{itemize}
\item if $(t,t')\in\LT_n\times\LT_{n-1}$ do not differ by erasing a leaf, $f_n(t,t')=0$;
\item if $(t,t')\in\LT_1\times\LT_{0}$, set $f_1(t,t')=1$; 
\item if $(t,t')\in\LT_n\times\LT_{n-1}$, where $n>1$, differ by erasing a leaf, consider the subtrees $L(t),L(t')$ containing the leftmost child of the root vertex and its descendants in $t,t'$ respectively; if $|L(t)|=i$ and $|L(t')|=i-1$ for some $i>0$, set
$$f_n(t,t')=\frac{i(i+1)(3n-2i-1)}{(n-1)n(n+1)}f_i(L(t),L(t'));$$
otherwise set $R(t)$, $R(t')$ to be the trees obtained by erasing $L(t), L(t')$ from $t,t'$ (as well as the edge joining the root vertex to its leftmost child); we then have $|R(t)|=|R(t')|+1=i$ for some $i>0$ and we set
$$f_n(t,t')=\frac{i(i+1)(3n-2i-1)}{(n-1)n(n+1)}f_i(R(t),R(t')).$$
\end{itemize}

The main reason why the collection of mappings $f_n$ can be used to construct mappings $g_n$ which satisfy the properties we require is the following:
\begin{lemma}\label{leaf erasing to face collapsing}
	If $(t,t')\in\LT_n\times\LT_{n-1}$, $\epsilon=\pm 1$ and $f_n(t,t')>0$, then $F(\phi(t',\epsilon))$ is obtained from $F(\phi(t,\epsilon))$ by collapsing a face, where $F:\bigcup_{i\geq0}\sQ_i^\bullet\rightarrow\bigcup_{i\geq0}\sQ_i$ is the mapping which forgets the distinguished vertex.
\end{lemma}

\begin{figure}\centering
\begin{tikzpicture}
	\begin{pgfonlayer}{nodelayer}
		\node [style=vertex, label=left:$a$] (0) at (0, 0) {};
		\node [style=vertex, label=above:$a$] (1) at (0, 1) {};
		\node [style=vertex, label=right:$a-1$] (2) at (2, 0.5) {};
	\end{pgfonlayer}
	\begin{pgfonlayer}{edgelayer}
		\draw[line width=2pt, brown] (0) to (1);
		
		\draw[quad, out=120, looseness=2.8] (0) to (2);
		\draw[quad, bend left=10] (1) to (2);
		\draw[quad, out=40, in=180, looseness=1] (0) to (2);
	\end{pgfonlayer}
\end{tikzpicture}
\begin{tikzpicture}
	\begin{pgfonlayer}{nodelayer}
		\node [style=vertex, label=left:$a$] (0) at (0, 0) {};
		\node [style=vertex, label=above:$a+1$] (1) at (0, 1) {};
		\node [style=vertex, label=right:$a-1$] (2) at (2, 0.5) {};
	\end{pgfonlayer}
	\begin{pgfonlayer}{edgelayer}
		\draw[line width=2pt, brown] (0) to (1);
		
		\draw[quad, out=120, looseness=3] (0) to (2);
		\draw[quad, bend left=25] (1) to (0);
		\draw[quad, out=40, in=180, looseness=1] (0) to (2);
	\end{pgfonlayer}
\end{tikzpicture}
\begin{tikzpicture}
	\begin{pgfonlayer}{nodelayer}
		\node [style=vertex, label=left:$a$] (0) at (0, 0) {};
		\node [style=vertex, label=above:$a-1$] (1) at (0, 1) {};
		\node [style=vertex, label=below:$a-1$] (2) at (1.7, 0.5) {};
		\node [style=vertex, label=right:$a-2$] (3) at (2.5, 0.5) {};
	\end{pgfonlayer}
	\begin{pgfonlayer}{edgelayer}
		\draw[line width=2pt, brown] (0) to (1);
		
		\draw[quad, bend left] (1) to (3);
		\draw[quad, bend left=10] (0) to (2);
		\draw[quad, bend left] (0) to (1);
		\draw[quad, bend left] (2) to (3);
	\end{pgfonlayer}
\end{tikzpicture}
\caption{\label{fig:erasing leaves to collapsing faces}Erasing a leaf (i.e.~a degree 1 vertex $v$ and the edge $(v,p(v))$ joining it to its parent) from a labelled tree $t\in\LT_n$, whatever the labels, corresponds to collapsing the face built around the edge $(v,p(v))$ in the pointed quadrangulation $\Phi(t,\epsilon)$.}	
\end{figure}
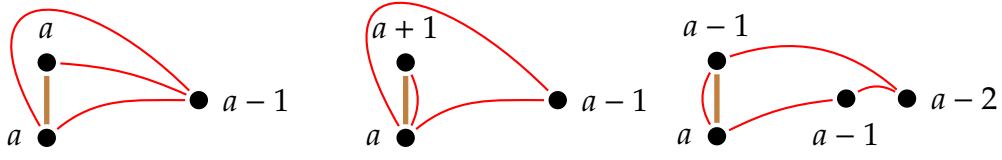

\begin{proof}
	The proof does of course rely on the specific definition of $\phi$ (and is the only part of this paper that does). The quadrangulation $\phi(t,\epsilon)$ can be drawn using the vertex set of $t$ and an added vertex $\delta$ as follows: consider a counterclockwise (cyclic) contour $c_1,\ldots,c_{2n}$ of the one face of $t$; for each $i$, draw a quadrangulation edge joining $c_i$ to its ``target'' corner, which we will take to be the next corner in the contour whose vertex has strictly smaller label than the vertex of $c_i$, or the corner around $\delta$ if the label of the vertex that $c_i$ belongs to is minimal. 
	
	Suppose $t'$ is obtained from $t$ by erasing a leaf $v$ and the edge $(v,p(v))$. If $l(v)=l(p(v))$ or $l(v)=l(p(v))+1$, then the the two quadrangulation edges issued from the corner of $t$ right before the one around $v$ and the one right after $v$ have the same ``target'' corner and enclose a degenerate face of the quadrangulation $\phi(t,\epsilon)$ (see Figure~\ref{fig:erasing leaves to collapsing faces}). Erasing $v$ collapses those two edges into a single edge; targets for corners other than the one around $v$ (which is eliminated) are unaffected. Furthermore, there is no issue with the rooting: $\phi(t,\epsilon)$ is rooted in the edge issued by the root corner of $t$, with an orientation given by $\epsilon$: such an edge does correspond to the edge issued from the root corner of $t'$.
	
	If $l(v)<l(p(v))$, then the matter slightly more complicated. The contour of $t'$ has two fewer corners than the contour of $t$: the quadrangulation edges $e_1$ and $e_2$ issued from the corner before $v$ and the corner around $v$ are eliminated. Let $c$ be the target corner of the corner immediately \emph{after} $v$, which must be around a vertex labelled $l(v)$. Suppose $c$ is \emph{not} the corner of $v$; then all corners having the corner of $v$ as a target in $t$ have $c$ as a target in $t'$: equivalently, all edges adjacent to $v$ in $\phi(t,\epsilon)$ become adjacent to the vertex of $c$ in $\phi(t',\epsilon)$. Eliminating edges $e_1,e_2$ and rerouting all edges adjacent to $v$ to the vertex of $c$ exactly amounts to collapsing the quadrangulation face which encloses the tree edge $(v,p(v))$. If $\tilde{c}$ is the \emph{quadrangulation} corner in $\phi(t,\epsilon)$ corresponding to the edge $e_1$, oriented towards $v$, the quadrangulation $F(\phi(t',\epsilon))$ is $\coll{F(\phi(t,\epsilon)),\tilde{c}}$ (again, the rooting is correctly preserved).
	
	If $c$ is the corner around $v$, then $l(v)$ is minimal and $v$ is the unique vertex carrying label $l(v)$; in that case, the face enclosing the edge $(v,p(v))$ is again degenerate and contains the vertex $\delta$, which is the furthest one from the origin in the quadrangulation $\phi(t,\epsilon)$. In this case, the quadrangulation $\phi(t',\epsilon)$ can be obtained from $t$ and $\phi(t,\epsilon)$ by simply eliminating the original pointed vertex $\delta$ from $\phi(t,\epsilon)$, erasing the tree edge $(v,p(v))$ and renaming vertex $v$ to $\delta$: that way, all one needs to do is erase the two quadrangulation edges that were drawn from the corner of $v$ and from the corner after $v$, which amounts to collapsing the degenerate face that corresponded to the tree edge $(v,p(v))$. The quadrangulation $\phi(t',\epsilon)$ also has its pointing ``moved'' (which is natural, since $\phi(t,\epsilon)$ was pointed in a vertex within the face to be collapsed), but this has no bearing on $F(\phi(t,\epsilon))$.
\end{proof}

\begin{lemma}
	The collection of mappings $g_n:\sQ_n\times\sQ_{n-1}\to\mathbb{R}$ defined as
	$$g_n(q,q')=\frac{1}{n+2}\sum_{v\in V(q)}\sum_{v'\in V(q')}1_{\epsilon_{q,v}=\epsilon_{q',v'}}f_n(t_{q,v},t_{q',v'}),$$
	where $f_n:\LT_n\times\LT_{n-1}\to\mathbb{R}_{\geq 0}$ is defined recursively as described before and where $\phi(t_{q,v},\epsilon_{q,v})=(q,v)$ and $\phi(t_{q',v'},\epsilon_{q',v'})=(q',v')$, satisfies properties (i), (ii) and (iii).
\end{lemma}

\begin{proof}
This is straightforward from the properties of $f_n$.

Indeed, property (i) for $g_n$ is a consequence of Lemma~\ref{leaf erasing to face collapsing}: if there are $v,v'$ in $V(q)$ and $V(q')$ respectively such that $\phi^{-1}(q,v)=(t,\epsilon)$ and $\phi^{-1}(q',v')=(t',\epsilon)$, where $t\in\LT_n$ and $t'\in\LT_{n-1}$ are such that $f_n(t,t')\neq 0$, then, since $t$ and $t'$ differ by erasing a leaf, the quadrangulations $q=F(\phi(t,\epsilon))$ and $q'=F(\phi(t',\epsilon))$ differ by collapsing a face.

As for property (ii), we have
$$\sum_{q'\in\sQ_{n-1}}g_n(q,q')=\frac{1}{n+2}\sum_{v\in V(q)}\sum_{(q',v')\in \sQ^\bullet_{n-1}}1_{\epsilon_{q,v}=\epsilon_{q',v'}}f_n(t_{q,v},t_{q',v'})=$$
$$=\frac{1}{n+2}\sum_{v\in V(q)}\sum_{(t',\epsilon)\in\LT_{n-1}\times\{1,-1\}}1_{\epsilon_{q,v}=\epsilon}f_n(t_{q,v},t')=
\frac{1}{n+2}\sum_{v\in V(q)}\sum_{t'\in\LT_{n-1}}f_n(t_{q,v},t')=
$$
$$=\frac{1}{n+2}\sum_{v\in V(q)}1=1.$$

Similarly, for (iii) one has
$$\sum_{q\in\sQ_{n}}g_n(q,q')=\frac{1}{n+2}\sum_{v'\in V(q')}\sum_{t\in\LT_n}f_n(t,t_{q',v'})=\frac{|\LT_n|(n+1)}{|\LT_{n-1}|(n+2)}=\frac{|\sQ_n|}{|\sQ_{n-1}|}.$$
\end{proof}

\section{Canonical paths}\label{sec:canonical paths}

Given two quadrangulations $q,q'\in\sQ_n$, we intend to build a random canonical path from $q$ to $q'$, that is a probability measure $\P_{q\rightarrow q'}$ on the set $\Gamma_{q\rightarrow q'}$ of all sequences $(q_i,e_i,s_i)_{i=1}^N$ such that
\begin{itemize}
\item for all $i=1,\ldots,N$, we have $q_i\in\sQ_n$ and $e_i\in E(q_i)\setminus\{\rho\}$, where $\rho$ is the root edge of $q_i$, while $s_i=\pm$;
\item $q_{i+1}=q_i^{e_i,s_i}$ for $i=1,\ldots,N-1$;
\item $q_1=q$ and $q_N^{e_N,s_N}=q'$.
\end{itemize}

Note that our aim is to construct these paths in such a way that, given an edge flip $(q,e,s)$, the quantity $\sum_{	q,q'\in\sQ_n}\P_{q\rightarrow q'}\{\gamma \in\Gamma_{q\rightarrow q'}\st (q,e,s)\mbox{ appears in }\gamma\}$ is as small as possible.


The main idea of the construction is to have a canonical way of splitting intermediate quadrangulations in the path into two parts: ideally, we want what we shall call the \emph{right part}, which shrinks with time, to retain as much memory of the initial quadrangulation $q$ as possible, while the \emph{left part} is a growing, increasingly accurate version of $q'$ (see Figure~\ref{fig:eye quadrangulation} for the decomposition). 

Because, however, our canonical split requires an external face to act as a ``separator'' between the left and right parts, it is not possible -- or at least it is not convenient -- to grow the complete quadrangulation $q'$ on the left, since we we have space for a quadrangulation of size at most $n-1$. That is why we select a mapping $F:\sQ_n^2\to\sQ_{n-1}$ (with certain properties) 
and construct the random path from $q$ to $q'$ as
\begin{itemize}
\item a random path from $q$ to a quadrangulation whose left part is $F(q,q')$ and whose right part is empty, distributed according to a probability which will later be called $\P_q^{F(q,q')}$;
\item a random path from the final quadrangulation of the path above to $q'$, whose reverse path is  distributed according to the probability $\P_{q'}^{F(q,q')}$.
\end{itemize}


Our objective will be to describe a random flip path distributed according to the probability measure $\P_q^{F(q,q')}$; this will consist of $n$ concatenated flip subpaths, of which
\begin{itemize}
\item the first is special: it collapses one appropriately chosen random face of $q$ and establishes a ``separating face'' to the right of the root edge; at the end of this flip sequence, the face directly to the right of the root edge separates an empty left quadrangulation $L_0$ from a right quadrangulation $R_0$ of size $n-1$;
\item the $(i+1)$th flip subpath (for $i=1,\ldots,n-1$) turns a quadrangulation with left part $L_{i-1}$ and right part $R_{i-1}$ into a quadrangulation with right part $R_i=\coll{R_{i-1},c}$ for some $c$, and left part $L_i$, where $L_i$ has an additional face with respect to $L_{i-1}$ (in the strong sense that $L_{i-1}=\coll{L_i,c'}$ for some corner $c'$ of $L_i$). Given $(L_{i-1},R_{i-1})$, the quadrangulations $L_i,R_i$ are random, distributed in a way that is based on the growth algorithm from Section~\ref{sec:growth algorithm}. The sequence of flips constituting this subpath will be later denoted by $P((L_{i-1},R_{i-1}),(L_{i},R_{i}))$, and itself consists of three distinct phases:
\begin{itemize}
\item \emph{right phase}: the face of $R_{i-1}$ containing $c$ is replaced, via a local sequence of flips, by a degenerate face, which is then moved within $R_{i-1}$ until it becomes adjacent to the ``separating face'';
\item \emph{central phase}: this is a very short sequence of just 4 edge flips which move the extra degenerate face from one side of the ``separating face'' to the other, making it now part of the left portion of the quadrangulation;
\item \emph{left phase}: the extra degenerate face is moved to the appropriate location in $L_{i-1}$ and then possibly replaced by a non-degenerate face via local flips in order to create the left quadrangulation $L_i$.
\end{itemize}
\end{itemize}

In conclusion, the full canonical path from $q$ 	to $q'$ will consist of
\begin{itemize}
\item a flip sequence modifying $q$ to have a separating face, with a quadrangulation $R_0$ of size $n-1$ on the right and an ``empty quadrangulation'' $L_0$ on the left;
\item for each $i=1,\ldots,n-1$, a right phase, central phase and left phase, after which a face has moved from $R_{i-1}$ into $L_{i-1}$, thus yielding left and right parts $L_i, R_i$, where $|R_i|=n-i-1=n-1-|L_i|$, on either side of a separating face. At the end of this whole process $R_{n-1}$ is empty and $L_{n-1}$ is $F(q,q')$;
\item $n-1$ sequences, each with a reverse left phase, reverse central phase, reverse right phase, which move a face from the left part of the quadrangulation to the right part, ending with a left part of size 0 and a right part of size $n-1$;
\item a final sequence which ``dismantles'' the separating face and moves it to the appropriate location to yield $q'$.
\end{itemize}

The next subsection will formalise the idea of a ``separating face'' and give the description of our canonical left-right decomposition, as well as the law of the sequence $(L_i,R_i)_{i=0}^{n-1}\in\prod_{i=0}^{n-1}\sQ_i\times\sQ_{n-1-i}$ as a function of the pair $(q,\tilde{q})\in\sQ_n\times\sQ_{n-1}$.

Section~\ref{sec:from q to ->coll(q,c)} describes the flip paths used to ``collapse'' a face by turning it into a degenerate face and those that move a degenerate face from one location to another within a quadrangulation. Section~\ref{sec:completing psi} finally explains how to build subpaths of the form $P((L_{i-1},R_{i-1}),(L_i,R_i))$ (which will turn out to be deterministic given $L_{i-1},L_i,R_{i-1},R_i$) by assembling flip sequences from Section~\ref{sec:from q to ->coll(q,c)} into a right phase, central phase and left phase, and establishes our desired estimates.

\subsection{Basic structure of canonical paths}\label{sec:canonical paths 1}


\begin{figure}\centering
\begin{tikzpicture}
	\begin{pgfonlayer}{nodelayer}
		\node [style=vertex] (0) at (-3, -0) {};
		\node [style=vertex] (1) at (-1, -0) {};
		\node [style=vertex] (2) at (-4, -0) {};
		\node [style=vertex] (3) at (-2.5, 0.75) {};
		\node [style=vertex] (4) at (-1.75, 1) {};
		\node at (-2.5,-0.8) {$L$};
	\end{pgfonlayer}
	\begin{pgfonlayer}{edgelayer}
		\draw [style=quad, blue] (2) to (0);
		\draw [style=quad, bend left=90, looseness=2.25, blue] (0) to (1);
		\draw [style=quad, blue] (0) to (3);
		\draw [style=quad, blue] (3) to (4);
		\draw [style=quad, blue] (4) to (1);
		\draw [style=quad, blue, ultra thick,->] (0) to (1);
	\end{pgfonlayer}
\end{tikzpicture}
\begin{tikzpicture}
	\begin{pgfonlayer}{nodelayer}
		\node [style=vertex] (0) at (0, -1) {};
		\node [style=vertex] (1) at (2.5, -1) {};
		\node [style=vertex] (2) at (-1.25, 1.25) {};
		\node [style=vertex] (3) at (2.5, 2.25) {};
		\node [style=vertex] (4) at (0, -1) {};
		\node [style=vertex] (5) at (0, 1) {};
		\node [style=vertex] (6) at (1.25, 1.25) {};
		\node [style=vertex] (7) at (0, -0) {};
		\node [style=vertex] (8) at (1.75, -0.25) {};
		\node [style=vertex] (9) at (0.75, 0.25) {};
		\node [style=vertex] (10) at (2, 1.5) {};
		\node at (1,-1.8) {$R$};
	\end{pgfonlayer}
	\begin{pgfonlayer}{edgelayer}
		\draw [style=quad,->,ultra thick] (0) to (1);
		\draw [style=quad] (0) to (2);
		\draw [style=quad] (3) to (2);
		\draw [style=quad] (1) to (3);
		\draw [style=quad, bend right=45, looseness=1.50] (6) to (3);
		\draw [style=quad] (6) to (5);
		\draw [style=quad] (5) to (2);
		\draw [style=quad] (5) to (7);
		\draw [style=quad] (7) to (0);
		\draw [style=quad, bend left=45, looseness=1.50] (7) to (8);
		\draw [style=quad] (9) to (8);
		\draw [style=quad] (8) to (6);
		\draw [style=quad] (10) to (6);
		\draw [style=quad] (8) to (1);
		\draw [style=quad] (6) to (3);
		\draw [style=quad] (7) to (8);
	\end{pgfonlayer}
\end{tikzpicture}\hspace{.5cm}
\begin{tikzpicture}
\clip[use as bounding box] (-1,-4) rectangle (5,1.8);
	\begin{pgfonlayer}{nodelayer}
		\node [style=vertex] (0) at (-1, -1) {};
		\node [style=vertex] (1) at (1, -1) {};
		\node [style=vertex] (2) at (-0.5, -2) {};
		\node [style=vertex] (3) at (-0.5, -0.5) {};
		\node [style=vertex] (4) at (0.5, -0.5) {};
		\node [style=vertex] (5) at (2.5, -0.5) {};
		\node [style=vertex] (6) at (2.25, -1.5) {};
		\node [style=vertex] (7) at (3, -1) {};
		\node [style=vertex] (8) at (2, 0.75) {};
		\node [style=vertex] (9) at (4, -1.25) {};
		\node [style=vertex] (10) at (3.75, 0.25) {};
		\node [style=vertex] (11) at (3.75, 1) {};
		\node [style=vertex] (12) at (5, -1) {};
		\node [style=vertex] (13) at (1, -1) {};
		\node [style=vertex] (14) at (3.25, -0.5) {};
		\node at (2,-3.8) {$L\cdot R$};
	\end{pgfonlayer}
	\begin{pgfonlayer}{edgelayer}
	\filldraw[red!10] (13.center) [bend right=60, looseness=1.3] to (12.center) [bend right=90, looseness=2.10] to (13.center);
	\filldraw[blue!10] (0.center) [bend right=90, looseness=3.20] to (1.center) [bend right=90, looseness=2.45] to (0.center);
		\draw [style=quad, blue] (2) to (0);
		\draw [style=quad, bend left=90, looseness=2.25,blue,<-, ultra thick] (0) to (1);
		\draw [style=quad, blue] (0) to (3);
		\draw [style=quad, blue] (3) to (4);
		\draw [style=quad, blue] (4) to (1);
		\draw [quad, bend right=45, looseness=1.00, blue] (0) to (1);
		\draw [style=quad] (5) to (6);
		\draw [style=quad] (12) to (11);
		\draw [style=quad] (6) to (9);
		\draw [style=quad] (14) to (5);
		\draw [style=quad, bend left=45, looseness=1.50] (6) to (9);
		\draw [style=quad] (11) to (8);
		\draw [style=quad, bend left=90, looseness=2.00, ultra thick, dashed] (13) to (12);
		\draw [style=quad, bend right=60, looseness=1.25, ultra thick, ->] (13) to (12);
		\draw [style=quad] (9) to (14);
		\draw [style=quad] (6) to (13);
		\draw [style=quad] (10) to (14);
		\draw [style=quad] (5) to (8);
		\draw [style=quad] (7) to (9);
		\draw [style=quad, bend right=45, looseness=1.50] (14) to (11);
		\draw [style=quad] (9) to (12);
		\draw [style=quad] (14) to (11);
		\draw [style=quad] (13) to (8);
		\draw [blue, ultra thick, bend right=90, looseness=3.00, dashed] (0) to (1);
	\end{pgfonlayer}
\end{tikzpicture}
\caption{\label{fig:eye quadrangulation}From two quadrangulations $L\in\sQ_3$ and $R\in\sQ_8$ to a quadrangulation $L\cdot R\in\sQ_{12}$ (whose root is the red one on the right, while the marked blue oriented edge is forgotten). Notice that it is possible to recover $L$ and $R$ from $L\cdot R$, by splitting the two cycles which form the boundary of the face containing the root corner, erasing the dashed edges and rooting appropriately.}
\end{figure}
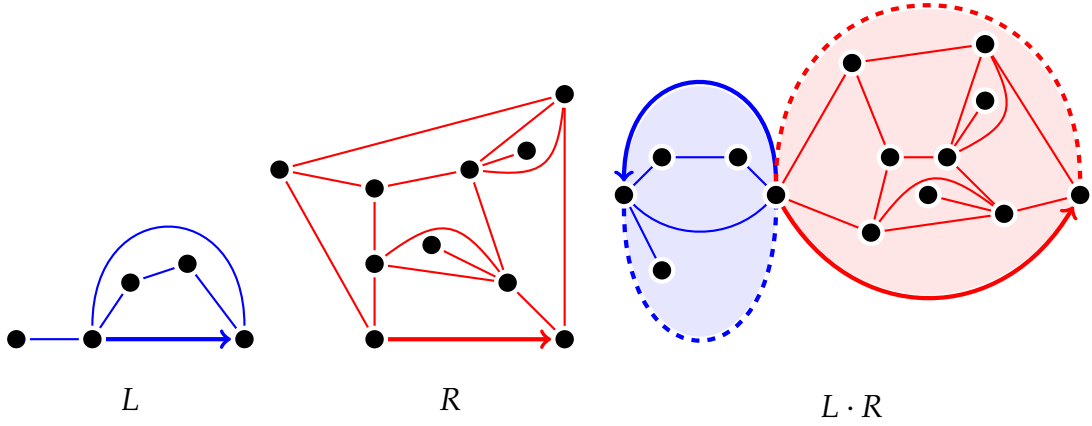

In order to describe the general structure of our canonical paths, it will be useful to introduce certain ``surgical operations'' that will enable us to assemble multiple quadrangulations into a single larger one. Given two quadrangulations $L\in\sQ_l$ and $R\in\sQ_r$, where $l,r\geq 1$, we shall write $L\cdot R$ for the quadrangulation in $\sQ_{l+r+1}$ obtained as follows (Figure~\ref{fig:eye quadrangulation}): first ``double'' the root edges of $L$ and $R$ by attaching a degree two face directly to their right; for convenience, draw this degree two face as the infinite face in the plane, so that  $L$ and $R$ are each ``enclosed'' within a cycle of length 2 containing the root edge; now draw both quadrangulation in the plane, identifying their origins, in such a way that both root edges are oriented clockwise (with respect to the infinite face); finally, forget the rooting of $L$ to obtain $L\cdot R$. It will be convenient to also consider the case where $l=0$ or $r=0$ (remember we have conventionally set $\sQ_0=\{\rightarrow\}$); we will set $\rightarrow\cdot q$, for any $q$ with $|q|\geq 1$, to be the quadrangulation obtained by adding a degenerate face directly to the right of the root edge of $q$ (equivalently, the operation described above is performed without actually doubling the root edge of $\rightarrow$). The quadrangulation $q\cdot\rightarrow$ is $\rightarrow\cdot q$, rerooted in the edge within the added degenerate face, so as not to change the origin.

We shall write $\sQ_l\cdot\sQ_r$ for the subset $\{L\cdot R \st (L,R)\in\sQ_l\times\sQ_r\}$ of $\sQ_{l+r+1}$. Notice that, given $q\in\sQ_l\cdot\sQ_r$ such that $q=L\cdot R$, one can quite simply reconstruct $L$ and $R$, since the rooting of $L$, which is the only information not trivially encoded, can still be recovered by following the contour of the face containing the root corner of $q$.

As previously described, the idea behind our canonical paths will be to ``destroy'' the starting quadrangulation $q$ on the right while ``growing'' a new quadrangulation on the left.

 Before dealing with the general case, we shall focus on the case where the ``final'' quadrangulation $q'$ is of the form $\tilde{q}\cdot\rightarrow$ for some $\tilde{q}\in\sQ_{n-1}$. Furthermore, we shall not yet build the full random canonical path from $q$ to $\tilde{q}\cdot\rightarrow$, but a random sequence of quadrangulations of the form $(L_i\cdot R_i)_{i=0}^{n-1}$, taking values in $\prod_{i=0}^{n-1}(\sQ_i\times\sQ_{n-i-1})$, that our random canonical path will ``go through''. Given this sequence, the path will actually be deterministic, as detailed within Sections~\ref{sec:from q to ->coll(q,c)} and~\ref{sec:completing psi}.

Given $q\in\sQ_n$ and $q'\cdot\rightarrow\in\sQ_n$, consider the probability distribution $\P_q^{q'}$ on the set $\prod_{i=0}^{n-1}\sQ_i\cdot\sQ_{n-i-1}$ defined as follows. Given $(L_i\cdot R_i)_{i=0}^{n-1}\in \prod_{i=0}^{n-1}\sQ_i\cdot\sQ_{n-i-1}$, set
$$\P_q^{q'}((L_i\cdot R_i)_{i=0}^{n-1})=g_n(q,R_0)1_{L_{n-1}=q'}\prod_{i=0}^{n-2} g_{n-i-1}(R_i,R_{i+1})g_{i+1}(L_{i+1},L_i).$$

It should be clear that $\P_q^{q'}$ is a probability distribution: a random sequence $(\lambda_i\cdot\rho_i)_{i=0}^{n-1}$ distributed according to $\P_q^{q'}$ is simply built in such a way that $\lambda_{n-1},\lambda_{n-2},\ldots,\lambda_0$ and $q,\rho_0,\ldots,\rho_{n-1}$ are independent sequences of random quadrangulations, started at $q'$ and $q$ respectively, built so as to collapse one random face according to the probability distribution given by $g_i(-,\cdot)$ at each step.

The key feature of the probability distribution $\P_q^{q'}$ which we will use to complete the necessary estimates on the congestion given by our random canonical paths is expressed in the following lemma:
\begin{lemma}\label{key estimate}
Given positive integers $n, a<n-1, b<n$ and quadrangulations $l\in\sQ_a,r\in\sQ_b$, we have
$$\sum_{q\in\sQ_n, q'\in\sQ_{n-1}}\P_q^{q'}\left(\{(L_i,R_i)_{i=0}^{n-1}\st L_a=l,R_{n-b-1}=r\}\right)\leq 12^{2n-b-a-1}.$$
\end{lemma}

\begin{proof}
The expression in the statement can be rewritten as
$$\left(\sum_{q\in\sQ_n}\sum_{\substack{(R_i)_{i=0}^{n-b-2}\in\prod_{i=0}^{n-b-2}\sQ_{n-i-1}\\R_{n-b-1}=r\\(R_i)_{i=n-b}^{n-1}\in\prod_{i=n-b}^{n-1}\sQ_{n-i-1}}}g_n(q,R_0)\prod_{i=0}^{n-2} g_{n-i-1}(R_i,R_{i+1})\right)%
\left(\sum_{q'\in\sQ_{n-1}}\sum_{\substack{(L_i)_{i=0}^{a-1}\in\prod_{i=0}^{a-1}\sQ_i\\L_a=l\\(L_i)_{i=a+1}^{n-1}\in\prod_{i=a+1}^{n-1}\sQ_i}}1_{L_{n-1}=q'}\prod_{i=0}^{n-2} g_{i+1}(L_{i+1},L_{i})\right).$$
Let us give an upper bound for the second factor above: the computations involved in bounding the first factor will be entirely similar.

By appropriately exchanging sums and products, we can rewrite it as
$$\sum_{(L_i)_{i=a+1}^{n-1}\in\prod_{i=a+1}^{n-1}\sQ_i}g_{a+1}(L_{a+1},l)\prod_{i=a+1}^{n-2}g_{i+1}(L_{i+1},L_{i})\sum_{(L_i)_{i=0}^{a-1}\in\prod_{i=0}^{a-1}\sQ_i}g_a(l,L_{a-1})\prod_{i=0}^{a-2}g_{i+1}(L_{i+1},L_i);$$
the entire internal sum is equal to 1 by property (ii) of the mappings $g_1,\ldots,g_a$; the external sum can thus be evaluated by using property (iii) of the mappings $g_{a+1},\ldots, g_{n-1}$ (and by summing over $L_{n-1},L_{n-2},\ldots,L_{a+1}$ separately, in turn). We obtain that the above is
$$\prod_{i=a}^{n-2} \frac{|\sQ_{i+1}|}{|\sQ_i|}\leq 12^{n-a-1},$$
where we have used the simple fact that, for all $i\geq 0$, $|\sQ_{i+1}|=3^{i+1}\operatorname{Cat}(i+1)\geq 3\cdot 4\cdot 3^i\operatorname{Cat}(i)=12|\sQ_i|$.

As for the first factor above, a similar argument yields that it is equal to $\prod_{i=b}^{n-1}\frac{|\sQ_{i+1}|}{|\sQ_i|}$, and therefore bounded above by $12^{n-b}$, which concludes the proof of the lemma.
\end{proof}

Consider now the general case of a pair of quadrangulations $q_1,q_2\in\sQ_n$: we are almost ready to construct our probability measure $\P_{q_1\rightarrow q_2}$ on the set of all possible paths $\Gamma_{q_1\rightarrow q_2}$. This will require three fundamental ingredients: one is the family of probability spaces $(\Gamma_q^{q'},\P_q^{q'})$ (for $q\in\sQ_n,q'\in\sQ_{n-1}$) we just built and discussed; one is a mapping $F:\sQ_n^2\rightarrow\sQ_{n-1}$, which we will use to assign to the pair $q_1,q_2$ the probability space $\left(\Gamma_{q_1}^{F(q_1,q_2)}\times\Gamma_{q_2}^{F(q_1,q_2)},\P_{q_1}^{F(q_1,q_2)}\otimes \P_{q_2}^{F(q_1,q_2)}\right)$; the last one is a mapping $\Psi_{q_1,q_2}:\Gamma_{q_1}^{F(q_1,q_2)}\times\Gamma_{q_2}^{F(q_1,q_2)}\rightarrow \Gamma_{q_1\rightarrow q_2}$, which will enable us to simply define $\P_{q_1\rightarrow q_2}$ as the push-forward via $\Psi_{q_1,q_2}$ of the probability measure $\P_{q_1}^{F(q_1,q_2)}\otimes \P_{q_2}^{F(q_1,q_2)}$.

For the mapping $F$, we may choose any which satisfies the condition that, given $q\in\sQ_n$ and $q'\in\sQ_{n-1}$, we have $\left|\left\{\tilde{q}\in\sQ_n\st F(q,\tilde{q})=q'\right\}\right|\leq 12$, and similarly $\left|\left\{\tilde{q}\in\sQ_n\st F(\tilde{q},q)=q'\right\}\right|\leq 12$. The fact that such a mapping exists is an immediate consequence of the fact that $|\sQ_n|\leq 12|\sQ_{n-1}|$: we shall from here on use $F:\sQ_n^2\rightarrow\sQ_{n-1}$ under the assumption that we have chosen one such mapping.

The next section will be devoted to the construction of a mapping $\Psi_{q_1,q_2}:\Gamma_{q_1}^{F(q_1,q_2)}\times\Gamma_{q_2}^{F(q_1,q_2)}\rightarrow\Gamma_{q_1\rightarrow q_2}$, which will consist in essentially ``interpolating'' sequences $(L^1_i\cdot R^1_i)_{i=0}^{n-1}\in\Gamma_{q_1}^{F(q_1,q_2)}$ and $(L^2_i\cdot R^2_i)_{i=0}^{n-1}\in\Gamma_{q_2}^{F(q_1,q_2)}$ by filling in the ``gap'' between successive quadrangulations via sequences of edge flips and making sure to run the complete flip sequence constructed from $(L^1_i\cdot R^1_i)_{i=0}^{n-1}$ forward, then the one constructed from $(L^2_i\cdot R^2_i)_{i=0}^{n-1}\in\Gamma_{q_2}^{F(q_1,q_2)}$ backwards. This needs to be done with some care: in particular, our aim is to be able to give an upper bound for the quantity
$$\sum_{q_1,q_2\in\sQ_n}\P_{q_1\rightarrow q_2}(\{\gamma\in\Gamma_{q_1\rightarrow q_2}\mbox{ containing } (q,e,s)\})$$
independent of the flip $(q,e,s)$ by invoking Lemma~\ref{key estimate}. Indeed, we wish to build $\Psi_{q_1,q_2}$ in such a way that knowing a flip $(q,e,s)$ appears in a path $\Psi_{q_1,q_2}((L^1_i\cdot R^1_i)_{i=0}^{n-1},(L^2_i\cdot R^2_i)_{i=0}^{n-1})$ gives as much information as possible about the actual quadrangulations $L_i^j,R_i^j$.

\subsection{The flip path from $q$ to $\rightarrow \cdot \coll{q,c}$}\label{sec:from q to ->coll(q,c)}

We now begin the task of constructing our mappings $\Psi_{q_1,q_2}$, for $q_1,q_2\in\sQ_n$. In order to do this, given $(L_i\cdot R_i)_{i=0}^{n-1}\in\Gamma_q^{q'}$ we wish to construct flip sequences leading from the quadrangulation $L_i\cdot R_i$ to the quadrangulation $L_{i+1}\cdot R_{i+1}$, plus a flip path from $q$ to $L_0\cdot R_0$. Notice that with probability 1 (according to $\P_q^{q'}$) the quadrangulation $R_{i+1}$ differs from $R_i$ by collapsing a face; the same is true for $L_i$ and $L_{i+1}$ and for $R_0$ and $q$. We may therefore assume this is the case when constructing $\Psi_{q_1,q_2}$.

First of all, we shall construct the very first part of the flip path, which will transform a quadrangulation $q$ into $L_0\cdot R_0$, where $|L_0|=0$ (hence $L_0=\rightarrow$) and $R_0$ is of the form $\coll{q,c}$ for some corner $c$ of $q$. Once this construction is made, all others will be rather straightforward generalisations of it.

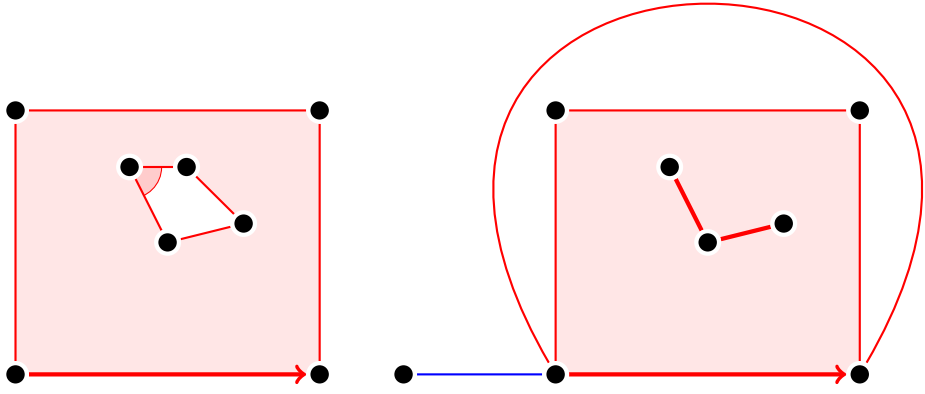
\begin{figure}\centering
\begin{tikzpicture}
\clip[use as bounding box] (-2,-2.5) rectangle (3,2);
	\begin{pgfonlayer}{nodelayer}
		\node [style=vertex] (0) at (-2, -2) {};
		\node [style=vertex] (1) at (2, -2) {};
		\node [style=vertex] (2) at (-2, 1.5) {};
		\node [style=vertex] (3) at (2, 1.5) {};
		\node [style=vertex] (4) at (-0.5, 0.75) {};
		\node [style=vertex] (5) at (0, -0.25) {};
		\node [style=vertex] (6) at (1, -0) {};
		\node [style=vertex] (7) at (0.25, 0.75) {};
	\end{pgfonlayer}
	\begin{pgfonlayer}{edgelayer}
	\filldraw[red!10] (-2,-2) rectangle (2,1.5);
		\begin{scope}
	\clip (4.center) to (5.center) to (6.center) to (7.center) to (4.center);
	\fill[white] (4.center) to (5.center) to (6.center) to (7.center) to (4.center);
	\draw[draw=red, fill=red!20] (4) circle (12pt);
	\end{scope}
		\draw [style=quad,ultra thick,->] (0) to (1);
		\draw [style=quad] (1) to (3);
		\draw [style=quad] (3) to (2);
		\draw [style=quad] (2) to (0);
		\draw [style=quad] (4) to (5);
		\draw [style=quad] (5) to (6);
		\draw [style=quad] (6) to (7);
		\draw [style=quad] (7) to (4);
	\end{pgfonlayer}
\end{tikzpicture}
\begin{tikzpicture}
\clip[use as bounding box] (-4,-2.5) rectangle (3,2);
	\begin{pgfonlayer}{nodelayer}
		\node [style=vertex] (0) at (-2, -2) {};
		\node [style=vertex] (1) at (2, -2) {};
		\node [style=vertex] (2) at (-2, 1.5) {};
		\node [style=vertex] (3) at (2, 1.5) {};
		\node [style=vertex] (4) at (-0.5, 0.75) {};
		\node [style=vertex] (5) at (0, -0.25) {};
		\node [style=vertex] (6) at (1, -0) {};
		\node [style=vertex] (7) at (-4, -2) {};
	\end{pgfonlayer}
	\begin{pgfonlayer}{edgelayer}
	\filldraw[red!10] (-2,-2) rectangle (2,1.5);
	\draw [quad, blue] (7) to (0);
		\draw [style=quad,ultra thick,->] (0) to (1);
		\draw [style=quad] (1) to (3);
		\draw [style=quad] (3) to (2);
		\draw [style=quad] (2) to (0);
		\draw [style=quad, ultra thick] (4) to (5);
		\draw [style=quad, ultra thick] (5) to (6);
		\draw [quad, out=120, in=60, looseness=4.5] (0) to (1);
	\end{pgfonlayer}
\end{tikzpicture}
\caption{\label{fig:q,->xcoll(q,c)}On the left, a quadrangulation $q\in\sQ_n$, drawn in the plane so that the infinite face lies directly to the right of the root edge, with a marked corner $c$ within a face $f$. To the right, the quadrangulation $\rightarrow\cdot\coll{q,c}$: the face $f$ is ``collapsed'' and a degenerate face is added directly to the right of the root edge. If $\rho$ is the root corner of the quadrangulation $q'$ drawn on the right, $\coll{q',\rho}$ is $\coll{q,c}$.}
\end{figure}

Hence our objective is this: given a quadrangulation $q\in\sQ_n$ and a corner $c$ of $q$, we shall build a unique canonical path that, through a sequence of edge flips, transforms $q$ into the quadrangulation $\rightarrow \cdot \coll{q,c}$ (see Figure~\ref{fig:q,->xcoll(q,c)} for a representation of a quadrangulation of the form $\rightarrow \cdot \coll{q,c}$).

We shall say that such a path has two phases: the first phase has the aim of replacing the face $f_c$ containing $c$ with a degenerate face in such a way that the appropriate vertices of $q$ are identified; the second phase consists in ``moving'' the degenerate face so that it ends up lying directly to the right of the root edge.  We shall first concern ourselves with the second phase, that is, build a canonical path $P(q,c)$ from $q$ to $\rightarrow\cdot \coll{q,c}$ in the case where $c$ is a corner within a degenerate face; note that the specific case where the internal edge of this face is the root edge of $q$ is a little different and will be dealt with separately.

\begin{figure}
\centering
\begin{tikzpicture}
	\begin{pgfonlayer}{nodelayer}
		\node [style=vertex, label=above:$v_i$,label=left:$4$] (0) at (0, 4) {};
		\node [style=vertex, label=right:$w_i$,label=left:$3$] (1) at (0, 2) {};
		\node at (0.25,3.4) {$\eta_i$};
		\node [style=vertex] (2) at (0, 3) {};
		\node [style=vertex,label=left:$2$] (3) at (0, 1) {};
		\node [style=vertex,label=left:$1$] (4) at (0, -0) {};
		\node [style=vertex,label=left:$0$, label=below:$\emptyset$] (5) at (0, -1) {};
		\node [style=vertex] (6) at (1, -1) {};
	\end{pgfonlayer}
	\begin{pgfonlayer}{edgelayer}
		\draw [thick] (5) to (1);
		\draw [very thick,->,red] (5) to (6);
		\draw [thick, bend right=60, looseness=1.00] (1) to (0);
		\draw [thick, bend left=60, looseness=1.00] (1) to (0);
		\draw [very thick, blue] (2) to (0);
	\end{pgfonlayer}
\end{tikzpicture}
\begin{tikzpicture}
	\begin{pgfonlayer}{nodelayer}
		\node [style=vertex, label=above:$w_{i+1}$,label=left:$4$] (0) at (0, 4) {};
		\node [style=vertex, label=right:$v_{i+1}$,label=left:$3$] (1) at (0, 2) {};
		\node at (0,2.5) {\contour{white}{$\eta_{i+1}$}};
		\node [style=vertex] (2) at (0, 3) {};
		\node [style=vertex,label=left:$2$] (3) at (0, 1) {};
		\node [style=vertex,label=left:$1$] (4) at (0, -0) {};
		\node [style=vertex,label=left:$0$, label=below:$\emptyset$] (5) at (0, -1) {};
		\node [style=vertex] (6) at (1, -1) {};
	\end{pgfonlayer}
	\begin{pgfonlayer}{edgelayer}
		\draw [thick] (5) to (1);
		\draw [very thick,->,red] (5) to (6);
		\draw [thick, bend right=60, looseness=1.00] (1) to (0);
		\draw [thick, bend left=60, looseness=1.00] (1) to (0);
		\draw [very thick, blue, dashed] (2) to (0);
		\draw [very thick, blue] (2) to (1);
	\end{pgfonlayer}
\end{tikzpicture}
\hspace{1cm}
\begin{tikzpicture}
	\begin{pgfonlayer}{nodelayer}
		\node [style=vertex, label=above:$v_{i}$,label=left:$4$] (0) at (0, 4) {};
		\node [style=vertex, label=right:$w_{i}$,label=left:\contour{white}{$3$}] (1) at (0, 2) {};
		\node at (-0.9,3) {\contour{white}{$\tilde{\eta}_{i}$}};
		\node [style=vertex] (2) at (0, 3) {};
		\node [style=vertex,label=left:$2$] (3) at (0, 1) {};
		\node [style=vertex,label=left:$1$] (4) at (0, -0) {};
		\node [style=vertex,label=left:$0$, label=below:$\emptyset$] (5) at (0, -1) {};
		\node [style=vertex] (6) at (1, -1) {};
		\node [style=vertex] (x) at (-1.5, 2.3) {};
	\end{pgfonlayer}
	\begin{pgfonlayer}{edgelayer}
		\draw [thick] (5) to (1);
		\draw [thick] (x) to (1);
		\draw [very thick,->,red] (5) to (6);
		\draw [thick, bend right=60, looseness=1.00] (1) to (0);
		\draw [blue, very thick, bend left=60, looseness=1.00] (1) to (0);
		\draw [thick] (2) to (1);
	\end{pgfonlayer}
\end{tikzpicture}
\begin{tikzpicture}
	\begin{pgfonlayer}{nodelayer}
		\node [style=vertex, label=above:$v_{i+1}$,label=left:$4$] (0) at (0, 4) {};
		\node [style=vertex, label=right:$w_{i+1}$,label=left:\contour{white}{$3$}] (1) at (0, 2) {};
		\node [style=vertex] (2) at (0, 3) {};
		\node [style=vertex,label=left:$2$] (3) at (0, 1) {};
		\node [style=vertex,label=left:$1$] (4) at (0, -0) {};
		\node [style=vertex,label=left:$0$, label=below:$\emptyset$] (5) at (0, -1) {};
		\node [style=vertex] (6) at (1, -1) {};
		\node [style=vertex] (x) at (-1.5, 2.3) {};
	\end{pgfonlayer}
	\begin{pgfonlayer}{edgelayer}
		\draw [thick] (5) to (1);
		\draw [thick] (x) to (1);
		\draw [very thick,->,red] (5) to (6);
		\draw [thick, bend right=60, looseness=1.00] (1) to (0);
		\draw [blue, very thick, bend left=60, looseness=1.00, dashed] (1) to (0);
		\draw [thick] (2) to (1);
		\draw [very thick, blue, out=90, in=50, looseness=3] (x) to (1);
	\end{pgfonlayer}
\end{tikzpicture}
\caption{\label{fig:move leafy face}}
\end{figure}

\begin{lemma}\label{lemma:P_1(q,c)}Let $c$ be a corner within a degenerate face $f$ of a quadrangulation $q\in\sQ_n$ and suppose the root edge of $q$ is \emph{not} the internal edge of $f$. Define the path $P(q,c)=(q_i,e_i,+)_{i=1}^N$ recursively as follows (it may be useful to refer to Figure~\ref{fig:move leafy face}):
\begin{itemize}
\item set $q_1=q$.
\item Let $f_1=f$ and $\eta_1$ be the internal edge of $f$; let $f_i$, for $i\geq 2$, be the face of $q_i$ that contains the (possibly flipped) image of edge $\eta_{i-1}$ in $q_i$, and $\eta_i$ the internal edge of $f_i$ (which will automatically be a degenerate face). Let $v_i$ be the vertex on the boundary of $f_i$ that is an endpoint of $\eta_i$, let $w_i$ be the other vertex on the external boundary of $f_i$ and let $\tilde{\eta}_i$ be the edge immediately after $\eta_i$ in counterclockwise order around vertex $v_i$.
\item If $\dgr(v_i,\emptyset)>\dgr(w_i,\emptyset)$ (where $\dgr$ is the graph distance on the vertex set of $q_i$), set $e_i=\eta_i$. If, on the other hand, $\dgr(v_i,\emptyset)<\dgr(w_i,\emptyset)$, set $e_i=\tilde{\eta}_i$.
\item Set $q_{i+1}=q_i^{e_i,+}$.
\item Set $N$ to be the first non-negative integer for which $q_N^{e_N,s_N}$ is the quadrangulation $\rightarrow\cdot\coll{q,c}$.
\end{itemize}

The path above is well defined, in the sense that $f_i$ is always degenerate (so that the construction can be performed), $e_i$ is never the root edge of $f$ and $N$ is a positive integer.

Furthermore, we have $|P(q,c)|=N\leq 6n$ and, for $i=1,\ldots,N$, we have $\coll{q_i,c_i}=\coll{q,c}$, where $c_i$ is the corner corresponding to the edge $e_i$, oriented towards $v_i$.

\end{lemma}

\begin{proof}
The fact that $f_i$ is degenerate is easily shown by induction. Indeed, flipping $\eta_i$ does not change the fact that it is an internal edge in a degenerate face. On the other hand, suppose $\dgr(v_i,\emptyset)<\dgr(w_i,\emptyset)$ and let $u_i$ be the endpoint of $\eta_i$ that is different from $v_i$. Then flipping $\tilde{\eta}_i$ \emph{clockwise} does not increase the degree of $u_i$, so that $\eta_i$ remains within a degenerate face in $q_{i+1}$.

Now, since $\eta_1$ is not the root edge of $q$, the edge $\eta_i$ (which is the image of $\eta_1$ after multiple flips in the path) cannot at any point be the root edge. On the other hand, if the root edge were $\tilde{\eta}_i$ and we had $\dgr(v_i,\emptyset)<\dgr(w_i,\emptyset)$, hence $v_i=\emptyset$, we would actually have $q_i=\rightarrow\cdot \coll{q,c}$.

The fact that $N$ is finite can be seen as a consequence of the fact that $\dgr(v_i,\emptyset)$ is weakly decreasing (since it is not increased by the flip of $\tilde{\eta}_i$ and is decreased when flipping $\eta_i$). After we have $v_i=\emptyset$, flipping $\tilde{\eta_i}$ repeatedly will eventually make $f_i$ the face immediately to the right of the root edge, yielding exactly the quadrangulation $\rightarrow\cdot\coll{q,c}$.

Let us now check the bound on $N$. Consider a step $(q_i,e_i,+)$ in the path, where $e_i\neq \eta_i$ and $e_{i-1}\neq \eta_{i-1}$; the edge $e_i$, which is then $\tilde{\eta}_i$, has never been flipped before (i.e.~it is not the image in $q_i$ of any $e_j$ for $j<i$). On the other hand, $\{i\leq N \st e_i=\eta_i\}\leq \dgr^q(v_0,\emptyset)\leq 2n$, hence the bound.

Finally, $c_i$ is a corner of the degenerate face $f_i$ (since $f_i$ contains $\eta_i$ and lies directly to the right to the right of $\tilde{\eta}_i$), and we can show that $\coll{q_i,c_i}=\coll{q_{i+1},c_{i+1}}$. This is obvious if $e_i=\eta_i$; if $e_i=\tilde{\eta}_i$, the quadrangulation $q_{i+1}$ differs by $q_i$ only by the fact that the degenerate face $f_i$ is ``rotated'' onto the edge after $\tilde{\eta}_i$ in counterclockwise order around $w_i$, then labelled $f_{i+1}$: collapsing it after the procedure will still yield $\coll{q_i,c_i}$.
\end{proof}

We shall then perform an ad hoc construction in the case  where the root edge is the internal edge within the degenerate face of $q$ containing $c$:

\begin{figure}\centering
\begin{tikzpicture}
\clip[use as bounding box] (-1.5,-0.8) rectangle (2.3,2);
	\begin{pgfonlayer}{nodelayer}
		\node [style=vertex, label=left:$v$] (0) at (-1, 1) {};
		\node [style=vertex, label=right:$u$] (1) at (0, 1) {};
		\node [style=vertex, label=above right:$w$] (2) at (1, 1) {};
		\node [style=vertex] (3) at (2, 1) {};
		\node at (0, 1.8) {\contour{white}{$e_3$}};
		\node at (1.5, 0.8) {\contour{white}{$e_2$}};
		\node at (0, 0.2) {\contour{white}{$e_1$}};
	\end{pgfonlayer}
	\begin{pgfonlayer}{edgelayer}
		\draw [very thick, red, <-] (0) to (1);
		\draw [thick, bend left=60, looseness=1.00] (0) to (2);
		\draw [thick, bend right=60, looseness=1.00] (0) to (2);
		\draw [thick] (2) to (3);
	\end{pgfonlayer}
\end{tikzpicture}	
\begin{tikzpicture}
\clip[use as bounding box] (-1.5,-0.8) rectangle (2.3,2);
	\begin{pgfonlayer}{nodelayer}
		\node [style=vertex, label=left:$v$] (0) at (-1, 1) {};
		\node [style=vertex, label=right:$u$] (1) at (0, 1) {};
		\node [style=vertex, label=above right:$w$] (2) at (1, 1) {};
		\node [style=vertex] (3) at (2, 1) {};
		\node at (0, 1.8) {\contour{white}{$e_3$}};
		\node at (1.5, 0.8) {\contour{white}{$e_2$}};
		\node at (0, 0.2) {\contour{white}{$e_1$}};
	\end{pgfonlayer}
	\begin{pgfonlayer}{edgelayer}
		\draw [very thick, red, <-] (0) to (1);
		\draw [thick, bend left=60, looseness=1.00] (0) to (2);
		\draw [very thick, bend right=60, looseness=1.00, blue, dashed] (0) to (2);
		\draw [very thick, bend right=60, looseness=1.00, blue] (1) to (3);
		\draw [thick] (2) to (3);
	\end{pgfonlayer}
\end{tikzpicture}
\begin{tikzpicture}
\clip[use as bounding box] (-1.5,-0.8) rectangle (2.3,2);
	\begin{pgfonlayer}{nodelayer}
		\node [style=vertex, label=left:$v$] (0) at (-1, 1) {};
		\node [style=vertex, label=right:$u$] (1) at (0, 1) {};
		\node [style=vertex, label=above right:$w$] (2) at (1, 1) {};
		\node [style=vertex] (3) at (2, 1) {};
		\node at (0, 1.8) {\contour{white}{$e_3$}};
		\node at (1.5, 0.8) {\contour{white}{$e_2$}};
	\end{pgfonlayer}
	\begin{pgfonlayer}{edgelayer}
		\draw [very thick, red, <-] (0) to (1);
		\draw [thick, bend left=60, looseness=1.00] (0) to (2);
		\draw [thick, bend right=60, looseness=1.00] (1) to (3);
		\draw [very thick, blue, dashed] (2) to (3);
		\draw[very thick, blue, out=-20, in=60, looseness=8] (1) to (0);
	\end{pgfonlayer}
\end{tikzpicture}	
\begin{tikzpicture}
\clip[use as bounding box] (-1.5,-0.8) rectangle (2.3,2);
	\begin{pgfonlayer}{nodelayer}
		\node [style=vertex, label=left:$v$] (0) at (-1, 1) {};
		\node [style=vertex, label=right:$u$] (1) at (0, 1) {};
		\node [style=vertex, label=above right:$w$] (2) at (1, 1) {};
		\node [style=vertex] (3) at (2, 1) {};
		\node at (0, 1.8) {\contour{white}{$e_3$}};
	\end{pgfonlayer}
	\begin{pgfonlayer}{edgelayer}
		\draw [very thick, red, <-] (0) to (1);
		\draw [bend left=60, looseness=1.00, blue, very thick, dashed] (0) to (2);
		\draw [bend left=60, looseness=1.00, blue, very thick] (1) to (2);
		\draw [thick, bend right=60, looseness=1.00] (1) to (3);
		\draw[thick, out=-20, in=60, looseness=8] (1) to (0);
	\end{pgfonlayer}
\end{tikzpicture}\\
\begin{tikzpicture}
\clip[use as bounding box] (-1.5,-0.8) rectangle (2.8,2);
	\begin{pgfonlayer}{nodelayer}
		\node [style=vertex, label=left:$v$] (0) at (-1, 1) {};
		\node [style=vertex, label=right:$u$] (1) at (0, 1) {};
		\node [style=vertex, label=above right:$w$] (2) at (1, 1) {};
		\node [style=vertex] (3) at (2.5, 1) {};
		\node at (0, 1.8) {\contour{white}{$e_1$}};
		\node at (2, 0.8) {\contour{white}{$e_2$}};
		\node at (0, 0.2) {\contour{white}{$e_3$}};
	\end{pgfonlayer}
	\begin{pgfonlayer}{edgelayer}
		\draw [very thick, red, ->] (0) to (1);
		\draw [thick, bend left=60, looseness=1.00] (0) to (2);
		\draw [thick, bend right=60, looseness=1.00] (0) to (2);
		\draw [thick] (2) to (3);
	\end{pgfonlayer}
\end{tikzpicture}	
\begin{tikzpicture}
\clip[use as bounding box] (-1.5,-0.8) rectangle (2.8,2);
	\begin{pgfonlayer}{nodelayer}
		\node [style=vertex, label=left:$v$] (0) at (-1, 1) {};
		\node [style=vertex, label=right:\contour{white}{$u$}] (1) at (0, 1) {};
		\node [style=vertex, label=above right:$w$] (2) at (1, 1) {};
		\node [style=vertex] (3) at (2.5, 1) {};
		\node at (0, 1.8) {\contour{white}{$e_1$}};
		\node at (2, 0.8) {\contour{white}{$e_2$}};
		\node at (0, 0.2) {\contour{white}{$e_3$}};
	\end{pgfonlayer}
	\begin{pgfonlayer}{edgelayer}
		\draw [very thick, red, ->] (0) to (1);
		\draw [very thick, blue, bend left=60, looseness=1.00, dashed] (0) to (2);
		\draw [very thick, blue, bend left=60] (1) to (3);
		\draw [thick, bend right=60, looseness=1.00] (0) to (2);
		\draw [thick] (2) to (3);
	\end{pgfonlayer}
\end{tikzpicture}	
\begin{tikzpicture}
\clip[use as bounding box] (-1.5,-0.8) rectangle (2.8,2);
	\begin{pgfonlayer}{nodelayer}
		\node [style=vertex, label=left:$v$] (0) at (-1, 1) {};
		\node [style=vertex, label=right:\contour{white}{$u$}] (1) at (0, 1) {};
		\node [style=vertex, label=above right:\contour{white}{$w$}] (2) at (1, 1) {};
		\node [style=vertex] (3) at (2.5, 1) {};
		\node at (2, 0.8) {\contour{white}{$e_2$}};
		\node at (0, 0.2) {\contour{white}{$e_3$}};
	\end{pgfonlayer}
	\begin{pgfonlayer}{edgelayer}
		\draw [very thick, red, ->] (0) to (1);
		\draw [thick, bend left=60] (1) to (3);
		\draw [thick, bend right=60, looseness=1.00] (0) to (2);
		\draw [very thick, blue, dashed] (2) to (3);
		\draw[very thick, blue, out=20, in=-60, looseness=8] (1) to (0);
	\end{pgfonlayer}
	\end{tikzpicture}
\caption{\label{fig:internal root}}
\end{figure}
\begin{lemma}\label{lemma:annoying root}
Let $q\in\sQ_n$ be a quadrangulation whose root edge $\rho$ is the internal edge within a degenerate face $f$ and let $c$ be a corner within $f$. Let $u$ be the degree one endpoint of $\rho$, let $v$ be its other endpoint and let $w$ be the third vertex adjacent to $f$. 

If $u$ is the origin of $q$, let $e_1,\ldots, e_{\deg(w)}$ be the edges incident to $w$, in counterclockwise order, indexed in such a way that $e_1$ and $e_{\deg(v)}$ are the boundary edges of $q$. Set $P(q,c)=(q_i,e_i,+)_{i=1}^N$, where $q_1=q$ and $q_{i+1}=q_i^{e_i,+}$ for $i=1,\ldots,\deg(w)-1=N$ (Figure~\ref{fig:internal root}, above).

If $u$ is \emph{not} the origin of $q$ (hence $v$ is), let  $e_1,\ldots, e_{\deg(w)}$ be the edges incident to $w$, in \emph{clockwise} order, indexed in such a way that $e_1$ and $e_{\deg(w )}$ are the boundary edges of $q$. Set $P(q,c)=(q_i,e_i,-)_{i=1}^N$, where $q_1=q$ and $q_{i+1}=q_i^{e_i,+}$ for $i=1,\ldots,\deg(w)-2=N$ (Figure~\ref{fig:internal root}, below).

We then have $q_N^{e_N,s_N}=\rightarrow\cdot \coll{q,c}$, where $s_N=+$ in the first case and $s_N=-$ in the second. Notice that in any case we have $N< 2n$.

In the first case, let $c_i$ be the corner corresponding to edge $e_i$ in $q_i$, oriented away from $w$; in the second, let $c_i$ be the corner corresponding to $e_i$ in $q_i$, oriented towards $w$. In both cases, we have $\coll{q_i,c_i}=\coll{q,c}$.
\end{lemma}

\begin{proof}
Notice that the root edge $\rho$ does not have $w$ as an endpoint, hence all flips we perform are allowed, and that $N\leq\deg(w)<2n$.

Also remark that the quadrangulation $\rightarrow\cdot\coll{q,v}$ can be obtained from $q$ by ``detaching'' the edges $e_2,\ldots,e_{\deg(w)-1}$ from $w$ and rerouting them to $u$, replacing $e_1$ with an edge joining $u$ to $v$ in such a way as to create a face containing $w$ (which has now degree 1) directly on the right of the root edge, and finally replacing $e_{\deg{w}}$ with an edge between $w$ and $u$ in the case where $u$ is the origin of $q$.

But indeed, this is exactly the effect achieved by the sequence of flips given: when flipping $e_i$ we are erasing it in favour of an edge that is a version of $e_{i+1}$ rerouted towards $u$ rather than $w$. The flip of $e_{\deg(w)-1}$ creates an edge between $u$ and $v$ enclosing $w$ within a degenerate face, and flipping $e_{\deg(w)}$ in the case where $u$ is the origin ensures that the degree 1 vertex $w$ is a neighbour of the origin (see Figure~\ref{fig:internal root}).

Indeed, one can identify $u$ and $w$ in $q_i$ by collapsing the face lying directly to the right of the root and obtain the quadrangulation $\coll{q,c}$; the corner $c_i$ is defined in such a way that this is exactly the effect of taking $\coll{q,c_i}$.
\end{proof}

We will now construct a path of flips from $q$ to $\rightarrow\cdot\coll{q,c}$ in the case where $c$ is a corner within a non-degenerate face.

\begin{lemma}\label{lemma:complete path}
Let $c$ be a corner within a non-degenerate face $f$ of a quadrangulation $q\in\sQ_n$; if $f$ has four distinct vertices then let $c=c_1,c_2,c_3,c_4$ be a clockwise contour of $f$, and let $v_1,v_2,v_3,v_4$ be the corresponding vertices. If $f$ has three distinct vertices, then let $c_1,c_2,c_3,c_4$ be a clockwise contour of $f$ such that $c_1$ and $c_4$ are adjacent to the same vertex $v_1$, and let $v_2, v_3$ be the vertices of $c_2,c_3$.
If the root edge of $q$ has $v_2$ as an endpoint, let $w=v_4$; otherwise, let $w=v_2$.

Let $e_1,e_2,\ldots, e_{\deg(w)}$ be the edges adjacent to $w$, in clockwise order, indexed in such a way that $e_1$ and $e_{\deg(w)}$ are on the boundary of $f$. Set $P_1(q,c)=(q_i,e_i,-)_{i=1}^{\deg(w)-1}$, where $q_1=q$ and $q_{i+1}=q_i^{e_i,-}$ for $i=1,\ldots,\deg(w)$. Now set $P_2(q,c)=P(q_{\deg(w)},c')$, where $c'$ is any corner of the face containing the edge $e_{\deg(w)}$ in $q_{\deg(w)}$ (which is a degenerate face). Set $P(q,c)$ to be the concatenation of $P_1(q,c)$ and $P_2(q,c)$. Then $P(q,c)=(q_i,e_i,s_i)_{i=1}^{N}$ is well defined and we have $q_N^{e_N,s_N}=\rightarrow\cdot \coll{q,c}$.

Moreover, we have $N<8n$ and, setting $c'_i$ to be the corner corresponding to the edge $e_i$ in $q_i$, oriented towards $w$ for $i=1,\ldots,\deg w$, and oriented towards the vertex $v_i$ from the construction of Lemma~\ref{lemma:P_1(q,c)} for $N\geq i>\deg w$, we have $\coll{q	_i,c'_i}=\coll{q,c}$.
\end{lemma}


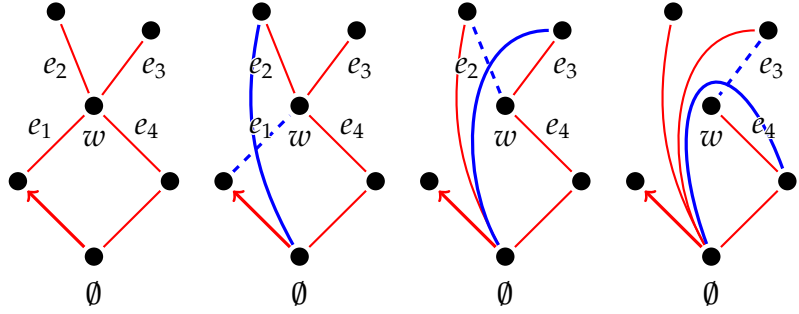
\begin{figure}\centering
\begin{tikzpicture}
\clip[use as bounding box] (-1.3,-1.7) rectangle (1.3,2.5);
	\begin{pgfonlayer}{nodelayer}
	\begin{scope}[yscale=1,xscale=-1]
		\node [style=vertex] (0) at (-1, -0) {};
		\node [style=vertex, label=below:$w$] (1) at (0, 1) {};
		\node [style=vertex] (2) at (1, -0) {};
		\node [style=vertex,label=below:$\emptyset$] (3) at (0, -1) {};
		\node [style=vertex] (4) at (-0.75, 2) {};
		\node [style=vertex] (5) at (0.5, 2.25) {};
		\node at (0.7,0.7) {$e_1$};
		\node at (0.5,1.5) {$e_2$};
		\node at (-0.8,1.5) {$e_3$};
		\node at (-0.7,0.7) {$e_4$};
		\end{scope}
	\end{pgfonlayer}
	\begin{pgfonlayer}{edgelayer}
		\draw[quad] (4) to (1);
		\draw[quad] (1) to (5);
		\draw[quad] (1) to (0);
		\draw[quad] (0) to (3);
		\draw[quad,very thick,->] (3) to (2);
		\draw[quad] (2) to (1);
	\end{pgfonlayer}
\end{tikzpicture}
\begin{tikzpicture}
\clip[use as bounding box] (-1.3,-1.7) rectangle (1.3,2.5);
	\begin{pgfonlayer}{nodelayer}
	\begin{scope}[yscale=1,xscale=-1]
		\node [style=vertex] (0) at (-1, -0) {};
		\node [style=vertex, label=below:$w$] (1) at (0, 1) {};
		\node [style=vertex] (2) at (1, -0) {};
		\node [style=vertex,label=below:$\emptyset$] (3) at (0, -1) {};
		\node [style=vertex] (4) at (-0.75, 2) {};
		\node [style=vertex] (5) at (0.5, 2.25) {};
		\node at (0.5,0.7) {\contour{white}{$e_1$}};
		\node at (0.5,1.5) {\contour{white}{$e_2$}};
		\node at (-0.8,1.5) {$e_3$};
		\node at (-0.7,0.7) {$e_4$};
		\end{scope}
	\end{pgfonlayer}
	\begin{pgfonlayer}{edgelayer}
		\draw[quad] (4) to (1);
		\draw[quad] (1) to (5);
		\draw[quad] (1) to (0);
		\draw[quad] (0) to (3);
		\draw[quad,very thick,->] (3) to (2);
		\draw[very thick, dashed, blue] (2) to (1);
		\draw[very thick, blue, bend right=20] (5) to (3);
	\end{pgfonlayer}
\end{tikzpicture}	
\begin{tikzpicture}
\clip[use as bounding box] (-1.3,-1.7) rectangle (1.3,2.5);
	\begin{pgfonlayer}{nodelayer}
	\begin{scope}[yscale=1,xscale=-1]
		\node [style=vertex] (0) at (-1, -0) {};
		\node [style=vertex, label=below:$w$] (1) at (0, 1) {};
		\node [style=vertex] (2) at (1, -0) {};
		\node [style=vertex,label=below:$\emptyset$] (3) at (0, -1) {};
		\node [style=vertex] (4) at (-0.75, 2) {};
		\node [style=vertex] (5) at (0.5, 2.25) {};
		\node at (0.5,1.5) {\contour{white}{$e_2$}};
		\node at (-0.8,1.5) {$e_3$};
		\node at (-0.7,0.7) {$e_4$};
		\end{scope}
	\end{pgfonlayer}
	\begin{pgfonlayer}{edgelayer}
		\draw[quad] (4) to (1);
		\draw[very thick, blue, dashed] (1) to (5);
		\draw[quad] (1) to (0);
		\draw[quad] (0) to (3);
		\draw[quad,very thick,->] (3) to (2);
		\draw[quad, bend right=20] (5) to (3);
		\draw[very thick, blue, out=180, in=120] (4) to (3);
	\end{pgfonlayer}
\end{tikzpicture}	
\begin{tikzpicture}
\clip[use as bounding box] (-1.3,-1.7) rectangle (1.3,2.5);
	\begin{pgfonlayer}{nodelayer}
	\begin{scope}[yscale=1,xscale=-1]
		\node [style=vertex] (0) at (-1, -0) {};
		\node [style=vertex, label=below:$w$] (1) at (0, 1) {};
		\node [style=vertex] (2) at (1, -0) {};
		\node [style=vertex,label=below:$\emptyset$] (3) at (0, -1) {};
		\node [style=vertex] (4) at (-0.75, 2) {};
		\node [style=vertex] (5) at (0.5, 2.25) {};
		\node at (-0.8,1.5) {$e_3$};
		\node at (-0.7,0.7) {\contour{white}{$e_4$}};
		\end{scope}
	\end{pgfonlayer}
	\begin{pgfonlayer}{edgelayer}
		\draw[very thick, blue, dashed] (4) to (1);
		\draw[quad] (1) to (0);
		\draw[quad] (0) to (3);
		\draw[quad,very thick,->] (3) to (2);
		\draw[quad, bend right=20] (5) to (3);
		\draw[quad, out=180, in=120] (4) to (3);
		\draw[very thick, blue, out=110, in=110, looseness=4] (0) to (3);
	\end{pgfonlayer}
\end{tikzpicture}	
\caption{\label{fig:face collapse}The quadrangulations $q_1,q_2,\ldots,q_{\deg w}$ in the path $P(q,c)$, where $c$ is a corner not belonging to a degenerate face.}
\end{figure}

\begin{proof}
First of all, notice that the root edge does not appear in $\{e_1,e_2,\ldots,e_{\deg w}\}$ so that all of the first $\deg w-1$ flips are ``allowed'': if it did, given our choice of $w$ then the root edge would have both $v_2$ and $v_4$ as endpoints; but this would create a cycle of length $3$ in the quadrangulation $q$, which is bipartite.

We can show inductively that $\coll{q_i,c'_i}=\coll{q,c}$ for $i=2,\ldots,\deg w-1$.

Let $e_1,e_{\deg_w},\eta,\eta'$ be the edges forming the boundary of $f$ in $q$, named in clockwise order. The quadrangulation $\coll{q,c}$ is obtained by identifying $e_1$ with $\eta'$ and $e_{\deg w}$ with $\eta$, thus collapsing $f$; equivalently, it is obtained by first erasing either $e_1$ or $\eta'$ (i.e.~an edge adjacent to $c$ in $f$), and identifying $e_{\deg w}$ with $\eta$ (the edges opposite $c$).

Consider now the quadrangulation $q_2=q^{e_1,-}$; the clockwise boundary of the face lying to the left of the flipped oriented edge $e_1$ is formed by $e_2,e_{\deg w},\eta,-e_1$, with $e_1$ and $e_2$ being adjacent to $c'_2$. We thus have that $\coll{q_2,c'_2}$ can be obtained by first erasing $e_1$, then identifying $e_{\deg w}$ with $\eta$. But, since the map obtained from $q$ by erasing $e_1$ and the map obtained from $q_2$ by erasing the flipped $e_1$ are exactly the same (with all labels assigned to objects in the same way), it follows that $\coll{q_2,c'_2}=\coll{q,c}$.

The argument can be repeated to show that, for $i\leq\deg w$, $\coll{q_i,c'_i}=\coll{q_{i-1},c'_{i-1}}$ (because the two are obtained in the same way from the coinciding maps created by erasing $e_{i-1}$ from $q_{i-1}$ and $q_i$).

Consider now the quadrangulation $q_{\deg w}$; in it, the degree of $w$ is $1$, and therefore $e_w$ is the internal edge of a degenerate face $f'$ (and is not the root edge). Moreover, collapsing $f'$ yields $\coll{q,c}$. We can thus invoke Lemma~\ref{lemma:P_1(q,c)}, which tells us that $P(q_{\deg w},c')$, where $c'$ is a corner of $f'$, is a flip path of length at most $6n$ ending with $\rightarrow\cdot \coll{q_{\deg w},c'}=\rightarrow\cdot \coll{q,c}$, and that $\coll{q_i,c'_i}=\coll{q_{\deg w}, c'}=\coll{q,c}$ for $i>\deg w$.

The estimate for $|P(q,c)|$ follows from the fact that $\deg w < 2n=|E(q)|$.
\end{proof}

Via the three lemmas above, for all pairs $(q,c)$, where $q\in\sQ_n$ and $c$ is a corner of $q$, we have constructed a canonical path $P(q,c)=(q_i,e_i,s_i)_{i=1}^N$ such that $q_N^{e_N,s_N}=\rightarrow\cdot\coll{q,c}$. The crucial property of these canonical paths is highlighted by the corollary below:

\begin{cor}\label{cor:few corners}
Consider any triple $(q,e,s)$, where $q\in\sQ_n$, $e$ is an edge of $q$ other than the root edge and $s=\pm$. Suppose $(q,e,s)$ appears in the sequence $P(q',c')$ for some $q'\in\sQ_n$ and some corner $c'$ of $q'$. Let $c_1,c_2$ be the corners of $q$ that correspond to the two possible orientations of $e$; we have $\coll{q',c'}\in\{\coll{q,c_1},\coll{q,c_2}\}$.\end{cor}



\subsection{Completing the description of $\Psi_{q_1,q_2}$}\label{sec:completing psi}

\begin{figure}[t]\centering
\begin{tikzpicture}[scale=.5]
	\begin{pgfonlayer}{nodelayer}
		\node [style=vertex] (0) at (-3, -0) {};
		\node [style=vertex] (1) at (2, -0) {};
		\node [style=vertex] (2) at (-6, -0) {};
		\node [style=vertex] (3) at (-1.5, 1.25) {};
		\node [style=vertex] (4) at (1, 1.5) {};
		\node [style=vertex] (5) at (-2, 0.25) {};
		\node [style=vertex] (6) at (0.75, -1.5) {};
		\node [style=vertex] (7) at (-0.5, 0.75) {};
		\node [style=vertex] (8) at (0.25, 1) {};
		\node [style=vertex] (9) at (-5, 0.5) {};
		\node [style=vertex] (10) at (-4, 0.5) {};
		\node [style=vertex] (11) at (-5, -0.5) {};
		\node [style=vertex] (12) at (-4, -0.5) {};
		\node (x) at (-0.25, 3.2) {\small\contour{white}{$\eta$}};
		
	\end{pgfonlayer}
	\begin{pgfonlayer}{edgelayer}
	\begin{scope}
	\clip (3.center) [bend right=75, looseness=1.75] to (4.center) [bend right=30, looseness=1.00] to (3.center);
	\draw[draw=red, fill=red!20] (3) circle (20pt);
	\end{scope}
		\draw [bend left=90, looseness=1.75, very thick, red] (0) to (1);
		\draw [bend right=90, looseness=1.75, very thick, red, ->] (0) to (1);
		\draw [bend left=90, looseness=1.50, very thick, blue, <-] (2) to (0);
		\draw [bend right=90, looseness=1.50, very thick, blue] (2) to (0);
		\draw [bend right=60, looseness=1.50] (0) to (3);
		\draw [bend right=75, looseness=1.75] (3) to (4);
		\draw (4) to (1);
		\draw (0) to (3);
		\draw (5) to (0);
		\draw (4) to (6);
		\draw [bend right, looseness=1.00] (0) to (6);
		\draw (8) to (4);
		\draw [bend left, looseness=1.00] (7) to (4);
		\draw [bend right=45, looseness=1.00] (7) to (4);
		\draw [bend left, looseness=1.00] (3) to (4);
		\draw (2) to (9);
		\draw (9) to (10);
		\draw (10) to (0);
		\draw (2) to (11);
		\draw (11) to (12);
		\draw (12) to (0);
		\draw (2) to (0); 
	\end{pgfonlayer}
\end{tikzpicture}\quad
\begin{tikzpicture}[scale=.5]
	\begin{pgfonlayer}{nodelayer}
		\node [style=vertex] (0) at (-3, -0) {};
		\node [style=vertex] (1) at (2, -0) {};
		\node [style=vertex] (2) at (-6, -0) {};
		\node [style=vertex] (3) at (-1.5, 1.25) {};
		\node [style=vertex] (4) at (1, 1.5) {};
		\node [style=vertex] (5) at (-2, 0.25) {};
		\node [style=vertex] (6) at (0.75, -1.5) {};
		\node [style=vertex] (7) at (-2, 2.45) {};
		\node [style=vertex] (8) at (-0.25, 1) {};
		\node [style=vertex] (9) at (-5, 0.5) {};
		\node [style=vertex] (10) at (-4, 0.5) {};
		\node [style=vertex] (11) at (-5, -0.5) {};
		\node [style=vertex] (12) at (-4, -0.5) {};
		\node (x) at (-0.25, 3.7) {\small\contour{white}{$\eta$}};
		\node (x) at (-2.5, 1.5) {\small\contour{white}{$\eta'$}};
	\end{pgfonlayer}
	\begin{pgfonlayer}{edgelayer}
		\draw [bend left=90, looseness=2.2, very thick, red] (0) to (1);
		\draw [bend right=90, looseness=1.75, very thick, red, ->] (0) to (1);
		\draw [bend left=90, looseness=1.50, very thick, blue, <-] (2) to (0);
		\draw [bend right=90, looseness=1.50, very thick, blue] (2) to (0);
		\draw [bend right=60, looseness=1.50] (0) to (3);
		\draw [bend right=75, looseness=1.75] (3) to (4);
		\draw (4) to (1);
		\draw (0) to (3);
		\draw (5) to (0);
		\draw (4) to (6);
		\draw [bend right, looseness=1.00] (0) to (6);
		\draw (8) to (4);
		\draw [bend left, looseness=1.00] (3) to (4);
		\draw (2) to (9);
		\draw (9) to (10);
		\draw (10) to (0);
		\draw (2) to (11);
		\draw (11) to (12);
		\draw (12) to (0);
		\draw (2) to (0); 
		\draw [bend left=80, looseness=1.5] (0) to (1);
		\draw [out=85, in=220, looseness=0.8] (0) to (7);
	\end{pgfonlayer}
\end{tikzpicture}
\begin{tikzpicture}[scale=.5]
	\begin{pgfonlayer}{nodelayer}
		\node [style=vertex] (0) at (-3, -0) {};
		\node [style=vertex] (1) at (2, -0) {};
		\node [style=vertex] (2) at (-6, -0) {};
		\node [style=vertex] (3) at (-1.5, 1.25) {};
		\node [style=vertex] (4) at (1, 1.5) {};
		\node [style=vertex] (5) at (-2, 0.25) {};
		\node [style=vertex] (6) at (0.75, -1.5) {};
		\node [style=vertex] (7) at (-4.5, 2.1) {};
		\node [style=vertex] (8) at (-0.25, 1) {};
		\node [style=vertex] (9) at (-5, 0.5) {};
		\node [style=vertex] (10) at (-4, 0.5) {};
		\node [style=vertex] (11) at (-5, -0.5) {};
		\node [style=vertex] (12) at (-4, -0.5) {};
		\node (x) at (-4.5, 3) {\small\contour{white}{$\eta$}};
		\node (x) at (-3.8, 2) {\small\contour{white}{$\eta'$}};
	\end{pgfonlayer}
	\begin{pgfonlayer}{edgelayer}
		\draw [bend right=90, looseness=1.75, very thick, red, ->] (0) to (1);
		\draw [bend left=90, looseness=1.50] (2) to (0);
		\draw [bend right=90, looseness=1.50, very thick, blue] (2) to (0);
		\draw [bend right=60, looseness=1.50] (0) to (3);
		\draw [bend right=75, looseness=1.75] (3) to (4);
		\draw (4) to (1);
		\draw (0) to (3);
		\draw (5) to (0);
		\draw (4) to (6);
		\draw [bend right, looseness=1.00] (0) to (6);
		\draw (8) to (4);
		\draw [bend left, looseness=1.00] (3) to (4);
		\draw (2) to (9);
		\draw (9) to (10);
		\draw (10) to (0);
		\draw (2) to (11);
		\draw (11) to (12);
		\draw (12) to (0);
		\draw (2) to (0); 
		\draw [bend left=90, looseness=2.50, very thick, blue, <-] (2) to (0);
		\draw [bend left=90, looseness=1.5, red, very thick] (0) to (1);
		\draw [out=90, in=0, looseness=0.8] (0) to (7);
	\end{pgfonlayer}
\end{tikzpicture}
\begin{tikzpicture}[scale=.5]
	\begin{pgfonlayer}{nodelayer}
		\node [style=vertex] (0) at (-3, -0) {};
		\node [style=vertex] (1) at (2, -0) {};
		\node [style=vertex] (2) at (-6, -0) {};
		\node [style=vertex] (3) at (-1.5, 1.25) {};
		\node [style=vertex] (4) at (1, 1.5) {};
		\node [style=vertex] (5) at (-2, 0.25) {};
		\node [style=vertex] (6) at (0.75, -1.5) {};
		\node [style=vertex] (7) at (-4.5, 2.1) {};
		\node [style=vertex] (8) at (-0.25, 1) {};
		\node [style=vertex] (9) at (-5, 0.5) {};
		\node [style=vertex] (10) at (-4, 0.5) {};
		\node [style=vertex] (11) at (-5, -0.5) {};
		\node [style=vertex] (12) at (-4, -0.5) {};
		\node (x) at (-4.5, 3) {\small\contour{white}{$\eta$}};
		\node (x) at (-5.2, 2) {\small\contour{white}{$\eta'$}};
	\end{pgfonlayer}
	\begin{pgfonlayer}{edgelayer}
		\draw [bend right=90, looseness=1.75, very thick, red, ->] (0) to (1);
		\draw [bend left=90, looseness=1.50] (2) to (0);
		\draw [bend right=90, looseness=1.50, very thick, blue] (2) to (0);
		\draw [bend right=60, looseness=1.50] (0) to (3);
		\draw [bend right=75, looseness=1.75] (3) to (4);
		\draw (4) to (1);
		\draw (0) to (3);
		\draw (5) to (0);
		\draw (4) to (6);
		\draw [bend right, looseness=1.00] (0) to (6);
		\draw (8) to (4);
		\draw [bend left, looseness=1.00] (3) to (4);
		\draw (2) to (9);
		\draw (9) to (10);
		\draw (10) to (0);
		\draw (2) to (11);
		\draw (11) to (12);
		\draw (12) to (0);
		\draw (2) to (0); 
		\draw [bend left=90, looseness=2.50, very thick, blue, <-] (2) to (0);
		\draw [bend left=90, looseness=1.5, red, very thick] (0) to (1);
		\draw [out=90, in=180, looseness=0.8] (2) to (7);
	\end{pgfonlayer}
\end{tikzpicture}\\
\hspace{.5cm}\small{$(\coll{L,c}\cdot q_i^R, e_i^R, s_i^R)_{i=1}^{N_R}$\hspace{.7cm} $\rightsquigarrow$ flip $\eta$ clockwise $\rightsquigarrow$\hspace{.3cm} $\rightsquigarrow$ flip $\eta'$ clockwise $\rightsquigarrow$}\\	
\begin{tikzpicture}[scale=.5]
	\begin{pgfonlayer}{nodelayer}
		\node [style=vertex] (0) at (-3, -0) {};
		\node [style=vertex] (1) at (2, -0) {};
		\node [style=vertex] (2) at (-6, -0) {};
		\node [style=vertex] (3) at (-1.5, 1.25) {};
		\node [style=vertex] (4) at (1, 1.5) {};
		\node [style=vertex] (5) at (-2, 0.25) {};
		\node [style=vertex] (6) at (0.75, -1.5) {};
		\node [style=vertex] (7) at (-4.5, -2.1) {};
		\node [style=vertex] (8) at (-0.25, 1) {};
		\node [style=vertex] (9) at (-5, 0.5) {};
		\node [style=vertex] (10) at (-4, 0.5) {};
		\node [style=vertex] (11) at (-5, -0.5) {};
		\node [style=vertex] (12) at (-4, -0.5) {};
		\node (x) at (-4.5, -3) {\small\contour{white}{$\eta$}};
		\node (x) at (-5.2, -2) {\small\contour{white}{$\eta'$}};
	\end{pgfonlayer}
	\begin{pgfonlayer}{edgelayer}
		\draw [bend right=90, looseness=1.75, very thick, red, ->] (0) to (1);
		\draw [bend left=90, looseness=1.50, very thick, blue, <-] (2) to (0);
		\draw [bend right=90, looseness=1.50] (2) to (0);
		\draw [bend right=60, looseness=1.50] (0) to (3);
		\draw [bend right=75, looseness=1.75] (3) to (4);
		\draw (4) to (1);
		\draw (0) to (3);
		\draw (5) to (0);
		\draw (4) to (6);
		\draw [bend right, looseness=1.00] (0) to (6);
		\draw (8) to (4);
		\draw [bend left, looseness=1.00] (3) to (4);
		\draw (2) to (9);
		\draw (9) to (10);
		\draw (10) to (0);
		\draw (2) to (11);
		\draw (11) to (12);
		\draw (12) to (0);
		\draw (2) to (0); 
		\draw [bend right=90, looseness=2.50, very thick, blue] (2) to (0);
		\draw [bend left=90, looseness=1.5, red, very thick] (0) to (1);
		\draw [out=-90, in=180, looseness=0.8] (2) to (7);
	\end{pgfonlayer}
\end{tikzpicture}
\begin{tikzpicture}[scale=.5]
	\begin{pgfonlayer}{nodelayer}
		\node [style=vertex] (0) at (-3, -0) {};
		\node [style=vertex] (1) at (2, -0) {};
		\node [style=vertex] (2) at (-6, -0) {};
		\node [style=vertex] (3) at (-1.5, 1.25) {};
		\node [style=vertex] (4) at (1, 1.5) {};
		\node [style=vertex] (5) at (-2, 0.25) {};
		\node [style=vertex] (6) at (0.75, -1.5) {};
		\node [style=vertex] (7) at (-4.5, -2.1) {};
		\node [style=vertex] (8) at (-0.25, 1) {};
		\node [style=vertex] (9) at (-5, 0.5) {};
		\node [style=vertex] (10) at (-4, 0.5) {};
		\node [style=vertex] (11) at (-5, -0.5) {};
		\node [style=vertex] (12) at (-4, -0.5) {};
		\node (x) at (-4.5, -3) {\small\contour{white}{$\eta$}};
		\node (x) at (-3.7, -2) {\small\contour{white}{$\eta'$}};
	\end{pgfonlayer}
	\begin{pgfonlayer}{edgelayer}
		\draw [bend right=90, looseness=1.75, very thick, red, ->] (0) to (1);
		\draw [bend left=90, looseness=1.50, very thick, blue, <-] (2) to (0);
		\draw [bend right=90, looseness=1.50] (2) to (0);
		\draw [bend right=60, looseness=1.50] (0) to (3);
		\draw [bend right=75, looseness=1.75] (3) to (4);
		\draw (4) to (1);
		\draw (0) to (3);
		\draw (5) to (0);
		\draw (4) to (6);
		\draw [bend right, looseness=1.00] (0) to (6);
		\draw (8) to (4);
		\draw [bend left, looseness=1.00] (3) to (4);
		\draw (2) to (9);
		\draw (9) to (10);
		\draw (10) to (0);
		\draw (2) to (11);
		\draw (11) to (12);
		\draw (12) to (0);
		\draw (2) to (0); 
		\draw [bend right=90, looseness=2.50, very thick, blue] (2) to (0);
		\draw [bend left=90, looseness=1.5, red, very thick] (0) to (1);
		\draw [out=-90, in=0, looseness=0.8] (0) to (7);
	\end{pgfonlayer}
\end{tikzpicture}	\qquad
\begin{tikzpicture}[scale=.5]
	\begin{pgfonlayer}{nodelayer}
		\node [style=vertex] (0) at (-3, -0) {};
		\node [style=vertex] (1) at (2, -0) {};
		\node [style=vertex] (2) at (-6, -0) {};
		\node [style=vertex] (3) at (-1.5, 1.25) {};
		\node [style=vertex] (4) at (1, 1.5) {};
		\node [style=vertex] (5) at (-2, 0.25) {};
		\node [style=vertex] (6) at (0.75, -1.5) {};
		\node [style=vertex] (8) at (-0.25, 1) {};
		\node [style=vertex] (9) at (-5, 0.5) {};
		\node [style=vertex] (10) at (-4, 0.5) {};
		\node [style=vertex] (11) at (-5, -0.5) {};
		\node [style=vertex] (12) at (-4, -0.5) {};
		\node [style=vertex] (13) at (-5, 1.5) {};
		\node [style=vertex] (14) at (-4, 1.5) {};
	\end{pgfonlayer}
	\begin{pgfonlayer}{edgelayer}
		\begin{scope}
	\clip (2.center) to (13.center) to (9.center) to (2.center);
	\draw[draw=blue, fill=blue!20] (2) circle (25pt);
	\end{scope}
		\draw [bend right=90, looseness=1.75, very thick, red, ->] (0) to (1);
		\draw [bend right=90, looseness=1.50, very thick, blue] (2) to (0);
		\draw [bend right=60, looseness=1.50] (0) to (3);
		\draw [bend right=75, looseness=1.75] (3) to (4);
		\draw (4) to (1);
		\draw (0) to (3);
		\draw (5) to (0);
		\draw (4) to (6);
		\draw [bend right, looseness=1.00] (0) to (6);
		\draw (8) to (4);
		\draw [bend left, looseness=1.00] (3) to (4);
		\draw (2) to (9);
		\draw (9) to (10);
		\draw (10) to (0);
		\draw (2) to (11);
		\draw (11) to (12);
		\draw (12) to (0);
		\draw (2) to (0); 
		\draw [bend left=90, looseness=2.50, very thick, blue, <-] (2) to (0);
		\draw [bend left=90, looseness=1.5, red, very thick] (0) to (1);
		\draw (2)--(13)--(14)--(0);
	\end{pgfonlayer}
\end{tikzpicture}\\
\hspace{-4cm}\small{$\rightsquigarrow$ flip $\eta$ clockwise $\rightsquigarrow$\hspace{.3cm} $\rightsquigarrow$ flip $\eta'$ clockwise $\rightsquigarrow$ \hspace{.7cm}$((q_i^L\cdot \coll{R,c'}, e_i^L, s_i^L)_{i=1}^{N_L})^{\rm{rev}}$}
\caption{\label{fig:complete path}The path $P(\coll{L,c}\cdot R, L\cdot \coll{R,c'})$; notice how the four flips in the ``central phase'' of the path turn the result of the ``right phase'', which is $\coll{L,c}\cdot (\rightarrow \cdot \coll{R,c'})$, into the quadrangulation $(\rightarrow \cdot \coll{L,c}) \cdot \coll{R,c'}$, so that the ``left phase'' can begin and turn the quadrangulation into the desired $L\cdot \coll{R,c'}$. Note that the root edge of the quadrangulation is always the one marked in red appearing in the lower right part of the picture; the arrow marked in blue represents the root edge of the ``left quadrangulation'' and is marked to help confirm the fact above. 
} 
\end{figure}
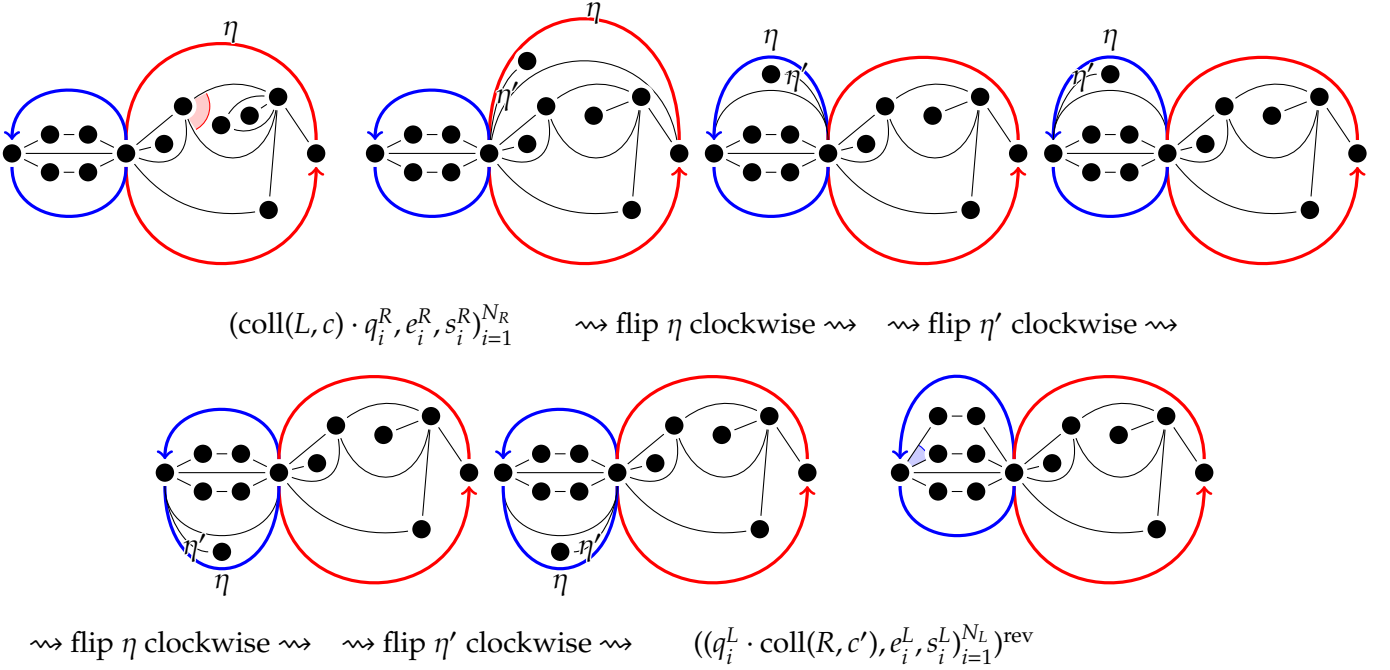

Given $(L_i\cdot R_i)_{i=0}^{n-1}\in\Gamma_q^{q'}$, we now wish to build a flip path turning the quadrangulation $L_i\cdot R_i$ into $L_{i+1}\cdot R_{i+1}$. That is, given $L\in\sQ_a$ and $R\in\sQ_{n-a-1}$ and two corners $c$ and $c'$ of $L$ and $R$ respectively, we wish to build a flip path from $\coll{L,c}\cdot R$ to $L\cdot\coll{R,c'}$.

This we shall do by simply combining multiple constructions from the previous section. Indeed, consider $P(L,c)=(q^L_i,e^L_i,s^L_i)_{i=1}^{N_L}$ and $P(R,c')=(q^R_i,e^R_i,s^R_i)_{i=1}^{N_R}$ as constructed previously. Though the edge $e_1^R$ is an edge of $R$, it can be uniquely identified with an edge of $\coll{L,c}\cdot R$; inductively, though $e_i^R$ is an edge of $q_i^R$, we can see it as an edge of $\coll{L,c}\cdot q_i^R=\coll{L,c}\cdot (q_{i-1}^R)^{e^R_{i-1},s^R_{i-1}}$. We may therefore consider the sequence of flips $(\coll{L,c}\cdot q^R_i,e^R_i,s^R_i)_{i=1}^{N_R}$, which is such that $(\coll{L,c}\cdot q^R_{N_R})^{e^R_{N_R},s^R_{N_R}}$ is equal to $\coll{L,c}\cdot (\rightarrow\cdot \coll{R,c'})$.

Now consider the face $f$ lying directly to the right of the root in $\coll{L,c}\cdot (\rightarrow\cdot \coll{R,c'})$, let $\eta$ be the edge immediately after the root edge in the clockwise contour of $f$ and let $\eta'$ be the internal edge of the degenerate face adjacent to $\eta$ within the ``right'' quadrangulation $\rightarrow\cdot \coll{R,c'}$. By alternatively flipping $\eta$ and $\eta'$, one can have the degenerate face containing $\eta'$ ``slide'' along the boundary of $f$. Consider in particular the sequence of four flips
$$(\coll{L,c}\cdot (\rightarrow\cdot \coll{R,c'}), \eta, +)$$
$$((\coll{L,c}\cdot (\rightarrow\cdot \coll{R,c'}))^{\eta,+}, \eta', +)$$
$$(((\coll{L,c}\cdot (\rightarrow\cdot \coll{R,c'}))^{\eta,+})^{\eta',+}, \eta, +)$$
$$((((\coll{L,c}\cdot (\rightarrow\cdot \coll{R,c'}))^{\eta,+})^{\eta',+})^{\eta,+}, \eta', +)$$
as depicted in Figure~\ref{fig:complete path}. After the first flip, the degenerate face containing $\eta'$ lies immediately to the left of the root edge in the ``left quadrangulation'' obtained as described in Section~\ref{sec:canonical paths 1} and shown in Figure~\ref{fig:eye quadrangulation}; the next three flips make it so that the degenerate face lies immediately to the \emph{right} of the root edge of the ``left quadrangulation'', with $\eta'$ adjacent to the origin. The result of the four flips is therefore $(\rightarrow\cdot\coll{L,c})\cdot \coll{R,c'}$.

We can thus define the whole path from $\coll{L,c}\cdot R$ to $L\cdot\coll{R,c'}$, which we shall denote by $P(\coll{L,c}\cdot R, L\cdot\coll{R,c'})$, by a concatenation of the following  sequences of flips, which we will refer to as the ``right phase'', the ``central phase'' (consisting of $4$ flips), the ``left phase'':
\newcommand{\rev}{\operatorname{rev}}
\begin{itemize}
\item right phase: $$(\coll{L,c}\cdot q^R_i,e^R_i,s^R_i)_{i=1}^{N_R}$$
\item central phase:
$$(\coll{L,c}\cdot (\rightarrow\cdot \coll{R,c'}), \eta, +)$$
$$((\coll{L,c}\cdot (\rightarrow\cdot \coll{R,c'}))^{\eta,+}, \eta', +)$$
$$(((\coll{L,c}\cdot (\rightarrow\cdot \coll{R,c'}))^{\eta,+})^{\eta',+}, \eta, +)$$
$$((((\coll{L,c}\cdot (\rightarrow\cdot \coll{R,c'}))^{\eta,+})^{\eta',+})^{\eta,+}, \eta', +)$$
\item left phase: $$((q_i^L\cdot \coll{R,c'},e^L_i,s^L_i)_{i=1}^{N_L})^{\rev},$$
\end{itemize}
where, given a flip path $P=(q_i,e_i,s_i)_{i=1}^N\in\Gamma_{q_1\rightarrow q_N^{e_N,s_N}}$, we set $P^{\rev}$ to be the flip path
$$(q_{N-i}^{e_{N+1-i},s_{N+1-i}},e_{N+1-i},-s_{N+1-i})_{i=1}^N$$
in $\Gamma_{q_N^{e_N,s_N}\rightarrow q_1}$.
%
%
%
%

We are now ready to fully describe the mapping $\Psi_{q_1,q_2}$: given $q_1,q_2\in\sQ_n$, consider any pair of sequences $((L^1_i\cdot R^1_i)_{i=0}^{n-1},(L^2_i\cdot R^2_i)_{i=0}^{n-1})\in\Gamma_{q_1}^{F(q_1,q_2)}\times\Gamma_{q_2}^{F(q_1,q_2)}$ that has nonzero probability according to $\P_{q_1}^{F(q_1,q_2)}\otimes\P_{q_2}^{F(q_1,q_2)}$. Set $\Psi_{q_1,q_2}((L^1_i\cdot R^1_i)_{i=0}^{n-1},(L^2_i\cdot R^2_i)_{i=0}^{n-1})$ to be the successive concatenation of
\begin{itemize}
\item $P(q_1,c)$, where $R^1_0=\coll{q_1,c}$;
\item $P((L^1_i,R^1_i),(L^1_{i+1},R^1_{i+1}))$ for $i=0,\ldots,n-2$;
\item $P((L^2_{i-1},R^2_{i-1}),(L^2_i,R^2_i))^{\operatorname{rev}}$ for $i=n-1,n-2,\ldots,1$;
\item $P(q_2,c')^{\operatorname{rev}}$, where $R^2_0=\coll{q_2,c'}$.
\end{itemize}
We also have all the setup necessary to show the following important estimate:

\begin{proposition}\label{prop:final estimate}
Consider a quadrangulation $q\in\sQ_n$, an edge $e$ of $q$ other than the root edge	and an element $s\in\{+,-\}$. We have
$$\sum_{q_1,q_2\in\sQ_n}\P_{q_1\rightarrow q_2}(\{\gamma\in\Gamma_{q_1\rightarrow q_2}\mbox{ containing } (q,e,s)\})\leq 8\cdot 12^{n+1}.$$
\end{proposition}
\begin{proof}
By our definition of $\P_{q_1\rightarrow q_2}$, the expression we wish to estimate is 
$$\sum_{q_1,q_2\in\sQ_n}\P_{q_1}^{F(q_1,q_2)}\otimes\P_{q_2}^{F(q_1,q_2)}(\{X\in\Gamma_{q_1}^{F(q_1,q_2)}\times\Gamma_{q_2}^{F(q_1,q_2)}\st \Psi_{q_1,q_2}(X)\mbox{ contains } (q,e,s)\}).$$

We will use as an upper bound the one we obtain by summing the terms corresponding to the following three possibilities:
\begin{itemize}
\item the flip $(q,e,s)$ appears in $P(q_1,c)$ or $P(q_2,c')^{\rev}$, in which case $R^1_0=\coll{q_1,c}\in\{\coll{q,x_1},\coll{q,x_2}\}$ or $R^2_0=\coll{q_2,c'}\in\{\coll{q^e,y_1},\coll{q^e,y_2}\}$, where $x_1,x_2$ are the corners of $q$ that correspond to the two possible orientations of $e$ and $y_1,y_2$ are the corners of $q^e$ that correspond to the two possible orientations of the flipped version of $e$, by Corollary~\ref{cor:few corners}. Now, by Lemma~\ref{key estimate}, we have
$$\sum_{i=1,2}\sum_{q_1,q_2}\P_{q_1}^{F(q_1,q_2)}(R^1_0=\coll{q,x_i})+\P_{q_2}^{F(q_1,q_2)}(R^2_0=\coll{q^e,y_i})=$$
$$=\sum_{i=1,2}[\sum_{q_1,q'}|\{q_2\st F(q_1,q_2)=q'\}|\cdot \P_{q_1}^{q'}(R^1_0=\coll{q,x_i}, L^1_0=\rightarrow)+$$
$$\sum_{q_2,q'}|\{q_1\st F(q_1,q_2)=q'\}|\cdot \P_{q_2}^{q'}(R^2_0=\coll{q^e,y_i}, L^2_0=\rightarrow)]\leq$$
$$\leq 4\cdot 12 \cdot 12^{2n-(n-1)-1}=4\cdot 12^n. $$
\item The flip $(q,e,s)$ appears in  $P(L^1_i\cdot R^1_i, L^1_{i+1}\cdot R^1_{i+1})$ for some $i$; we shall consider some separate subcases:
\begin{itemize} 
\item we have $q=q_L\cdot q_R\in\sQ_l\cdot\sQ_r$ and $e$ is the image of an edge other than the root edge in $E(q_L)$ (so that $q^{e,s}$ also lies in $\sQ_l\cdot\sQ_r$, and in fact in $\sQ_l\cdot q_R$). Let $c_1,c_2$ be the corners corresponding to the two possible orientations of $e$ in $q^e$. If $(q,e,s)$ is a flip in the ``central phase'' of the path, with $e=\eta$ or $e=\eta'$ (see Figure~\ref{fig:complete path}), then at least one of the corner $c_1,c_2$ lies in the degenerate face that is in the process of being moved along the boundary of the ``left'' quadrangulation; as a consequence, we have $L^1_i\in\{\coll{q^e_L,c_1},\coll{q^e_L,c_2}\}$, hence $i=l-1$ and $R^1_{i+1}=q_R$. If not, then the flip happens in the ``left phase'' of the path and Corollary~\ref{cor:few corners} implies that $L^1_i\in\{\coll{q^e_L,c_1},\coll{q^e_L,c_2}\}$, hence $i=l-1$ and $R^1_{i+1}=q_R$. Thus we have the term
$$\sum_{i\in\{1,2\}}\sum_{q_1,q_2}\P_{q_1}^{F(q_1,q_2)}(L^1_{l-1}=\coll{q^e_L,c_i},R^1_l=q_R)\leq 2\cdot 12\cdot 12^{2n-(l-1)-(n-l-1)-1}=2\cdot 12^{n+1}.$$
\item we have $q=q_L\cdot q_R\in\sQ_l\cdot\sQ_r$ and $e$ is the image of an edge other than the root edge in $E(q_R)$. This case is analogous: this time Corollary~\ref{cor:few corners} gives $L^1_l=q_L$ and $R^1_{l+1}\in\{\coll{q_R,c_1},\coll{q_R,c_2}\}$, where $c_1,c_2$ are the corners of $q$ corresponding to the two possible orientations of $e$. This yields another term of the form
$$\sum_{i\in\{1,2\}}\sum_{q_1,q_2}\P_{q_1}^{F(q_1,q_2)}(L^1_{l}=q_L,R^1_{l+1}=\coll{q_R,c_i})\leq 2\cdot 12\cdot 12^{2n-l-(n-l-2)-1}=2\cdot 12^{n+1}.$$
\item we have $q\in\sQ_l\cdot\sQ_r$ and $q^{e,s}\in\sQ_{l+1}\cdot\sQ_{r-1}$. This is the only case we are missing, i.e.~the one where $e$ is the edge right after the root edge of $q$ in the clockwise contour of the face lying directly to the right of the root edge (it can be seen that, by construction, all other flips in $P(L^1_i\cdot R^1_i, L^1_{i+1}\cdot R^1_{i+1})$ happen within $q_L$ or within $q_R$. In this case, if $q^{e,s}=q'_L\cdot q'_R$, we have $L^1_i=q_L$ and $R^1_{i+1}=q'_R$, hence $i=l$; we get the term
$$\sum_{q_1,q_2}\P_{q_1}^{F(q_1,q_2)}(L^1_{l}=q_L,R^1_{l+1}=q'_R)\leq 12\cdot 12^{2n-l-(n-l-2)-1}=12^{n+1}.$$
 \end{itemize}
 Globally, this yields a term that can be upper bounded by $2\cdot 12^{n+1}$.
\item The flip $(q,e,s)$ appears in  $P(L^2_i\cdot R^2_i, L^2_{i+1}\cdot R^2_{i+1})^{\rev}$ for some $i$; clearly, this case is entirely analogous to the previous one, and will yield another term upper bounded by $2\cdot 12^{n+1}$.
\end{itemize}

Summing the three upper bounds above proves the lemma.
\end{proof}

\section{The final bound}\label{sec:final bound}

All this being done, we can apply the technique of canonical paths of Diaconis and Saloff-Coste \cite{diaconis96} to bound the relaxation time of $\F_n$.

\begin{proof}[Proof of Theorem~\ref{main theorem}]

The fact that $\nu_n=\mu_n$ is an obvious consequence of Proposition~\ref{prop:from rotations to flips}. The upper bound for $\nu_n$ can be proven in exactly the same way as the one in \cite{CS2019}: because the only difference between the chains $\F_n$ and $\FF_n$ is the fact that the root edge can no longer be flipped and that each flip is assigned a probability of $\frac{1}{3(2n-1)}$ rather than $\frac{1}{6n}$, the proof of Proposition 4.1 in \cite{CS2019} also applies to the spectral gap $\nu_n$ of $\FF_n$. 

As for the lower bound, we have
$$\frac{1}{\nu_n}\leq\max_{(q,e,s)}\frac{1}{\pi(q)p(q,q^{e,s})}\sum_{q_1,q_2\in\sQ_n}\sum_{\substack{\gamma\in\Gamma_{q_1\to q_2}:\\(q,e,s)\in\gamma}}|\gamma|\P_{q_1\to q_2}(\gamma)\pi(q_1)\pi(q_2),$$
where $\pi$ is the uniform measure on $\sQ_n$, $(q,e,s)$ varies among all possible flips ($q\in\sQ_n$, $e\in E(q)$, $s=\pm$) and $p(q,q^{e,s})$ is the transition probability according to $\FF_n$.

Now, all instances of $\pi(\cdot)$ can be replaced by $\frac{1}{|\sQ_n|}$. Also, we have $p(q,q^{e,s})\geq \frac{1}{3(2n-1)}$ (hence $\frac{1}{p(q,q^{e,s})}\leq 6n$) for all $q\in\sQ_n$, $e\in E(q)$, $s=\pm$. Moreover, the length of our canonical paths as constructed is at most $32n^2$. This can be checked by going through the final construction from Section~\ref{sec:completing psi}: each path of non-zero weight in $\Gamma_{q_1\to q_2}$ is built as two sequences (one ``straight'' and one ``reversed'') of
\begin{itemize}
\item one path of the form $P(q,c)$;
\item $n-1$ paths of the form $P((L,R),(L',R'))$.
\end{itemize}
In turn, every path of the form $P((L,R),(L',R'))$ is built as a concatenation of
\begin{itemize}
\item one path of the form $P(q,c)$;
\item $4$ single flips;
\item one path of the form $P(q,c)$, reversed.
\end{itemize}
By the three lemmas in Section~\ref{sec:from q to ->coll(q,c)}, we know that the length of a path of the form $P(q,c)$ is at most $8n$, which yields the global upper bound of $32n^2$.

Applying the bound given by Proposition~\ref{prop:final estimate} we then obtain
$$\frac{1}{\nu_n}\leq \frac{6n\cdot 32n^2\cdot 8\cdot 12^{n+1}}{|\sQ_n|}.$$

Since $\frac{|\sQ_n|}{12^n}\geq \frac{C}{n^{5/2}}$, we have
$$\frac{1}{\nu_n}\leq 6n\cdot 32n^2\cdot 8\cdot 12 \cdot Cn^{5/2}\leq C_2n^{11/2}$$
for some appropriate constant $C_2$, as desired.
\end{proof}

\end{document}